\documentclass[12pt]{iopart}
\usepackage{ulem}
\usepackage{xcolor}
\usepackage{iopams}  
\usepackage{setstack}
\usepackage{graphicx}
\newcommand{\ind}{\,\mbox{d}}
\usepackage{bbm}
\usepackage{fullpage}
\newtheorem{lemma}{Lemma}
\newtheorem{corollary}{Corollary}
\newtheorem{prop}{Proposition}
\newtheorem{remark}{Remark}
\newtheorem{theorem}{Theorem}

\newcommand{\proofend}{\hfill $\Box$ \vspace{2mm}}

\begin{document}

\title[Shape-parameter identification of linear sampling method]{Shape and parameter identification by the linear sampling method for a restricted Fourier integral operator}

\author{Lorenzo Audibert$^{1,2}$~and~Shixu Meng$^3$}

\address{$^1$ PRISME, EDF R\&D, 6 Quai Watier, 78400 Chatou, France.}
\address{$^2$ UMA, Inria, ENSTA Paris, Institut Polytechnique de Paris, 91120 Palaiseau, France. }
\address{$^3$ Department of Mathematics, Virginia Tech, 24061
Blacksburg,  USA.}
\eads{\mailto{lorenzo.audibert@edf.fr}, \mailto{sgl22@vt.edu}}


\begin{abstract}
In this paper we provide a new linear sampling method based on the same data but a different definition of the data operator for two inverse problems: the multi-frequency inverse source problem for a fixed observation direction and the Born inverse scattering problems. We show that the associated regularized linear sampling indicator converges  to the average of the unknown in  a small neighborhood as the regularization parameter approaches to zero.
We develop both a shape identification theory and a parameter identification theory  which are stimulated, analyzed, and implemented with the help of the  prolate spheroidal wave functions and their generalizations.  We further propose a prolate-based implementation of the linear sampling method and provide numerical experiments to demonstrate how this linear sampling method is capable of reconstructing both the shape  and the parameter.
\end{abstract}

%
\vspace{2pc}

%
%
%

\section{Introduction}  \label{section intro}
Inverse scattering and inverse source problems play important roles in geophysical exploration, non-destructive testing, medical diagnosis, and numerous problems associated with shape and parameter identification. The \textit{linear sampling method}, first proposed in \cite{ColtonKirsch96}, is a non-iterative imaging method for shape identification. It requires little a priori information (such as boundary conditions  and the number of connected components) about the object, provides a direct computational implementation, and  is robust to noise. 
The \textit{factorization method} proposed later in  \cite{Kirsch98} gives a complete theory for shape identification and also provides a direct implementation. With the help of the factorization method, a complete theoretical justification of the linear sampling method was given by \cite{arens04,arenslechleiter09} and was further developed in   \cite{arenslechleiter15} (since they called their formulation as an alternative formulation of the linear sampling method so we simply refer to their method as the \textit{alternative linear sampling method} for convenience in this paper).  The \textit{generalized linear sampling method} proposed in \cite{audiberthaddar15} modifies the regularizer and also provides a complete theoretical justification of the linear sampling method with the minimal restriction.  The linear sampling and the factorization methods play   important roles in inverse problems associated with shape identification such as inverse scattering problem and electrical impedance tomography. We refer to the monographs \cite{cakoni2016qualitative,cakoni2016inverse,colton2012inverse,kirsch2008factorization} for a more comprehensive discussion.

In this paper we provide a linear sampling type method for shape identification based on a different definition of the data operator  and show that the indicator represents an average of the unknown which leads to  parameter identification/estimation. The demonstration in this paper is in the context of the multi-frequency inverse source problem for a fixed observation direction and the Born inverse scattering problems. Part of the motivations are due to the numerical examples performed by \cite{meng22} on a multi-frequency inverse source problem in waveguides and the data-driven basis based on  prolate spheroidal wave functions (\textit{PSWFs}) and their generalizations  in \cite{meng23data}.
In this paper we use the PSWFs to analyze and implement the linear sampling method. 
In a broader context, it is noted that there are many existing results on parameter identification for the multi-frequency inverse source problem  and the Born inverse scattering problem, such as \cite{Kirsch17,meng23data,moskowSchotland08,natterer2001book}  and diffraction tomography \cite{Kirisitsetal2021,quellmalzetal2023} and the numerous references therein.  In the recent paper \cite{novikov22}, it was shown that the convergence for parameter identification for the inverse source problem on the ball is of  H\"{o}lder-logarithmic type where their analysis was based on the PSWFs in one dimension and the Radon transform.

\emph{Shape identification}. Our formulation of the linear sampling method is based on solving the usual data equation for a slightly different data operator  by a general regularization scheme  which gives an indicator function that allows to characterize the shape.
This indicator function is similar to an alternative linear sampling method proposed in \cite{arenslechleiter15} and is closely related to the generalized linear sampling method in \cite{audiberthaddar15}, and as is noted in \cite{arenslechleiter15}, this formulation can be dated back to the first paper of linear sampling method \cite{ColtonKirsch96} where they suggested an alternative indicator function. Our investigation is in the context of the multi-frequency inverse source problem for a fixed observation direction and the Born inverse scattering problems. It is worth noting that the multi-frequency factorization method for the inverse source problem was initially studied  in \cite{GriesmaierSchmiedecke-source} that motives our present study. In this paper, we first obtain the shape identification result, based on the assumption that, roughly speaking, the factorization method theory applies. Several  remarks: (1) We   obtain  shape identification theories based on the alternative linear sampling method and the generalized linear sampling method. Essentially the factorization method, the alternative linear sampling method, and the generalized linear sampling method are all capable of shape and parameter identification in our case. (2) The regularized solution $g_{z,\alpha}$ can be obtained using any general regularization scheme (which can be the standard Tikhonov or the singular value cut off regularization) so that it is a general shape identification theory, as is similar to \cite{arenslechleiter15}.

\emph{Parameter identification}.   The parameter identification is based on  the same  indicator function and we  show that one can reconstruct the average of $1/q$ over a small user-defined region (where $q$ is the unknown parameter). This result is quite general and relates in the classical setting of the LSM to the convergence of solution to the usual data equation to a specific function depending on $q$. We provide a proof in this setting. Additionally we utilize the PSWFs to prove again the parameter identification theory. This result is stimulated by our efforts to use PSWFs as a tool for both the analysis and the implementation and demonstrate the relevance of PSWFs as it is striking, at least to us, that one could use a basis independant of $q$ to obtains such a result. Again the parameter identification theory is a general theory, since the regularized solution $g_{z,\alpha}$ can be obtained using any general regularization scheme.

Prolate spheroidal wave functions (PSWFs) and their generalizations. The PSWFs and their generalizations in the context of the restricted Fourier integral operator were studied in \cite{Slepian61,Slepian64,Slepian78} and play important roles in Fourier analysis, uncertainty quantification, and signal processing. Their remarkable property is due to that the PSWFs are eigenfunctions of a restricted Fourier integral operator (which is one of the factorized operator associated with the data operator)  and of a Sturm-Liouville differential operator at the same time. The generalizations of PSWFs in two dimension are referred to as the disk PSWFs. 
For a more complete picture on the theory and computation of the (disk) PSWFs   we refer to \cite{boyd05code,wang10,ZLWZ20} and the numerous references therein. For extension to domain that are not disk, however with less theoretical results, we refer to \cite{Simons1} and references therein. Recently, (disk) PSWFs were applied to the inverse source problem in \cite{novikov22} and to the Born inverse scattering problems in \cite{meng23data}.

Implementation of the linear sampling method. In addition to  our general theory on  shape and parameter identification of the linear sampling and factorization methods, we propose a prolate-based formulation of the linear sampling method.   The key observation is that one of the factorized operator has (disk) PSWFs as its eigenfunctions. In this way we obtain a reduced indicator function in a high dimensional subspace. For sake of rigor and completeness, we give the full details on the computation of the PSWFs and their corresponding prolate eigenvalues that are needed in our prolate-based linear sampling method.

The remaining of the paper is as follows. We first introduce in Section \ref{Section Model} an inverse source problem for a fixed observation direction and the Born inverse scattering problem, and summarize these two inverse problems into an inverse problem associated with a restricted Fourier integral operator in Section \ref{Model all cases}. For later purposes, we also introduce the (disk) PSWFs  that stimulate our analysis and computation. Section \ref{Section operator} introduces the data operator, analyzes its factorizations and shows a range identity. In Section \ref{Section FM and a LSM}, we   obtain the shape identification theories based on the alternative linear sampling method and the generalized linear sampling method with the help of the factorization method theory.  Section \ref{Section parameter LSM} is devoted to the general parameter identification theory. We prove  that our indicator function has capability in reconstructing the   average of $1/q$ over a small user-defined region (where $q$ is the unknown parameter).  In Section \ref{Section numerical preliminaries} we study an explicit example to discuss the nature of the inverse problem followed by several preliminaries on the computation of PSWFs and the Legendre-Gauss-Lobatto quadrature. We also propose the prolate-based formulation of the linear sampling method. Finally, Section \ref{Section numerical examples} is devoted to numerical experiments that illustrate the shape and parameter identification theory.
 
\section{The Mathematical Model for Inverse Source and Born Inverse Scattering Problems} \label{Section Model}

In this section, we first introduce an inverse source problem for a fixed observation direction and the Born inverse scattering problem. We then summarize these two inverse problems in Section \ref{Model all cases}.
\subsection{Multi-frequency inverse source problem  for a fixed observation direction} \label{Model inverse source fixed observation}
In this section we introduce the  inverse  source problem with  multi-frequency data measured at sparse observation directions.
Note that  the notations in this section are only for the purpose of introducing the  inverse source model which leads to a model given later by \eref{Section Model general}, thus remain relevant only in this section.   
When considering the acoustic wave propagation due to a source $f$ in an homogeneous isotropic medium in $\mathbb{R}^d$ ($d=2$), 
one has the nonhomogeneous Helmholtz equation
\begin{equation*}\label{Helmholtzequation}
\Delta u^s +k^2 u^s = -f \quad \mbox{in}\,\,\mathbb{R}^2,
\end{equation*}
where the  wave number is denoted by $k$,  the support of the unknown source is denote by $\Omega$ which is a bounded Lipschitz domain in $\mathbb{R}^2$ with connected complement $\mathbb{R}^2\backslash \overline{\Omega}$. We suppose that the support of $f$ is a subset of $B(0,L)$. 
The scattered field $u^s$ is required to satisfy the Sommerfeld radiation condition
\begin{equation*} 
\lim_{r=|x|\rightarrow \infty} \sqrt{r} \Big(\frac{\partial u^s}{\partial r} - iku^s \Big) = 0
\end{equation*}
uniformly for all directions $x/|x|$. It is known that
\begin{equation*} 
u^{s}(x,k)
=\frac{e^{i\frac{\pi}{4}}}{\sqrt{8k\pi}} \frac{e^{ikr}}{\sqrt{r}}\left\{u^{\infty}(\hat{x},k)+\mathcal{O}\left(\frac{1}{r}\right)\right\}\quad\mbox{as }\,r=|x|\rightarrow\infty,
\end{equation*}
and that 
(see, for instance, \cite{cakoni2016qualitative})
\begin{equation}\label{Section Model inverse source data1}
u^{\infty}(\hat{x},k)=\int_{\Omega}e^{-ik\hat{x}\cdot\,y}f(y) \ind y,
\end{equation}
where $\hat{x}$ is an observation direction belonging to $\mathbb{S}$ which denotes the unit circle, and the wavenumber $k$ belongs to the interval $(0,K)$ with $K>0$. Decomposing $s:=y\cdot \hat{x}$ and 
$y^\perp_{\hat{x}}:=y - s\hat{x}$, and identify $\hat{x}^\perp$ as the line orthogonal to $\hat{x}$, we proceed to 
\begin{equation*} 
u^{\infty}(\hat{x},k)=\int_{\mathbb{R}} e^{-ik s } \left(\int_{\hat{x}^\perp}f(s \hat{x} + y^\perp_{\hat{x}}) \ind y^\perp_{\hat{x}} \right) \ind s = \int_{\mathbb{R}} e^{-ik s } \mathcal{R} f (\hat{x},s) \ind s,
\end{equation*}
where $\mathcal{R} f (\hat{x},s)=\int_{\hat{x}^\perp}f(s \hat{x} + y^\perp_{\hat{x}}) \ind y^\perp_{\hat{x}}$ is the Radon transform (see, for instance,  \cite{natterer2001book}).

Note
  that
$$
u^{\infty}(\hat{x}, -k) = u^{\infty}(-\hat{x}, k), \quad k\in (0,K),
$$
then the knowledge of $\{u^{\infty}(\hat{x}, k): k\in (0,K)\} \cup \{u^{\infty}(-\hat{x}, k): k\in (0,K)\}$ amounts to the knowledge of $\{u^{\infty}(\hat{x}, k): k\in (-K,K)\}$ where
\begin{equation} \label{Section Model inverse source data2}
u^{\infty}(\hat{x},k)=\int_{\mathbb{R}}e^{-ik s } \mathcal{R} f (\hat{x},s) \ind s,\quad k \in (-K,K).
\end{equation}

By appropriate scaling, one will be led to the problem \eref{Section Model general} in one dimension. In particular set $q(s):=  \mathcal{R} f (\hat{x},s)$, equation \eref{Section Model inverse source data2} yields 
\begin{eqnarray*}  
u^{\infty}(\hat{x},-Kt)&=&\int_{\mathbb{R}}e^{-i( - Kt) s }  \mathcal{R} f (\hat{x},s) \ind s=\int_{\mathbb{R}} e^{ i K ts} q(s) \ind s,\quad t\in (-1,1).
\end{eqnarray*}
Note that we have assumed that the support of $f$ is a subset of $B(0,L)$ so that  the support of $q$ is a subset of $(-L,L)$, then one is led to
\begin{eqnarray} \label{Section Model inverse source data2-1}
u(t) = \int_{-1}^1 e^{ i 2c ts} q(Ls) \ind s,\quad t\in (-1,1),
\end{eqnarray}
where $c=KL/2$ and  and $u(t):=u^{\infty}(\hat{x},-2ct)$ for $t\in (-1,1)$. The inverse problem is to determine certain information (which will be made precise later) about $q$ from the knowledge of $\{u(t): t \in (-1,1)\}$. In this way one formulate the inverse problems in the form of \eref{Section Model inverse source data2-1} which is the one dimensional case of \eref{Section Model general}.

\begin{remark}
The above inverse problem associated with \eref{Section Model inverse source data2} and \eref{Section Model inverse source data2-1}  is the multi-frequency inverse problem for a fixed observation direction.
When considering the multi-frequency inverse problem of determining $\mathcal{R}f$ (or its support) from $\{u^{\infty}(\hat{x},k): k\in (-K,K),\, \hat{x} \in \mathbb{S} \}$ with  all possible observation directions $\hat{x}$ and all the $k$ in $(0,K)$ where $u^\infty$ is given by \eref{Section Model inverse source data1}, one is led to the  problem  \eref{Section Model general} in two dimensions (after appropriate scalling) with some corresponding parameter $c$.
\end{remark}
\begin{remark}
If one is concerned with recovering $f$ or its support, results in that direction could be found in \cite{GriesmaierSchmiedecke-source}. 
\end{remark}
\subsection{Born inverse scattering problem} \label{Model Born inverse scattering}
In this section we introduce the Born inverse scattering problem.
Note again that the notations in this section are only for the purpose of introducing the Born inverse scattering model which leads to \eref{Section Model general}, thus remain relevant only in this section.   Let $k>0$ be the  wave number. A plane wave  takes the following form 
\begin{equation*} 
e^{i k x \cdot \hat{\theta}}, \quad \hat{\theta} \in \mathbb{S} :=\{ x \in \mathbb{R}^2: |x|=1\},
\end{equation*}
where $\hat{\theta}$ is the direction of propagation.
Let $\Omega \subset \mathbb{R}^2$ be an open and bounded set with
Lipschitz boundary $\partial \Omega$ such that $\mathbb{R}^2 \backslash \overline{\Omega}$  is connected. The set $\Omega$ is referred to as the medium. Let the real-valued function $q(x) \in L^\infty(\Omega)$ be the contrast of the medium  and $q \ge 0$ on $\Omega$. The medium scattering due to a plane wave $e^{i k x \cdot \hat{\theta} }$ is to find total wave filed $e^{i k x \cdot \hat{\theta} } + u^s(x;\hat{\theta};k)$ belonging to $H^1_{loc}(\mathbb{R}^2)$  such that
\begin{eqnarray*}
\Delta_x \big( u^s(x;\hat{\theta};k) + e^{ikx\cdot \hat{\theta}} \big) + k^2 \left(1+q(x)\right) \big( u^s(x;\hat{\theta};k) + e^{ikx\cdot \hat{\theta}} \big) =  0 \quad &\mbox{in}& \quad \mathbb{R}^2, \quad  \\
\lim_{r:=|x|\to \infty} \sqrt{r}  \big( \frac{\partial u^s(x;\hat{\theta};k)}{\partial r} -ik u^s(x;\hat{\theta};k)\big) =0,  
\end{eqnarray*}
where the last equation, i.e., the Sommerfeld radiation condition, holds  uniformly for all directions (and a solution is called   radiating   if it satisfies this radiation condition). The scattered wave field is $u^s(\cdot;\hat{\theta};k)$.    This scattering problem is well-posed and there exists a unique radiating solution; see, for instance,  \cite{colton2012inverse,kirsch2008factorization}. This model is referred to as the full model.
 
 Born approximation is a widely used  method to treat inverse problems; see, for instance,  \cite{colton2012inverse,moskowSchotland08}.
 In the Born approximation region, one can approximate the solution $u^s(\cdot;\hat{\theta};k)$ by its Born approximation $u^s_b(\cdot;\hat{\theta};k)$, which is the unique radiating solution to
\begin{eqnarray*}
\Delta_x u^s_b(x;\hat{\theta};k) + k^2  u^s_b(x;\hat{\theta};k) = -k^2 q(x) e^{ikx \cdot \hat{\theta}} \quad &\mbox{in}& \quad \mathbb{R}^2. \quad  \label{medium us born eqn1}
\end{eqnarray*}
Note that (cf. \cite{cakoni2016qualitative})
\begin{equation*} 
u_b^{s}(x;\hat{\theta};k)
=\frac{e^{i\frac{\pi}{4}}}{\sqrt{8k\pi}} \frac{e^{ikr}}{\sqrt{r}}\left\{u_b^{\infty}(\hat{x};\hat{\theta};k)+\mathcal{O}\left(\frac{1}{r}\right)\right\}\quad\mbox{as }\,r=|x|\rightarrow\infty,
\end{equation*}
uniformly with respect to all directions $\hat{x}:=x/|x|\in\mathbb{S}$, we arrive at $u_b^{\infty}(\hat{x};\hat{\theta};k)$ which is known as the scattering amplitude or  far field pattern with $\hat{x}\in\mathbb{S}$ denoting the observation direction. It directly follows from \cite{cakoni2016qualitative} that
\begin{eqnarray} \label{born uinfty integral representation}
u_b^{\infty}(\hat{x};\hat{\theta};k) = k^2 \int_\Omega  e^{-ik \hat{x} \cdot p'} q(p') e^{ik p' \cdot \hat{\theta}} \ind p',
\end{eqnarray}
therefore the knowledge of $\{u_b^{\infty}(\hat{x};\hat{\theta};k) : \hat{x},\hat{\theta} \in \mathbb{S} \}$ amounts to   the knowledge of $\{u_{b,filter}^{\infty}(p;k): p \in B(0,2)\}$ 
where $u_{b,filter}^{\infty}(p;k)$ is  a truncated Fourier transform of $q$ given by 
\begin{eqnarray} \label{FM near field case 2layer filtered data def 1 born volume eqv}
u_{b,filter}^{\infty}(p;k)  := \int_\Omega e^{i k p\cdot p'} q(p')   \ind p',
\end{eqnarray}
and $B(0,2)$ is a disk centered at origin with radius $2$.
This equivalent formulation is due to \eref{born uinfty integral representation} and that $B(0,2)$ is the interior of  $\{ \hat{\theta}-\hat{x}:\hat{x}, \hat{\theta} \in \mathbb{S}\}$. 

Similar to Section \ref{Model inverse source fixed observation} by introducing $L$ such that $\Omega\subset B(0,L)$, one can reformulate the above problem \eref{FM near field case 2layer filtered data def 1 born volume eqv} to  \eref{Section Model general} with some corresponding parameter $c = Lk$ after some scaling. We omit the details. 
\subsection{A model that summarizes the inverse source and Born inverse scattering problem}  \label{Model all cases}

The inverse source problem for a fixed observation direction in Section \ref{Model inverse source fixed observation} and the Born inverse scattering problem in Section \ref{Model Born inverse scattering} can be summarized as follows: for an unknown function $q$, we consider determination of $q$ and its support   given the available data $u$ (and their perturbations which are called the noisy data) where
\begin{equation} \label{Section Model general}
u (x) = \int_{B} e^{i 2c x\cdot y} q(y) \ind y, \quad x \in B \subset \mathbb{R}^d,
\end{equation}
with $d=1,2$ and $B:=B(0,1)$ denoting the unit interval/disk in $\mathbb{R}^d$. This corresponds to the knowledge of a restricted Fourier transform. Here the unknown function $q \in L^\infty(\mathbb{R}^d)$ has compact support $\Omega \subset B$, and $\Omega$ denotes an open and bounded set with Lipschitz boundary $\partial \Omega$ such that $\mathbb{R}^d \backslash \overline{\Omega}$  is connected. The parameter $c$ is a positive constant that is given by the model (cf. Section \ref{Model inverse source fixed observation}).

In this paper we consider  two classical inverse problems using the linear sampling method for a new data operator based on $u$ instead of $u^\infty_b$: determination of the support of $q$ and  determination of the function $q$. The inverse problem of determining   the support of $q$ is referred to as shape identification and the one of determining the function $q$ is referred to as parameter identification.
\subsection{PSWFs and their generalization}
For later purposes we introduce the PSWFs and their generalizations which stimulate our analysis and computation. In one dimension $d=1$, the  PSWFs  \cite{Slepian61} are $\psi_n(x;c)$ that are   eigenfunctions of $\mathcal{F}^c_d$ where
\begin{equation} \label{Section PSWF multi-d psi_n def}
 \mathcal{F}^c_d \psi_n(\cdot;c)  =  \lambda_n(c) \psi_n(\cdot;c), \quad n \in \mathbbm{N}_d:=\{0,1,\cdots\}, \quad d=1,
\end{equation}
and (we choose to normalize the eigenfunctions so that)
\begin{equation*} 
\int_B \psi_m(x;c) \psi_n(x;c) \ind x = \delta_{mn}, \quad  m,n \in \mathbbm{N}_d,
\end{equation*}
here $\mathcal{F}^c_d$ denotes the operator $L^2(B) \to L^2(B)$ given by
\begin{equation} \label{Section PSWFs Fcd def}
\mathcal{F}^c_d(g) (x) := \int_{B} e^{i  c x\cdot y} g(y) \ind y , \quad x \in B \subset \mathbb{R}^d, \quad \forall g \in L^2(B).
\end{equation}
In two dimensions, the corresponding normalized eigenfunctions are related to the {\it generalized PSWFs} (specifically the radial part of $\psi_n$ is called the generalized PSWFs according to \cite{Slepian64}; in this paper we simply refer  to $\psi_n$ in two dimensions as the disk PSWFs for convenience). As such, $\psi_n(\cdot;c)$ is referred to as the (disk) PSWFs in dimension $d=1,2$. Note that in two dimensions the indexes in \eref{Section PSWF multi-d psi_n def}-\eref{Section PSWFs Fcd def} are multiple indexes given by \eref{Section PSWF multi-d Nd def} where 
\begin{equation} \label{Section PSWF multi-d Nd def}
\hspace{-5em}\mathbbm{N}_d:=\big\{(m,n,\ell): m,n \in \{0,1,2,\cdots\}, \ell \in \mathbbm{I}(m)\big\}, \, \mathbbm{I}(m) :=
\left\{
\begin{array}{cc}
\{1\}  & m=0     \\
\{1,2\}  &     m\ge 1 
\end{array}
\right.  , \, d=2.
\end{equation}
Note that the eigenfunctions $\{\psi_n\}_{n=0}^\infty$ are real-valued, analytic, orthonormal,  and complete in $L^2(B)$ in both one dimension and two dimensions. All the prolate eigenvalues $\lambda_n(c)$ are non-zero. For more details we refer to \cite{boyd05code,Slepian61,wang10} for the one dimensional case and \cite{meng23data,Slepian64,ZLWZ20} for the two dimensional case.



\section{Data operator, factorization, and range identity}\label{Section operator}
In this section we introduce a data operator defined by the given data \eref{Section Model general}, and study its factorization and a range  identity. Following \cite{GriesmaierSchmiedecke-source} (see also \cite{meng22}),  we introduce the data operator $\mathcal{N}: L^2(B) \to L^2(B)$ by
\begin{equation} \label{Section operator N def}
\left( \mathcal{N} g \right)(t) := \int_{B} u\Big(\frac{t-s}{2} \Big) g(s) \ind s,
\end{equation}
where the kernel is given by the data  \eref{Section Model general}. Note that the data $u$ are functions in $L^2(B)$. 

The above data operator enjoys a factorization as follows. Introduce $\mathcal{S}_\Omega: L^2(B) \to L^2(\Omega)$ by 
\begin{equation} \label{Section operator S_omega def}
\left( \mathcal{S}_\Omega g \right)(y) := \int_{B} e^{-i cs \cdot y} g(s) \ind s, \quad \forall g \in L^2(B), \quad y \in \Omega,
\end{equation}
and it follows directly that its adjoint $\mathcal{S}^*_\Omega: L^2(\Omega) \to L^2(B)$ is given by
\begin{equation} \label{Section operator S*_omega def}
\left( \mathcal{S}^*_\Omega h \right)(t) := \int_{\Omega} e^{i c t \cdot y} h(y) \ind y, \quad \forall h \in L^2(\Omega), \quad t \in B
\end{equation}
which is dictated by $\langle \mathcal{S}^*_\Omega h,g \rangle_{L^2(B)} =  \langle  h, \mathcal{S}_\Omega g \rangle_{L^2(\Omega)}$. Here $\langle , \rangle_{L^2(D)}$ represents the $L^2(D)$ inner product with conjugation in the second function, and we further denote $\|\cdot\|_{L^2(D)}$ the corresponding $L^2(D)$ norm. From now on we drop the subscript $L^2(B)$ when the inner product is in $L^2(B)$ and will explicitly indicate a subscript for other cases. Another operator $\mathcal{T}_\Omega$ is needed for the factorization, namely $\mathcal{T}_\Omega: L^2(\Omega) \to L^2(\Omega)$ which is given by 
\begin{equation} \label{Section operator T_omega def}
 \mathcal{T}_\Omega f   :=q f, \quad \forall f \in L^2(\Omega).
\end{equation}

Now we are ready to prove the factorization theorem.
\begin{theorem} \label{Section operator Omega fac theorem}
Let the data operator $\mathcal{N}: L^2(B) \to L^2(B)$ be given by \eref{Section operator N def}. Then it holds that
\begin{equation*}
\mathcal{N} = \mathcal{S}^*_\Omega \mathcal{T}_\Omega \mathcal{S}_\Omega
\end{equation*}
where $\mathcal{S}_\Omega$, $\mathcal{S}^*_\Omega$, and $\mathcal{T}_\Omega$ are given by \eref{Section operator S_omega def}, \eref{Section operator S*_omega def}, and \eref{Section operator T_omega def}, respectively. 
\end{theorem}
\begin{proof}
From the definition \eref{Section operator N def} of $\mathcal{N}$, one gets
\begin{eqnarray*}
\left( \mathcal{N} g \right)(t) &=& \int_{B} u\Big(\frac{t-s}{2} \Big) g(s) \ind s = \int_B \int_\Omega e^{ic (t-s) \cdot y} q(y) \ind y g(s) \ind s \\
&=& \int_\Omega \Big( \Big[ \int_B e^{-i c s \cdot y} g(s) \ind s  \Big] q(y) \Big) e^{ict\cdot y} \ind y = \Big(\mathcal{S}^*_\Omega \mathcal{T}_\Omega \mathcal{S}_\Omega g \Big)(t),
\end{eqnarray*}
for any $g\in L^2(B)$. This completes the proof. 
\proofend\end{proof}
 
Several properties hold. 
\begin{prop}  \label{Section operator Omega S_Omega compact injective dense range prop}
The operator $\mathcal{S}_\Omega: L^2(B) \to L^2(\Omega)$ is compact, injective and has dense range.
The operator $\mathcal{S}^*_\Omega: L^2(\Omega) \to L^2(B)$ is compact, injective and has dense range in $L^2(B)$. 
\end{prop}
\begin{proof}
Note that the kernel $e^{ict\cdot y}$ is analytic, then both $\mathcal{S}_\Omega$ and $\mathcal{S}^*_\Omega$ are compact. Note that $\mathcal{S}_\Omega = (\mathcal{F}^c_d)^*|_{\Omega}$ and $\Omega$ has non-empty interior,  then $\mathcal{S}_\Omega$ is injective follows directly from that $(\mathcal{F}^c_d)^*\varphi$  coincide with an analytic function, namely the Fourier transform of a function, $\varphi \in L^2(\mathbb R^d)$,  with compact support in $B$. This yields that $\mathcal{S}^*_\Omega$ has dense range in $L^2(B)$. Reversing $\mathcal{S}_\Omega$ and $\mathcal{S}^*_\Omega$ in the previous arguments give the injectivity of $\mathcal{S}^*_\Omega$. This completes the proof.\proofend
\end{proof}

\begin{prop}\label{Section operator middle positive definite prop}
Assume that $q_{\# sup} \ge \Re [ e^{i\tau} q(y)] \ge q_{\# inf}$, a.e., $y \in \Omega$, for some positive constants $q_{\#inf/sup}$ and some constant phase $\tau \in (-\pi,\pi]$. Then $\Re [ e^{i\tau}\mathcal{T}_\Omega]$ is self-adjoint and positive definite.
\end{prop}
\begin{proof}
From the definition \eref{Section operator T_omega def} of $\mathcal{T}_\Omega$, one has for any $f \in L^2(\Omega)$
\begin{eqnarray*}
\Re [ e^{i\tau}\mathcal{T}_\Omega] f = \frac{e^{i\tau}\mathcal{T}_\Omega + (e^{i\tau}\mathcal{T}_\Omega)^*}{2} f = \frac{e^{i\tau} q f +e^{-i\tau} \overline{q} f}{2} = \Re [ e^{i\tau} q] f,
\end{eqnarray*}
 then $\Re [ e^{i\tau}\mathcal{T}_\Omega]$ is self-adjoint and
 \begin{eqnarray*}
\langle \Re [ e^{i\tau}\mathcal{T}_\Omega] f,f \rangle_{L^2(\Omega)} \ge q_{\#inf} \|f\|^2_{L^2(\Omega)}
\end{eqnarray*}
which completes the proof.
\proofend\end{proof}

To proceed with the factorization method and linear sampling method, one works with another operator $\mathcal{N}_\#: L^2(B) \to L^2(B)$ given by
$
\mathcal{N}_\# = \Re [ e^{i\tau} \mathcal{N}].
$
To conveniently illustrate how the linear sampling and factorization methods go beyond shape identification, we choose to work in the case that 
\begin{equation} \label{Section parameter q assumption}
q_{\sup} \ge q \ge q_{\inf}>0 \quad \mbox{a.e. in } \Omega
\end{equation}
for some positive constants $q_{inf/sup}$. This is assumed throughout the remaining of this paper.

Note that $\mathcal{N}$ is self-adjoint and positive definite due to assumption \eref{Section parameter q assumption}. We now state the following lemma on range identity.

\begin{lemma}\label{Section operator Omega range characterization prop}
Assume that \eref{Section parameter q assumption} holds. Then it follows that
$Range(\mathcal{S}^*_\Omega) = Range(\mathcal{N}^{1/2})$.
\end{lemma}
\begin{proof}
Note that the middle operator $\mathcal{T}_\Omega$ is positive definite and self-adjoint, then the proof follows from \cite[Corollary 1.22]{kirsch2008factorization} and Proposition \ref{Section operator Omega S_Omega compact injective dense range prop}.
\proofend\end{proof}
\begin{remark}
The above factorization in Theorem \ref{Section operator Omega fac theorem} gives a factorization of the data operator. One can get another factorization of the data operator as follows. 
Introduce $\mathcal{S}: L^2(B) \to L^2(B)$ by 
\begin{equation} \label{Section operator S def}
\left( \mathcal{S} g \right)(y) := \int_{B} e^{-i cs \cdot y} g(s) \ind s, \quad \forall g \in L^2(B), \quad y \in B,
\end{equation}
and it follows directly that its adjoint $\mathcal{S}^*: L^2(B) \to L^2(B)$ is given by
\begin{equation} \label{Section operator S* def}
\left( \mathcal{S}^* h \right)(t) := \int_{B} e^{i c t \cdot y} h(y) \ind y, \quad \forall h \in L^2(B), \quad t \in B
\end{equation}
which is dictated by $\langle \mathcal{S}^*  h,g \rangle =  \langle  h, \mathcal{S}^* g \rangle$. Note that the operator $\mathcal{S}^*$ is nothing but the operator $\mathcal{F}^c_d$. Furthermore, we introduce  $\mathcal{T}: L^2(B) \to L^2(B)$ by 
\begin{equation} \label{Section operator T def}
 \mathcal{T}  f   :=\underline{q} f, \quad \forall f \in L^2(B),
\end{equation}
where $\underline{q}$ is the extension of $q$ given by
\begin{equation} \label{Section operator underline q def}
\underline{q}(x) := 
\left\{
\begin{array}{cc}
q(x)  & x\in \Omega    \\
0  &   \mbox{otherwise}
\end{array}
\right..
\end{equation}
It follows directly that  
\begin{equation} \label{Section operator fac theorem eqn}
\mathcal{N} = \mathcal{S}^* \mathcal{T} \mathcal{S}
\end{equation}
where $\mathcal{S}$, $\mathcal{S}^*$, and $\mathcal{T}$ are given by \eref{Section operator S def}, \eref{Section operator S* def}, and \eref{Section operator T def}, respectively. 
It is noted that the middle operator $\mathcal{T}$ is no longer positive definite unless $\Omega=B$, this is in contrast to the first factorization $\mathcal{S}^*_\Omega \mathcal{T}_\Omega \mathcal{S}_\Omega$ where $\mathcal{T}_\Omega$ is positive definite. It is also noted that  $\mathcal{S}^*$ and $\mathcal{S}$  are parameter-independent.
\end{remark}

For later purposes, we introduce the eigensystem $\{ \zeta_n,\mu_n\}_{n=0}^\infty$ of the self-adjoint, positive definite operator $\mathcal{N}$ by
\begin{equation} \label{Section operator Omega eigensystem}
\mathcal{N} \zeta_n = \mu_n \zeta_n, \quad n=0,1,\cdots,
\end{equation}
here $\zeta_n \in L^2(B)$ and $\mu_0 \ge \mu_1 \ge \cdots >0$.

We would also like to stress that PSWFs are an interesting object to analyse the data operator $\mathcal{N}$. Indeed we would like to prove that it sort of compresses the operator $\mathcal{N}$ since
\begin{eqnarray*}
&&\big|  \big< \mathcal{N}  \psi_j(\cdot;c ), \psi_i(\cdot;c )  \big> \big|
= \big| \big< \mathcal{T}  \lambda_j \psi_j(\cdot;c ),  \lambda_i \psi_i(\cdot;c )  \big>\big|  \\
&=&  \big| \lambda_j  \bar \lambda_i \int_{ \Omega} \psi_j(y ; c) \bar \psi_i(y;c )  q(y) \ind y \big|  \leq  \big| \lambda_j  \lambda_i \big|  \|q\|_{L^\infty(\Omega)},
\end{eqnarray*}
where in the last step we applied the Cauchy-Schwartz inequality twice and the fact that $\{\psi_k(\cdot;c )\}_{k \in \mathbbm{N}_d}$ is an orthonormal set.

Therefore as evidenced by the super fast decay of the prolate eigenvalues (see, for instance,  \cite[equation 2.17]{wang10}) where
\begin{equation*} 
|\lambda_j(c)\lambda_i(c)| \sim e^{(j+1/2) \left(\log\frac{ec}{4}-\log(j+1/2) \right)+(i+1/2) \left(\log\frac{ec}{4}-\log(i+1/2) \right)}, \quad i,j \gg 1.
\end{equation*}
we can deduce that the operator $\mathcal{N}$ has a compressed representation in the basis $\left\{\psi_j(\cdot;c ) \right\}_{j\in \mathbb{N}}$. This could allow to speed up the computation by truncated  $\mathcal{N}$ or to help for denoising the data.

\section{Linear sampling and factorization methods for shape identification} \label{Section FM and a LSM}
In this section we study the factorization method, generalized linear sampling method and a formulation of the linear sampling method for shape identification. To begin with, let $\phi_z \in L^2(B)$ be given by 
\begin{equation} \label{Section FM and a LSM phi_z def}
\phi_z(t) := \int_{\mathbb{R}^d} e^{i c t \cdot y} E_z(y) \ind y,
\end{equation}

where  
\begin{equation} \label{Section FM and a LSM E_z def}
E_z := \frac{1}{|R(z,\epsilon)|} 1_{R(z,\epsilon)}, \quad 
1_{R(z,\epsilon)}(x):=\left\{
\begin{array}{cc}
1  & x\in R(z,\epsilon)   \\
0  &   \mbox{otherwise}
\end{array}
\right.,
\end{equation}
here $ |R(z,\epsilon)|$ is the length/area of $R(z,\epsilon)$ that satisfies $R_{inf}|\epsilon|^d\le|R(z,\epsilon)|\le R_{sup}|\epsilon|^d$ for some positive constants $R_{inf/sup}$. In practical applications, we usually choose an interval/square region $\{x=(x_1,\cdots,x_d): |x_j-z_j|< \epsilon,\,j=1,\cdots,d\}$ or an interval/disk region $\{x: |x-z|<\epsilon\}$. 

Throughout the paper we fix the parameter $\epsilon$ in the analysis and thereby chose to omit the dependence of $\phi_z$ and $E_z$ on $\epsilon$; we also sample the sampling point $z$   so that $\overline{R(z,\epsilon)}  \subset  B$ which is assumed later on.

The function $\phi_z$ allows us to characterize the support of $\Omega$. More precisely we have the following lemma.
\begin{lemma} \label{Section FM and a LSM lemma phi_z S^*_Omega}
Let $\overline{R(z,\epsilon)} \subset B$. The  following characterizations of the support $\Omega$ hold.
\begin{itemize}
\item If $\overline{R(z,\epsilon)}  \subset \Omega$, then $\phi_z \in Range(\mathcal{S}^*_\Omega)$.
\item If $ R(z,\epsilon)    \not\subset  \overline{\Omega}$, then $\phi_z \not\in Range(\mathcal{S}^*_\Omega)$. 
\end{itemize}

\end{lemma}
\begin{proof}
We first prove the first part. Let $\overline{R(z,\epsilon)}  \subset  \Omega$, then $E_z$ is supported in $\Omega$ so that according to \eref{Section operator S*_omega def} and \eref{Section FM and a LSM phi_z def}
\begin{equation*}
\Big(\mathcal{S}^*_\Omega E_z\Big) (t) = \int_{\Omega} e^{ict\cdot y } E_z(y) \ind y = \phi_z(t),
\end{equation*}
which shows that $\phi_z \in Range(\mathcal{S}^*_\Omega)$.

For the second part, let $ R(z,\epsilon)    \not\subset  \overline{\Omega}$ and we prove by showing that if $\mathcal{S}^*_\Omega h = \phi_z$ then $h \in L^2(\Omega)$ mush vanish. To show this we first extend $h$ to $\underline{h}$ that 
\begin{equation*}
\underline{h}(x) := 
\left\{
\begin{array}{cc}
h(x)  & x\in \Omega    \\
0  &   \mbox{otherwise}
\end{array}
\right.,
\end{equation*} 
then $\mathcal{F}^c_d \underline{h} = \mathcal{S}^* \underline{h} = \mathcal{S}^*_\Omega h = \phi_z$, i.e., $\underline{h}=(\mathcal{F}^c_d)^{-1} \phi_z$  which yields that  $\underline{h}= E_z$.
Note that the left hand side is supported in $\Omega$ but the right hand side    is not  supported in $\Omega$ (since $ R(z,\epsilon)    \not\subset  \overline{\Omega}$), this is a contradiction which shows that $h \in L^2(\Omega)$ mush vanish and this completes the proof.
\proofend\end{proof}

The linear sampling method (LSM) and factorization method (FM) for shape identification state the following.
\begin{itemize}
\item[(LSM)] The linear sampling method solves the data equation
$$
\mathcal{N} g_z \sim \phi_z
$$
using a regularization scheme to get a regularized solution $g_{z,\alpha}$ and indicates that
$$
\|g_{z,\alpha}\|_{L^2(B)}
$$
is large for $z$ with $R(z,\epsilon)    \not\subset  \overline{\Omega}$ and is bounded for $z$ with $\overline{R(z,\epsilon)}  \subset  \Omega$ (due to Proposition \ref{Section operator Omega range characterization prop} and Lemma \ref{Section FM and a LSM lemma phi_z S^*_Omega}). This  is suggested by a partial theory similar to \cite{ColtonKirsch96}; we omit this partial theory since we will show a formulation of the linear sampling method and the generalized linear sampling method with complete theoretical justification later on.
\item[(FM)] A direct application of Lemma \ref{Section operator Omega range characterization prop}
 and Lemma \ref{Section FM and a LSM lemma phi_z S^*_Omega} yields the factorization method: If $\overline{R(z,\epsilon)}  \subset \Omega$, then $\phi_z \in Range(\mathcal{N}^{1/2})$. If $R(z,\epsilon)    \not\subset  \overline{\Omega}$, then $\phi_z \not\in Range(\mathcal{N}^{1/2})$. Here
\begin{equation}\label{Section FM and a LSM FM main result}
\phi_z \in Range(\mathcal{N}^{1/2}) \Longleftrightarrow \sum_{n=0}^\infty \frac{|\langle \phi_z, \zeta_n \rangle|^2}{\mu_n} < \infty,
\end{equation}
where  $\{ \zeta_n,\mu_n\}_{n=0}^\infty$ is the eigensystem of  $\mathcal{N}$ given by \eref{Section operator Omega eigensystem}.
\end{itemize}

In this paper, we study a formulation of the linear sampling method in the form of 
$$
\langle \mathcal{S}g_{z,\alpha}, 1_{R(z,\epsilon)}  \rangle
$$
and we show later that such a LSM has capability in both shape and parameter identification. We will also show that the factorization method and the generalized linear sampling method also have capability in   both shape and parameter identification.  In this section we first demonstrate its viability in shape identification. The idea is similar to the earlier work \cite{arens04,arenslechleiter09,audiberthaddar15} in inverse scattering to justify/generalize the linear sampling method.

To begin with, we introduce a family of regularization schemes $\{ \mathcal{R}_\alpha\}_{\alpha>0}$ by
\begin{equation} \label{Section FM and a LSM R_alpha def}
\mathcal{R}_\alpha h := \sum_{n=0}^\infty f_\alpha(\mu_n^2) \mu_n \langle h, \zeta_n \rangle \zeta_n,
\end{equation}
where $f_\alpha$ is a regularizing filter that is a bounded, real-valued, and piecewise continuous function $f_\alpha: (0, \infty) \to \mathbb{R}$ such that
\begin{equation} \label{Section FM and a LSM f_alpha prop}
\lim_{\alpha \to 0 } f_\alpha(\mu) = \frac{1}{\mu} \mbox{ for all } \mu>0,  \quad |\mu f_\alpha(\mu)| \le d_0 \mbox{ for all } \alpha \ge 0 \mbox{ and } \mu>0,
\end{equation}
here $d_0>0$ is a constant. 

With this family of  regularization schemes $\{ \mathcal{R}_\alpha\}_{\alpha>0}$, one can introduce a family of regularized solutions by $g_{z,\alpha} = \mathcal{R}_\alpha \phi_z$.
Classical  regularizations include the Tikhonov regularization with 
$$
f_\alpha(\mu) \to \frac{1}{\mu+\alpha} \quad\mbox{ so that }\quad
g_{z,\alpha} \to  \sum_{n=0}^\infty \frac{\mu_n}{\mu_n^2 + \alpha} \langle \phi_z, \zeta_n \rangle \zeta_n,
$$
and the singular value cut off regularization with
$$
f_\alpha(\mu) \to  \left\{
\begin{array}{cc}
1/\mu  & \mu \ge \alpha    \\
0  &   \mbox{otherwise}
\end{array}
\right.
\quad \mbox{ so that }\quad
g_{z,\alpha} \to  \sum_{\mu_n \ge \alpha}^\infty \frac{1}{\mu_n} \langle \phi_z, \zeta_n \rangle \zeta_n.
$$
Our shape identification result is as follows.
\begin{theorem}\label{Section FM and a LSM main of LSM theorem}
Suppose that $\{ \mathcal{R}_\alpha\}_{\alpha>0}$ is a family of regularization schemes given by \eref{Section FM and a LSM R_alpha def}-- \eref{Section FM and a LSM f_alpha prop} and set $g_{z,\alpha}  = \mathcal{R}_\alpha \phi_z$. The   following characterizations of the support $\Omega$ hold.
\begin{itemize}
\item If $\overline{R(z,\epsilon)}  \subset  \Omega$, then $\langle \mathcal{S}g_{z,\alpha}, 1_{R(z,\epsilon)}  \rangle$ remains bounded as $\alpha \to 0$. Moreover 
\begin{equation} \label{Section FM and a LSM main of LSM theorem LSM=FM}
\lim_{\alpha \to 0} \langle \mathcal{S}g_{z,\alpha}, E_z  \rangle = \|g_z^{FM}\|^2
\end{equation}
where $g_z^{FM} \in L^2(B)$ is the unique solution to $\mathcal{N}^{1/2} g_z^{FM} = \phi_z$.
\item If $R(z,\epsilon)    \not\subset  \overline{\Omega}$, then  $\lim_{\alpha \to 0}\langle \mathcal{S}g_{z,\alpha}, 1_{R(z,\epsilon)}  \rangle = \infty$.
\end{itemize}

\end{theorem}
\begin{proof}
It is sufficient to prove the theorem for $\langle \mathcal{S}g_{z,\alpha}, E_z  \rangle$ since $E_z$ given by \eref{Section FM and a LSM E_z def} differs from $1_{R(z,\epsilon)}$ by a scaling. To begin with, we first remind the readers that one always has $\overline{R(z,\epsilon)}  \subset  B$.
We first derive an expression of $\langle \mathcal{S}g_{z,\alpha}, E_z \rangle$. From the definition of $g_{z,\alpha}$, one gets  $g_{z,\alpha}  = \mathcal{R}_\alpha \phi_z$ whereby
$g_{z,\alpha} =  \sum_{n=0}^\infty f_\alpha(\mu_n^2) \mu_n \langle \phi_z, \zeta_n \rangle \zeta_n$,
in this way we obtain
\begin{eqnarray} \label{Section FM and a LSM main of LSM theorem proof expression}
&&\langle \mathcal{S}g_{z,\alpha}, E_z  \rangle =\langle g_{z,\alpha}, \mathcal{S}^*E_z  \rangle = \langle g_{z,\alpha},  \phi_z \rangle = \sum_{n=0}^\infty f_\alpha(\mu_n^2) \mu_n |\langle \phi_z, \zeta_n \rangle|^2.
\end{eqnarray}

We now prove the first part. Let $\overline{R(z,\epsilon)}  \subset  \Omega$, then from the factorization method result \eref{Section FM and a LSM FM main result} one can obtain that there  exists the unique solution $g_z^{FM} \in L^2(B)$ to $\mathcal{N}^{1/2} g_z^{FM} = \phi_z$   and
$
 \|g_z^{FM}\|^2 = \sum_{n=0}^\infty \frac{|\langle \phi_z, \zeta_n \rangle|^2}{\mu_n} < \infty$.
Note that $f_\alpha$ satisfies \eref{Section FM and a LSM f_alpha prop}, then we have from \eref{Section FM and a LSM main of LSM theorem proof expression} that
$$
\langle \mathcal{S}g_{z,\alpha}, E_z  \rangle  = \sum_{n=0}^\infty f_\alpha(\mu_n^2) \mu_n |\langle \phi_z, \zeta_n \rangle|^2  \le d_0  \sum_{n=0}^\infty \frac{|\langle \phi_z, \zeta_n \rangle|^2}{\mu_n} < \infty.
$$
i.e., $\langle \mathcal{S}g_{z,\alpha}, 1_{R(z,\epsilon)}  \rangle$ remains bounded as $\alpha \to 0$. Then from the dominated convergence theorem we can take the limit $\alpha \to 0$ so that
\begin{equation*}
\lim_{\alpha \to 0}   \langle \mathcal{S}g_{z,\alpha}, E_z  \rangle =\sum_{n=0}^\infty \big(\lim_{\alpha \to 0}f_\alpha(\mu_n^2) \mu^2_n\big) \frac{1}{\mu_n} |\langle \phi_z, \zeta_n \rangle|^2 = \sum_{n=0}^\infty  \frac{1}{\mu_n} |\langle \phi_z, \zeta_n \rangle|^2  = \|g_z^{FM}\|^2.
\end{equation*}
This proves the first part.

For the second part when $R(z,\epsilon)    \not\subset  \overline{\Omega}$, first note from the factorization method result \eref{Section FM and a LSM FM main result} that $\phi_z \not\in Range(\mathcal{N}^{1/2})$ so that 
$\sum_{n=0}^\infty \frac{|\langle \phi_z, \zeta_n \rangle|^2}{\mu_n} = \infty$,
then for any large $M>0$, there exists $N_M>0$ such that 
$
\sum_{n=0}^{N_M} \frac{|\langle \phi_z, \zeta_n \rangle|^2}{\mu_n} >2M
$.
Now we chose $\alpha_M>0$ (due to the property of $f_\alpha$ in  \eref{Section FM and a LSM f_alpha prop}) such that
$
f_\alpha(\mu_n^2) > \frac{1}{2\mu_n^2}, \,\forall \alpha \in (0,\alpha_M),\, \mbox{for all }  n=0,1,\cdots,N_M$,
this yields that  for any large $M>0$, there exists $\alpha_M>0$ such that
$$
\langle \mathcal{S}g_{z,\alpha}, E_z  \rangle \ge \sum_{n=0}^{N_M} f_\alpha(\mu_n^2) \mu_n |\langle \phi_z, \zeta_n \rangle|^2 > \frac{1}{2}\sum_{n=0}^{N_M} \frac{|\langle \phi_z, \zeta_n \rangle|^2}{\mu_n} >M,\quad \forall \alpha \in (0,\alpha_M).
$$
This proves $\lim_{\alpha \to 0}\langle \mathcal{S}g_{z,\alpha}, E_z  \rangle = \infty$, i.e., $\lim_{\alpha \to 0}\langle \mathcal{S}g_{z,\alpha}, 1_{R(z,\epsilon)}  \rangle = \infty$ which completes the proof.
\proofend\end{proof}

From the above theorem and its proof, one can also prove in the same way the following result, which uses the indicator introduces in the generalized linear sampling method first proposed by \cite{audiberthaddar15}.
\begin{theorem}\label{Section FM and a LSM main of GLSM theorem}
 Suppose that $\{ \mathcal{R}_\alpha\}_{\alpha>0}$ is a family of regularization schemes given by \eref{Section FM and a LSM R_alpha def}-- \eref{Section FM and a LSM f_alpha prop} and set $g_{z,\alpha}  = \mathcal{R}_\alpha \phi_z$. The   following characterizations of the support $\Omega$ hold.
\begin{itemize}
\item If $\overline{R(z,\epsilon)}  \subset  \Omega$, then $\langle  g_{z,\alpha},  \mathcal{N} g_{z,\alpha} \rangle$ remains bounded as $\alpha \to 0$. Moreover 
\begin{equation} \label{Section FM and a LSM main of LSM theorem GLSM=FM}
\lim_{\alpha \to 0} \langle  g_{z,\alpha},  \mathcal{N} g_{z,\alpha}  \rangle = \|g_z^{FM}\|^2
\end{equation}
where $g_z^{FM} \in L^2(B)$ is the unique solution to $\mathcal{N}^{1/2} g_z^{FM} = \phi_z$.
\item If $R(z,\epsilon)    \not\subset  \overline{\Omega}$, then  $\lim_{\alpha \to 0}  \langle  g_{z,\alpha},  \mathcal{N} g_{z,\alpha}  \rangle = \infty$.
\end{itemize}
\end{theorem}
\begin{proof}
The proof is almost the same as the proof of Theorem \ref{Section FM and a LSM main of LSM theorem}. We omit the details but only highlight the difference. The  difference arises from
\begin{eqnarray*}
&& \langle  g_{z,\alpha},  \mathcal{N} g_{z,\alpha}  \rangle   = \sum_{n=0}^\infty |f_\alpha(\mu_n^2)|^2 \mu_n^3 |\langle \phi_z, \zeta_n \rangle|^2,
\end{eqnarray*}
and one can complete the proof by following line by line the proof of Theorem \ref{Section FM and a LSM main of LSM theorem}.  
\proofend\end{proof}

Theorem \ref{Section FM and a LSM main of LSM theorem} and Theorem \ref{Section FM and a LSM main of GLSM theorem} allow to determine an $\epsilon$ neighborhood of the support  $\Omega$ since it is capable of determining whether $\overline{R(z,\epsilon)}  \subset  \Omega$ or $R(z,\epsilon)    \not\subset  \overline{\Omega}$, as is similar to \cite{GriesmaierSchmiedecke-source,meng22}. 
We show in the next section that such a formulation has the capability in parameter identification in addition to shape identification.

\section{Linear sampling and factorization methods   for parameter identification} \label{Section parameter LSM}
In this section we demonstrate that the linear sampling  and factorization methods have capability in parameter identification.  
We first prove the following lemma.
\begin{lemma}\label{Section parameter LSM lemma}
Suppose that $\{ \mathcal{R}_\alpha\}_{\alpha>0}$ is a family of regularization schemes given by \eref{Section FM and a LSM R_alpha def}-- \eref{Section FM and a LSM f_alpha prop} and set $g_{z,\alpha}  = \mathcal{R}_\alpha \phi_z$. Then it holds for any   $\overline{R(z,\epsilon)}  \subset  \Omega$ that 
\begin{equation*}
\|\mathcal{S}_\Omega g_{z,\alpha}\|_{L^2(\Omega)} \le \frac{d_0}{q_{inf}} \|E_z\|_{L^2(\Omega)}=\frac{d_0}{q_{inf}{ \sqrt{|R(z,\epsilon)|}}}, \quad \forall \alpha>0,
\end{equation*}
where $d_0$ given by \eref{Section FM and a LSM f_alpha prop} is a constant independent of $\alpha$, and $q_{inf}>0$ given by \eref{Section parameter q assumption} is the lower bound of $q$.
\end{lemma}
\begin{proof}
Note that  $\mathcal{T}_\Omega$ is self-adjoint and bounded below by $q_{\inf} I$ (here $I$ is the identity operator), then it follows that  
\begin{eqnarray}\label{Section parameter LSM lemma proof eqn1}
&& q_{\inf} \|\mathcal{S}_\Omega  g_{z,\alpha} \|^2_{L^2(\Omega)} \le  \langle \mathcal{T}_\Omega \mathcal{S}_\Omega  g_{z,\alpha},   \mathcal{S}_\Omega  g_{z,\alpha}\rangle_{L^2(\Omega)}= \langle \mathcal{N}  g_{z,\alpha},   g_{z,\alpha} \rangle.\end{eqnarray}
Note that $g_{z,\alpha}$ is given by  $g_{z,\alpha}  = \mathcal{R}_\alpha \phi_z$ whereby
\begin{equation*}  
g_{z,\alpha} =  \sum_{n=0}^\infty f_\alpha(\mu_n^2) \mu_n \langle \phi_z, \zeta_n \rangle \zeta_n, \qquad 
\mathcal{N}g_{z,\alpha} =  \sum_{n=0}^\infty f_\alpha(\mu_n^2) \mu^2_n \langle \phi_z, \zeta_n \rangle \zeta_n,
\end{equation*}
which yields (where one notes that $f_\alpha$ and $\mu_n$ are real-valued)
\begin{equation*}  
\hspace{-5em}\langle \mathcal{N}  g_{z,\alpha},   g_{z,\alpha} \rangle =  \sum_{n=0}^\infty [f_\alpha(\mu_n^2)]^2 \mu^3_n |\langle \phi_z, \zeta_n \rangle|^2 
\le
d_0\sum_{n=0}^\infty f_\alpha(\mu_n^2) \mu_n |\langle \phi_z, \zeta_n \rangle|^2 = d_0 \langle \phi_z,   g_{z,\alpha} \rangle,
\end{equation*}
this together with $\phi_z = \mathcal{S}^* E_z$ yields that
\begin{equation}  \label{Section parameter LSM lemma proof eqn2}
\langle \mathcal{N}  g_{z,\alpha},   g_{z,\alpha} \rangle \le d_0 \langle \phi_z,   g_{z,\alpha} \rangle = d_0 \langle \mathcal{S}^* E_z,   g_{z,\alpha} \rangle =d_0 \langle E_z,   \mathcal{S} g_{z,\alpha} \rangle 
<\infty
\end{equation}
where the last step is due to that $E_z$ is supported in $\overline{R(z,\epsilon)}  \subset  \Omega$.

Now combining \eref{Section parameter LSM lemma proof eqn1}--\eref{Section parameter LSM lemma proof eqn2} we have that $ \|\mathcal{S}_\Omega  g_{z,\alpha} \|^2_{L^2(\Omega)} <\infty$ and
\begin{eqnarray*}
&&q_{\inf} \|\mathcal{S}_\Omega g_{z,\alpha}\|^2_{L^2(\Omega)} \le \langle \mathcal{N}  g_{z,\alpha},   g_{z,\alpha} \rangle \le  d_0 |\langle E_z,   \mathcal{S} g_{z,\alpha}  \rangle| = d_0 |\langle E_z,   \mathcal{S}_\Omega g_{z,\alpha} \rangle_{L^2(\Omega)}| \\
&\le& d_0 \|E_z\|_{L^2(\Omega)}    \|\mathcal{S}_\Omega g_{z,\alpha}\|_{L^2(\Omega)},
\end{eqnarray*}
this proves the lemma by noting that $\|E_z\|_{L^2(\Omega)} = \|E_z\|_{L^2(B)}={  \frac{1}{ \sqrt{|R(z,\epsilon)|}}}$ due to the definition of $E_z$ in \eref{Section FM and a LSM E_z def}.  
\proofend\end{proof}

Now we are ready to prove the parameter identification theorem. For convenience we let $q^+$ be the pseudo inverse of $q$ given by
\begin{equation} \label{Section parameter LSM q+ def}
q^+:=\left\{
\begin{array}{cc}
1/q  & \Omega    \\
0  &   \mbox{otherwise}
\end{array}
\right..
\end{equation}
The following (disk) PSWFs expansion of $E_z$ and $\phi_z$ will be often used.
Let $\overline{R(z,\epsilon)}  \subset  B$ and let the (disk) PSWFs expansion of $E_z$ be
\begin{equation} \label{Section FM and a LSM E_z PSWF expansion}
E_z = \sum_{j \in \mathbbm{N}_d} \langle E_z, \psi_j(\cdot;c) \rangle  \psi_j(\cdot;c) \mbox{ in } L^2(B),
\end{equation}
then it follows directly from \eref{Section FM and a LSM phi_z def} that $\phi_z = \mathcal{F}^c_d E_z$ and \eref{Section PSWF multi-d psi_n def} so that
\begin{equation} \label{Section FM and a LSM phi_z PSWF expansion}
\phi_z =\sum_{j \in \mathbbm{N}_d}  \langle \phi_z, \psi_j(\cdot;c) \rangle  \psi_j(\cdot;c)= \sum_{j \in \mathbbm{N}_d} \lambda_j(c)\langle E_z, \psi_j(\cdot;c) \rangle  \psi_j(\cdot;c) \mbox{ in } L^2(B).
\end{equation}

\begin{theorem}\label{Section parameter LSM theorem}
Suppose that $\{ \mathcal{R}_\alpha\}_{\alpha>0}$ is a family of regularization schemes given by \eref{Section FM and a LSM R_alpha def}-- \eref{Section FM and a LSM f_alpha prop} and set $g_{z,\alpha}  = \mathcal{R}_\alpha \phi_z$. Then it holds for any  $\overline{R(z,\epsilon)}  \subset  \Omega$ that
\begin{equation}\label{Section parameter LSM theorem identity 1}
\lim_{\alpha \to 0}\langle \mathcal{S}g_{z,\alpha}, 1_{R(z,\epsilon)}  \rangle = \langle 1/q, E_z\rangle_{L^2(R(z,\epsilon))}.
\end{equation}
\end{theorem}
\begin{proof}
Throughout the proof, we let $\underline{f} \in L^2(B)$ denote the extension of a generic function $f\in L^2(\Omega)$ by setting $f=0$ in $B\backslash \overline{\Omega}$.

(a). We first give representations for $\langle \mathcal{S}g_{z,\alpha}, 1_{R(z,\epsilon)}  \rangle$ and $ \langle 1/q, E_z\rangle_{L^2(R(z,\epsilon))}$, respectively. For the regularized solution $g_{z,\alpha}  = \mathcal{R}_\alpha \phi_z$, we get the (disk) PSWFs expansion of $\underline{q \mathcal{S}_\Omega g_{z,\alpha}}$ by
$
 \underline{q \mathcal{S}_\Omega g_{z,\alpha}} =    \sum_{j \in \mathbbm{N}_d} \langle \underline{q \mathcal{S}_\Omega g_{z,\alpha}},  \psi_j(\cdot;c)\rangle   \psi_j(\cdot;c)$,
note that  $q$ is supported in $\Omega$, then we further get (where we recall $q^+$ is given by \eref{Section parameter LSM q+ def})
\begin{equation*}
 \mathcal{S}_\Omega g_{z,\alpha}  = \sum_{j \in \mathbbm{N}_d} \langle  \underline{q \mathcal{S}_\Omega g_{z,\alpha}},  \psi_j(\cdot;c)\rangle q^{+}  \psi_j(\cdot;c) \quad \mbox{in} \quad L^2(\Omega),
\end{equation*}
so that (by noting that $\overline{R(z,\epsilon)}  \subset  \Omega$, $q^+$ and $\psi_j(\cdot;c)$ are real-valued)
\begin{eqnarray}\label{Section parameter LSM theorem proof eqn1}
 \langle \mathcal{S}  g_{z,\alpha}, 1_{R(z,\epsilon)}\rangle &=& \langle \mathcal{S}_\Omega g_{z,\alpha}, 1_{R(z,\epsilon)}\rangle  = \sum_{j \in \mathbbm{N}_d} \langle \underline{q \mathcal{S}_\Omega g_{z,\alpha}},  \psi_j(\cdot;c)\rangle \langle q^{+}   \psi_j(\cdot;c), 1_{R(z,\epsilon)} \rangle \nonumber \\
&=& \sum_{j \in \mathbbm{N}_d}   \langle 1_{R(z,\epsilon)}q^{+} ,  \psi_j(\cdot;c)  \rangle \langle  \underline{q \mathcal{S}_\Omega g_{z,\alpha}},  \psi_j(\cdot;c)\rangle.
\end{eqnarray}

On the other hand,
\begin{equation}\label{Section parameter LSM theorem proof eqn2}
 \langle 1/q, E_z\rangle_{L^2(R(z,\epsilon))} =  \langle 1_{R(z,\epsilon)} q^+, E_z\rangle  =   \sum_{j \in \mathbbm{N}_d} \langle 1_{R(z,\epsilon)} q^+, \psi_j(\cdot;c)\rangle  \langle  E_z, \psi_j(\cdot;c)\rangle,
\end{equation}
where in the last step we used the (disk) PSWFs expansion of $1_{R(z,\epsilon)} q^+$ and \eref{Section FM and a LSM E_z PSWF expansion} to evaluate their inner product, and the fact that $E_z$ and $\psi_j(\cdot;c)$ are real-valued. 

(b). Note that $1_{R(z,\epsilon)} q^+ \in L^2(B)$  and that the $ L^2(B)$-norm of $ \underline{q\mathcal{S}_\Omega g_{z,\alpha}}$ is bounded uniformly with respect to $\alpha$ (due to Lemma \ref{Section parameter LSM lemma}), then the infinite series in \eref{Section parameter LSM theorem proof eqn1} is uniformly convergent. By the dominated convergence theorem, one then gets
\begin{eqnarray}  \label{Section parameter LSM theorem proof eqn3}
\lim_{\alpha \to 0} \langle \mathcal{S}  g_{z,\alpha}, 1_{R(z,\epsilon)}\rangle    &=& \sum_{j \in \mathbbm{N}_d}   \langle 1_{R(z,\epsilon)}q^{+} ,  \psi_j(\cdot;c)  \rangle \lim_{\alpha \to 0}  \langle  \underline{q \mathcal{S}_\Omega g_{z,\alpha}},  \psi_j(\cdot;c)\rangle.
\end{eqnarray}
By noting that 
$$
\mathcal{S} \psi_j(\cdot;c)= \overline{\lambda_j(c)}  \psi_j(\cdot;c), \quad  \lim_{\alpha \to 0}\mathcal{N}  g_{z,\alpha}= \lim_{\alpha \to 0} \sum_{n=0}^\infty f_\alpha(\mu_n^2) \mu^2_n \langle \phi_z, \zeta_n \rangle \zeta_n =\phi_z,
$$
one has
 \begin{eqnarray*} 
 \hspace{-5em}
 \lim_{\alpha \to 0}  \langle  \underline{q \mathcal{S}_\Omega g_{z,\alpha}},  \psi_j(\cdot;c)\rangle &=&  \lim_{\alpha \to 0}  \langle  \mathcal{S}^*\underline{q \mathcal{S}_\Omega g_{z,\alpha}},  \overline{\lambda_j(c)^{-1}}  \psi_j(\cdot;c)\rangle = \lim_{\alpha \to 0}  \langle  \mathcal{N} g_{z,\alpha},  \overline{\lambda_j(c)^{-1}}  \psi_j(\cdot;c)\rangle \\
 &=&  \langle  \phi_z,  \overline{\lambda_j(c)^{-1}}  \psi_j(\cdot;c)\rangle =  \langle \mathcal{S}^* E_z,  \overline{\lambda_j(c)^{-1}}  \psi_j(\cdot;c)\rangle = \langle  E_z,    \psi_j(\cdot;c)\rangle.
 \end{eqnarray*}
This equation together with \eref{Section parameter LSM theorem proof eqn3} yield that
\begin{eqnarray*}  
\lim_{\alpha \to 0} \langle \mathcal{S}  g_{z,\alpha}, 1_{R(z,\epsilon)}\rangle    &=& \sum_{j \in \mathbbm{N}_d}   \langle 1_{R(z,\epsilon)}q^{+} ,  \psi_j(\cdot;c)  \rangle  \langle  E_z,    \psi_j(\cdot;c)\rangle 
=
\langle 1/q, E_z\rangle_{L^2(R(z,\epsilon))} . 
\end{eqnarray*}
This completes the proof.
\proofend\end{proof}

\begin{remark}
It is possible to prove the result of Theorem \ref{Section parameter LSM theorem} using the singular system of $\mathcal{S}_\Omega$. We include such a proof  since this idea is expected to be generalized to other inverse scattering problems.

\begin{proof}[Alternate Proof of Theorem \ref{Section parameter LSM theorem}]
The proof is very similar to the previous one except that we expand the quantity of interest with respect to an orthonormal basis $\{\beta_n\}_{n \in \mathbbm{N}_d} \in L^2(\Omega)$ given by $\mathcal{S}_\Omega\mathcal{S}^*_\Omega \beta_n = \gamma_n \beta_n$ and denote by $\alpha_n = \mathcal{S}^*_\Omega \beta_n$.
First we have the following expression
\begin{equation*}
 \mathcal{S}_\Omega g_{z,\alpha}  = \sum_{j \in \mathbbm{N}_d} \langle  {q \mathcal{S}_\Omega g_{z,\alpha}},  \beta_n\rangle_{L^2(\Omega)} q^{-1} \beta_n \quad \mbox{in} \quad L^2(\Omega),
\end{equation*}
then we have the intermediate result :
 \begin{eqnarray*} 
 \hspace{-5em}
 \lim_{\alpha \to 0}  \langle  {q \mathcal{S}_\Omega g_{z,\alpha}},  \beta_n\rangle _{L^2(\Omega)}&=&  \lim_{\alpha \to 0}  \langle  \mathcal{S}^*\underline{q \mathcal{S}_\Omega g_{z,\alpha}},  {\gamma_n^{-1}}  \alpha_n\rangle_{L^2(B)} = \lim_{\alpha \to 0}  \langle  \mathcal{N} g_{z,\alpha}, {\gamma_n^{-1}}  \alpha_n\rangle_{L^2(B)} \\
 &=&  \langle  \phi_z,   {\gamma_n^{-1}}  \alpha_n\rangle_{L^2(B)} =  \langle \mathcal{S}^* E_z,  {\gamma_n^{-1}}  \alpha_n\rangle_{L^2(B)}= \langle  E_z,    \beta_n\rangle_{L^2(\Omega)}.
 \end{eqnarray*}
 Finally using 
 \begin{equation*}
  \hspace{-5em}\langle q^{-1}, E_z\rangle_{L^2(R(z,\epsilon))} =  \langle 1_{R(z,\epsilon)} q^{-1}, E_z\rangle_{L^2(\Omega)}  =   \sum_{j \in \mathbbm{N}_d} \langle 1_{R(z,\epsilon)} q^{-1}, \beta_n\rangle_{L^2(\Omega)}  \langle  E_z, \beta_n\rangle_{L^2(\Omega)}.
\end{equation*}
 and combining the previous expression we obtain the previous results.
\proofend
\end{proof}

\end{remark}

 Note the connections between the linear sampling method and the factorization method, we can immediately obtain the following.
 \begin{theorem} 
 \label{Section parameter FM theorem}
 Suppose that $\{ \mathcal{R}_\alpha\}_{\alpha>0}$ is a family of regularization schemes given by \eref{Section FM and a LSM R_alpha def}-- \eref{Section FM and a LSM f_alpha prop} and set $g_{z,\alpha}  = \mathcal{R}_\alpha \phi_z$. 
For any  $\overline{R(z,\epsilon)}  \subset  \Omega$, let $g_z^{FM} \in L^2(B)$ be the unique solution to $\mathcal{N}^{1/2} g_z^{FM} = \phi_z$. Then it holds that
\begin{equation*}
|R(z,\epsilon)| \lim_{\alpha \to 0} \langle  g_{z,\alpha},  \mathcal{N} g_{z,\alpha}  \rangle  = |R(z,\epsilon)|\big\|g_z^{FM}\big\|^2 =  \langle 1/q, E_z\rangle_{L^2(R(z,\epsilon))},
\end{equation*}
where $|R(z,\epsilon)|$ denotes the length/area of $R(z,\epsilon)$.
 \end{theorem}
 \begin{proof}
 This follows directly from  \eref{Section FM and a LSM main of LSM theorem LSM=FM} and   \eref{Section FM and a LSM main of LSM theorem GLSM=FM},  Theorem \ref{Section parameter LSM theorem}, and $E_z = \frac{1}{|R(z,\epsilon)|} 1_{R(z,\epsilon)}$.
 \proofend\end{proof}
 
 \begin{remark}
 We highlight the fact that Theorem  \ref{Section FM and a LSM main of LSM theorem}  and \ref{Section FM and a LSM main of GLSM theorem} are able to determine the shape (almost) exactly. However our above result on parameter identification is "blind" near the boundary of the obstacle.
 
 \end{remark}
Note that Theorem \ref{Section FM and a LSM main of LSM theorem} is not valid when $q$ is not sign definite. It is mainly an open question in the general case to deal with sign changing contrast. Two types of results exist one \cite{kirsch2008factorization} when it is assume that one know in advance two domains that include respectively positive and negative sign definite contrast and the other \cite{GLSMsignChange} when $q$ changes sign strictly inside the support of the scatterer. The first approach could be extended straightforwardly to our cases however we won't pursue this as it will need additional apriori information. The second approach is not possible as our operator $T_\Omega$ is defined over $L^2(B)$ and not over more regular spaces that allow one to analyse the contribution of $q$ inside $\Omega$ as a compact perturbation.
Yet our results on parameter identification allow us to us to retrieve information even when $q$ changes sign even if the shape identification is not valid.  To do so one needs to introduced,
$$ \tilde q := q_{inf } 1_D + q 1_\Omega, $$  
where $D\subset B$ is a domain that contains $\Omega$. Clearly $\tilde q $ is positive definite its data operator $\tilde \mathcal{N} $  is given by
$$
(\tilde \mathcal{N}g) (t) =  \ \int_D \Big( \Big[ \int_B e^{-i c s \cdot y} g(s) \ind s  \Big] q_{inf }\Big) e^{ict\cdot y} \ind y    + (\mathcal{N}g) (t)=  (\mathcal{N}_{inf}g) (t)+ (\mathcal{N}g) (t)
$$

By applying Theorem \ref{Section parameter LSM theorem} to $\tilde \mathcal{N} $, one can retrieve information on $q$ through the following corollary. In spirit our method is related or inspired by imaging using differential measurements first introduce in \cite{DLSM}, here we compare $\tilde \mathcal{N}$ and $ \mathcal{N}_{inf}$.

\begin{corollary}\label{corollary 1}
We introduce $\{ \tilde \mathcal{R}_\alpha\}_{\alpha>0}$ a family of regularization schemes given by \eref{Section FM and a LSM R_alpha def}-- \eref{Section FM and a LSM f_alpha prop}  where one as substitute $ \mathcal{N} $ by $\tilde \mathcal{N} $and set $\tilde g_{z,\alpha}  = \tilde \mathcal{R}_\alpha \phi_z$. Then it holds for any  $\overline{R(z,\epsilon)}  \subset  D$ that
\begin{equation}\label{Section parameter LSM theorem identity tilde 1}
\lim_{\alpha \to 0}\langle \mathcal{S} \tilde g_{z,\alpha}, 1_{R(z,\epsilon)}  \rangle = \langle 1/\tilde q, E_z\rangle_{L^2(R(z,\epsilon))} =  \langle \frac{q_{inf}}{1+q/q_{inf}}, E_z\rangle_{L^2(R(z,\epsilon))} .
\end{equation}
\end{corollary}

\begin{remark}\label{remark num sign}
It could be of interest numerically to consider $g^{inf}_{z,\alpha}$ associated to $\mathcal{N}_{inf}$, which will give access to :
$$
\lim_{\alpha \to 0}\langle \mathcal{S}g^{inf}_{z,\alpha}, 1_{R(z,\epsilon)}  \rangle = \langle 1/q_{inf}, E_z\rangle_{L^2(R(z,\epsilon))}.
$$
Finally introducing the reconstruction formula
$$
I(z) = \langle \mathcal{S} \tilde g_{z,\alpha}, 1_{R(z,\epsilon)}  \rangle^{-1} - \langle \mathcal{S}g^{inf}_{z,\alpha}, 1_{R(z,\epsilon)}  \rangle^{-1}
$$
which will prove to give better numerical reconstruction.

\end{remark}

\section{An explicit example and numerical preliminaries} \label{Section numerical preliminaries}
In this section, we first study an explicit example to discuss the nature of the inverse problem. Our numerical experiments in shape and parameter identification later on will be based on   the inverse source problem with multi-frequency measurements for a fixed observation direction. This   motivates us to  discuss preliminaries for the computation of PSWFs and the evaluation of integrals involved in the prolate-Galerkin linear sampling method.
 
 \subsection{An explicit example}\label{Section numerics subsection explicit example}
In this section, we study an explicit example in one dimension where $q$ is constant one supported in $(-r,r) \subset (-1,1)$. In this regard, one can directly show that for all $j=0,1,\cdots$,
\begin{eqnarray*}
&&\big[ \mathcal{N}  \psi_j(\cdot;c r ) \big](t)  
=  \int_B \Big( \Big[ \int_B e^{-i c s \cdot y} \psi_j(s;c r ) \ind s  \Big] 1_{B(0,r)} (y) \Big) e^{ict\cdot y} \ind y \\
&=&  \int_{ B(0,r) } \Big( \Big[ \int_B e^{-i c s \cdot y} \psi_j(s ;c r) \ind s  \Big]   \Big) e^{ict\cdot y} \ind y  \quad (\mbox{change of variable } y \to y'r)\\
&=& r \int_{B} \Big( \Big[ \int_B e^{-i cr s \cdot y'} \psi_j(s ;c r) \ind s  \Big]   \Big) e^{icr t\cdot y'} \ind y' \\
&=&  r \overline{\lambda_j(cr)}  \int_B  \psi_j(y';cr)  e^{icrt\cdot y'} \ind y' =  r|\lambda_j(cr)|^2   \psi_j(t;cr)
\end{eqnarray*}
i.e.,
\begin{equation*}
 \mathcal{N}  \psi_j(\cdot ;cr) =  r |\lambda_j(cr)|^2 \psi_j(\cdot ;cr), \quad j=0,1,\cdots
\end{equation*}
where $\lambda_j(cr)$ is given by \eref{Section PSWF multi-d psi_n def} 
with $c$ replacing by $cr$. This gives the explicit eigensystem of $ \mathcal{N} $ for this particular case. This example is extremely simple but delivers several important messages.

The first message is that the inverse problem is challenging as evidenced by the super fast decay of the prolate eigenvalues (see, for instance,  \cite[equation 2.17]{wang10}) where
\begin{equation*} 
|\lambda_n(cr)| \sim e^{(n+1/2) \left(\log\frac{ecr}{4}-\log(n+1/2) \right)}, \quad n \gg 1.
\end{equation*}
This indicates that the smaller the radius $r$, the more ill-conditioned the inverse problem; it also indicates that the number of  eigenvalues (say for a range of $c\in(5,100)$) larger than machine epsilon is  limited (and we will see more in the numerical examples).

The second message is that one is necessarily led to the computation of the  eigensystem $\{\psi_j(\cdot ;cr), \lambda_j(cr)\}_{j=0}^\infty$ for this particular case and in general one needs appropriate quadrature rules for evaluating the integrals involved in the implement of LSM. Note that the PSWFs can be approximated by truncated Legendre series \cite{boyd05code}, the Legendre-Gauss-Lobatto (LGL) quadrature rule is a decent method (which requires more quadrature nodes than other prolate based Gaussian quadrature rules such as \cite{boyd05code}) that at least serves our needs in this paper. Note also that there exist Gaussian quadrature rules such as \cite{boyd05code} but this requires a little more computational efforts. However note that equidistant quadrature nodes may be less efficient  for approximating some integral equations (cf. \cite[Example 1.16]{Kirsch21}).

In the next subsections, we discuss the numerical approximation of the PSWFs eigensystem and introduce the Legendre-Gauss-Lobatto (LGL) quadrature rule for numerical evaluation of integrals involved in the implement of LSM. 

\subsection{Legendre polynomials and Legendre-Gauss-Lobatto  quadrature} \label{Section Legendre}
In the following we  introduce the Legendre-Gauss-Lobatto (LGL) quadrature rule with $N_q$ points that integrates all polynomials of degree less than or equal to $2{N_q}-3$ exactly (see \cite[Section 10.1--10.4]{quarteroni2000book} for more details).  Denote by $P_n(x)$ the Legendre polynomial of degree $n$ which satisfies the following recurrence relation
\begin{equation*}
\hspace{-5em}P_0(x)=1, \quad P_1(x) = x, \quad P_{n+1}(x) = \frac{1}{n+1}\left[ (2n+1)xP_n(x) - n P_{n-1}(x) \right], \quad n=1,2,\cdots,
\end{equation*}
and let $\overline{P}_n(x) := P_n(x) \sqrt{n+1/2}$ (where the overline bar associated with $P_n$ is not supposed to be confused with the conjugation) be the normalized Legendre polynomial where
\begin{equation} \label{Section Legendre orthogonality}
\int_{-1}^1 \overline{P}_m(x) \overline{P}_n(x)\ind x = \delta_{mn},
\end{equation}
with $\delta_{mn}$ denoting the Kronecker delta.

Let $x_0, \cdots, x_{N_q-1}$ be $N_q$ given distinct points over the interval $[-1,1]$, for the approximation of $I(f) = \int_{-1}^1 f(x) \ind x$, we consider quadrature rules of the type
\begin{equation*}
I_{N_q-1}(f) = \sum_{j=0}^{N_q-1} w_j f(x_j),
\end{equation*}
 where the points $x_j \in \mathbb{R}$ and coefficients $w_j \in \mathbb{R}$ are referred to as the nodes and weights of the quadrature, respectively.  The Legendre-Gauss-Lobatto (LGL) quadrature rule has nodes and weights given by
 \begin{eqnarray*} 
&&x_0=-1,\quad x_{N_q-1}=1, \quad x_j \mbox{ zeros of } P_{N_q-1}'(x), \quad j=1,\cdots,{N_q}-2, \label{Section Legendre LGL nodes}\\
&& w_j = \frac{2}{N_q(N_q-1)} \frac{1}{[P_{N_q-1}(x_j)]^2}. \label{Section Legendre LGL weights}
\end{eqnarray*}
The Legendre-Gauss-Lobatto quadrature rule, which includes the end points $-1$ and $1$, has degree of exactness $2{N_q}-3$, i.e., the quadrature formula integrates all polynomials of degree less than or equal to $2{N_q}-3$.
\subsection{Computation of PSWFs  system} \label{Section PSWFs computation}
One can approximate the PSWFs by the Legendre-Galerkin method and the coefficients are determined by solving a linear system with a symmetric, tridiagonal matrix. It is based on another remarkable property of PSWFs that
 $\{{\psi}_n\}_{n=0}^\infty$ are also eigenfunctions to the following Sturm-Liouville differential operator (cf. \cite[Section V]{Slepian61} or \cite[equation 2.1]{wang10})
\begin{equation}\label{Section PSWF 1d Dc eigensystem}
 \mathcal{D}^c_d {\psi}_n(\cdot;c)  ={\chi}_n(c) {\psi}_n(\cdot;c), \quad n=0,1,\cdots,
\end{equation}
where $ \mathcal{D}^c$ is the Sturm-Liouville differential operator given by
\begin{equation*}
\mathcal{D}^c_d := -\frac{d}{dx}(1-x^2) \frac{d}{dx} +c^2x^2, \quad x\in(-1,1), \quad d=1.
\end{equation*}
Here the corresponding Sturm-Liouville eigenvalues $\{{\chi}_n(c)\}$ are ordered in strictly increasing order and they satisfy 
\begin{equation*} 
n(n+1) <  {\chi}_n(c) <n(n+1) +c^2.
\end{equation*}

In particular to approximate the first $N+1$ PSWFs and Sturm-Liouville eigenvalues $\{\psi_n, \chi_n\}_{n=0}^N$, following \cite[Section 2]{boyd05code}, one expands
\begin{equation}  \label{Section PSWFs computation psi_n Legendre expansion}
\psi_n(x;c) \approx \sum_{j=0}^{N_t-1} B_{jn} \overline{P}_j (x), \quad n=0,1,\cdots,N,
\end{equation}
where $N_t$ determines the truncation of the Legendre series. We then substitute this expansion into the Sturm-Liouville problem \eref{Section PSWF 1d Dc eigensystem}  (and note that the Legendre polynomials satisfies this equation when $c=0$)   to get the linear system 
\begin{equation*} 
D b  = \chi_n^{approx} b 
\end{equation*}
where $\chi_n^{approx}$ is an approximation of the exact eigenvalue $\chi_n$, $b_j = B_{jn} $, and the matrix $D$ has non-zero entries given by
\begin{eqnarray*} 
D_{jj} &=& j(j+1)+ c^2 \frac{2j(j+1) -1}{(2j+3)  (2j-1)}, \quad j=0,1,\cdots \label{Section PSWFs computation matrix A def 1}\\
D_{j(j+2)} =D_{(j+2)j} &=& c^2 \frac{(j+1)(j+2)}{(2j+3) \sqrt{ (2j+1) (2j+5) }},\quad j=0,1,\cdots. \label{Section PSWFs computation matrix A def 2}
\end{eqnarray*}
\cite[Section 2]{boyd05code} suggested a truncation with $N_t \ge 2N+30$ to have a good approximation of the eigenvalue $\chi_n(c)$ and $\psi_n(\cdot;c)$.

After the evaluation of the PSWFs, one can compute the prolate eigenvalues as follows (cf. \cite[Section 2]{wang10}). First set $x=0$ in equation \eref{Section PSWF multi-d psi_n def}   to get
\begin{equation*}  
 \lambda_n(c)  \psi_n(0;c) =  \int_{-1}^1   \psi_n(y;c) \ind y \approx \sqrt{2} B_{0n},
\end{equation*}
where in the last step we applied \eref{Section PSWFs computation psi_n Legendre expansion} and \eref{Section Legendre orthogonality}. Note  that $\psi_n(x;c)$ is even for even $n$ and is odd for odd $n$ (see, for instance,  \cite[Section 2]{wang10}), thereby $\psi_n(0;c)$ vanishes for odd $n$ and we first get the approximation for eigenvalues with even indexes by
\begin{equation}  \label{Section PSWFs computation lambda_n computation even}
\lambda_{n}(c)  \approx     \frac{\sqrt{2} B_{0n}}{ \psi_n(0;c)}, \quad n \mbox{ even}.
\end{equation}
Similarly differentiating  \eref{Section PSWF multi-d psi_n def} allows us to get for odd $n$ that
\begin{equation}  \label{Section PSWFs computation lambda_n computation odd}
\lambda_{n}(c)  = \frac{ic}{  \partial_x \psi_n(0;c)} \int_{-1}^1 y \psi_n(y;c) \ind y \approx  \sqrt{\frac{2}{3}}\frac{ic B_{1n}}{  \partial_x \psi_n(0;c)}, \quad n \mbox{ odd}.
\end{equation}
In this paper the formulas \eref{Section PSWFs computation lambda_n computation even}--\eref{Section PSWFs computation lambda_n computation odd} are sufficient to help us implement the linear sampling method.

\subsection{A prolate-based formulation of the linear sampling method} \label{Section reduced LSM}
Note that we have highly accurate algorithms to compute the (disk) PSWFs system, 
in this section we propose a prolate-based formulation of LSM. To begin with, let $\mathbbm{J} \subset \mathbbm{N}_d$ be the following set
\begin{equation*}
\mathbbm{J}:=\{0,1,\cdots,J\} \quad (d=1) \quad \mbox{and}\quad  \mathbbm{J}:=\{(0,0,0),(0,0,1),\cdots,J\} \quad (d=2)
\end{equation*}
where we simply identify $J$ as the maximum index (which is a scalar index in one dimension and a multiple index in two dimensions).
We consider a prolate-Galerkin formulation by
\begin{equation} \label{Section reduced LSM linear system}
\sum_{j\in\mathbbm{J}} A^{\mathbbm{J}}_{j\ell} g^{\mathbbm{J}}_{zj} = \phi^{\mathbbm{J}}_{z\ell}, \quad \ell \in \mathbbm{J},
\end{equation}
where the data operator $\mathcal{N} $ and the z-dependent function $\phi_z$ gives the matrix and the right hand side by
\begin{equation} \label{Section reduced LSM matrix A def}
A^{\mathbbm{J}}_{j\ell}: = \langle \mathcal{N} \psi_j(\cdot,;c),  \psi_\ell(\cdot,;c) \rangle, \quad  j,\ell \in \mathbbm{J}, \quad\mbox{and}\quad \phi^{\mathbbm{J}}_{z\ell}=  \langle \phi_z,  \psi_\ell(\cdot,;c) \rangle, \quad  \ell \in \mathbbm{J}.
\end{equation}
This is a prolate-based formulation of the linear sampling method where we seek a reduced solution  in the span of (disk) PSWFs $\{\psi_j(\cdot;c)\}_{j\in \mathbbm{J}}$.
We further define
\begin{equation*} 
g^{\mathbbm{J}}_{z}:= \sum_{j\in\mathbbm{J}}g^{\mathbbm{J}}_{zj} \psi_j(\cdot;c), \quad \phi^{\mathbbm{J}}_{z}:= \sum_{j\in\mathbbm{J}}\phi^{\mathbbm{J}}_{zj} \psi_j(\cdot;c).
\end{equation*}

Let $g^{\mathbbm{J}}_{z,\alpha}$ be a family of regularized solution  obtained by regularizing \eref{Section reduced LSM linear system} with a family of regularization schemes (see, for instance,   Section \ref{Section FM and a LSM}; standard schemes include such as the Tikhonov regularization and the singular value cut off), then
according to Theorem \ref{Section FM and a LSM main of LSM theorem}, it is expected that the indicator function
$$
\langle \mathcal{S} g^{\mathbbm{J}}_{z,\alpha}, 1_{R(z,\epsilon)}  \rangle
$$
remains bounded as $\mathbbm{J} \to \mathbbm{N}_d$ and $\alpha \to 0$ for $\overline{R(z,\epsilon)} \subset \Omega$ and cannot be bounded as $\mathbbm{J} \to \mathbbm{N}_d$ and  $\alpha \to 0$  for $R(z,\epsilon)    \not\subset  \overline{\Omega}$.   Moreover according to Theorem \ref{Section parameter LSM theorem}, 
$$
\langle \mathcal{S} g^{\mathbbm{J}}_{z,\alpha}, 1_{R(z,\epsilon)}  \rangle \to  \langle 1/q, E_z\rangle_{L^2(R(z,\epsilon))}, \quad  \forall \overline{R(z,\epsilon)}\subset \Omega,
$$
as $\mathbbm{J} \to \mathbbm{N}_d$ and $\alpha \to 0$. It is also possible to establish a semi-explicit convergence result for parameter identification with both noiseless and noisy data, this is  ongoing work and will be reported in a forthcoming paper. Finally we remark that the prolate-based linear sampling method shares a similar spirit to the modal formulation of the linear sampling method in waveguide \cite{bourgeois2008linear}; see also \cite{meng2021}.

\section{Numerical experiments for parameter and shape identification} \label{Section numerical examples}
To demonstrate the   shape and parameter identification theory, in this section we perform relevant numerical experiments for the inverse source problem with multi-frequency measurements for a fixed observation direction.  The inverse source model was given by Section \ref{Model inverse source fixed observation} and we in particular consider the following different parameters that can be divided into the following four types:
\begin{enumerate}
\item Constant $q$. This can be obtained by a constant source $f$ supported in a square given by $(-r,r)\times (-r,r)$ and $\hat{x}=(1,0)^T$. Here $0<r<1$. This leads to 
\begin{equation}\label{Section numerical example q type 1}
q(s)=\mathcal{R} f (\hat{x},s)=  2 r 1_{(-r,r)}(s).
\end{equation}
Note that in this case $q$ is a constant.
\item ``Increasing-decreasing'' $q$. This can be obtained by a constant source supported in a rhombus given by
$$
\{(y_1,y_2): y_1 \in (-r,r), y_2 \in (|y_1|-r,r-|y_1|) \}
$$
 and $\hat{x}=(1,0)^T$ (which is equivalent to a constant source $f$ supported in a square but with a $\pi/4$-rotated observation direction). Here $0<r<1$. This leads to 
\begin{equation}\label{Section numerical example q type 2}
q(s)=\mathcal{R} f (\hat{x},s)=  (2 r-2|s|) 1_{(-r,r)}(s).
\end{equation}
Note that in this case $q$ is increasing in $(-r,0)$ and decreasing in $(0,r)$.
\item ``Decreasing-increasing'' $q$. This can be obtained by a constant source $f$ supported in ``M'' given by
$$
\{(y_1,y_2): y_1 \in (-r,r), y_2 \in (0, 0.5r+0.5|y_1|) \}
$$
 and $\hat{x}=(1,0)^T$ (which is equivalent to a source supported in a square but with non-constant intensities). Here $0<r<1$. This leads to 
\begin{equation}\label{Section numerical example q type 3}
q(s)=\mathcal{R} f (\hat{x},s)=  \big(0.5r+0.5|s| \big) 1_{(-r,r)}(s).
\end{equation}
Note that in this case $q$ is decreasing in $(-r,0)$ and increasing in $(0,r)$.
\item 
Oscillatory $q$. This can be obtained by a constant source $f$ supported in an oscillatory waveguide given by
$$
\{(y_1,y_2): y_1 \in (-r,r), y_2 \in (0, 0.5r - 0.25r \cos(m\pi y_1/r)) \}
$$
 and $\hat{x}=(1,0)^T$ (which is equivalent to a source supported in a square but with oscillatory intensities). Here $m$ is a   positive integer that introduce the oscillatory nature, $0<r<1$. This leads to 
\begin{equation}\label{Section numerical example q type 4}
q(s)=\mathcal{R} f (\hat{x},s) = \big(0.5r - 0.25r \cos(m\pi s/r))\big) 1_{(-r,r)}(s).
\end{equation}
Note that in this case $q$ is   oscillatory in $(-r,r)$.
\end{enumerate}

    \begin{figure}[tbhp]
      \centering
\includegraphics[width=0.3\linewidth]{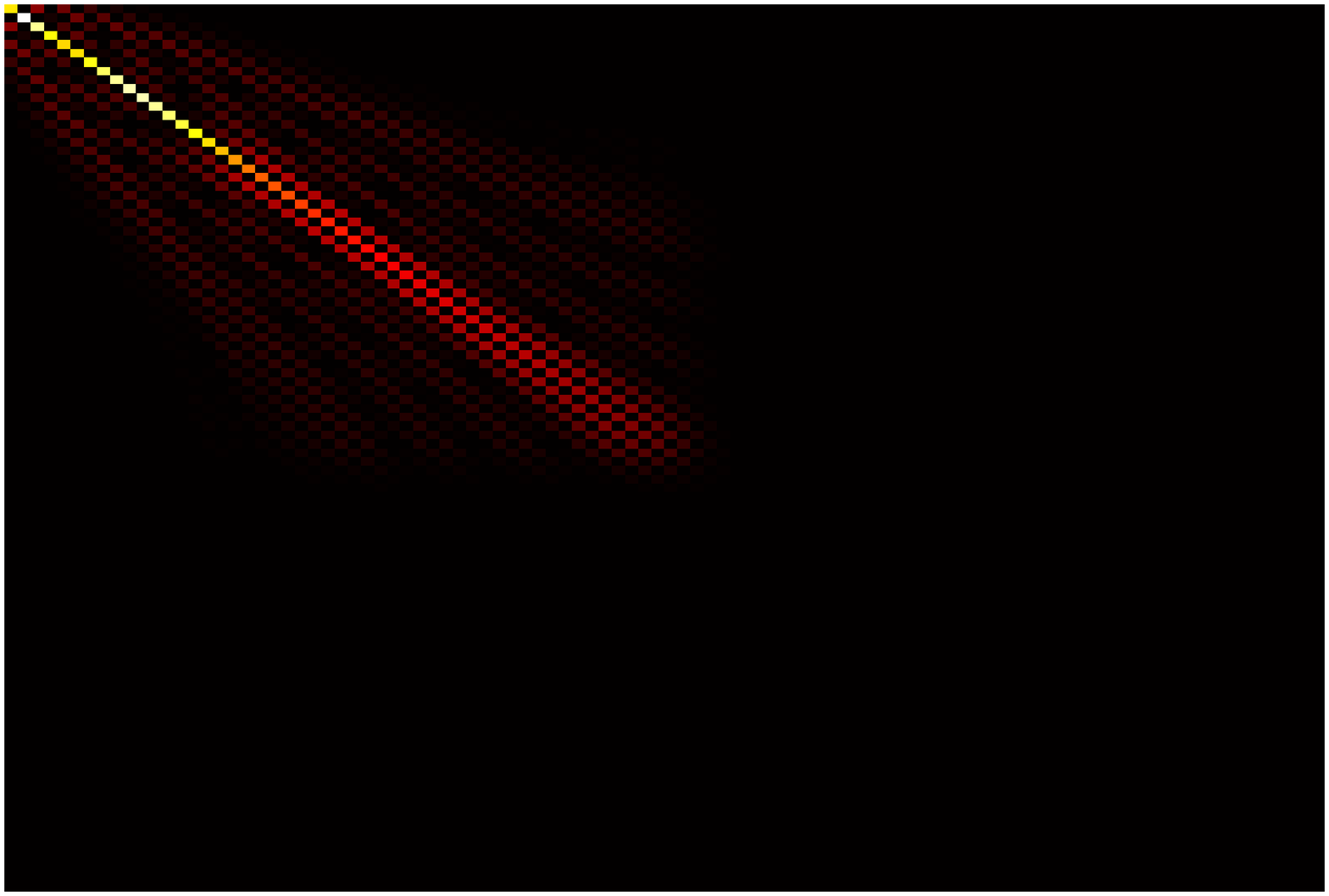}
\includegraphics[width=0.3\linewidth]{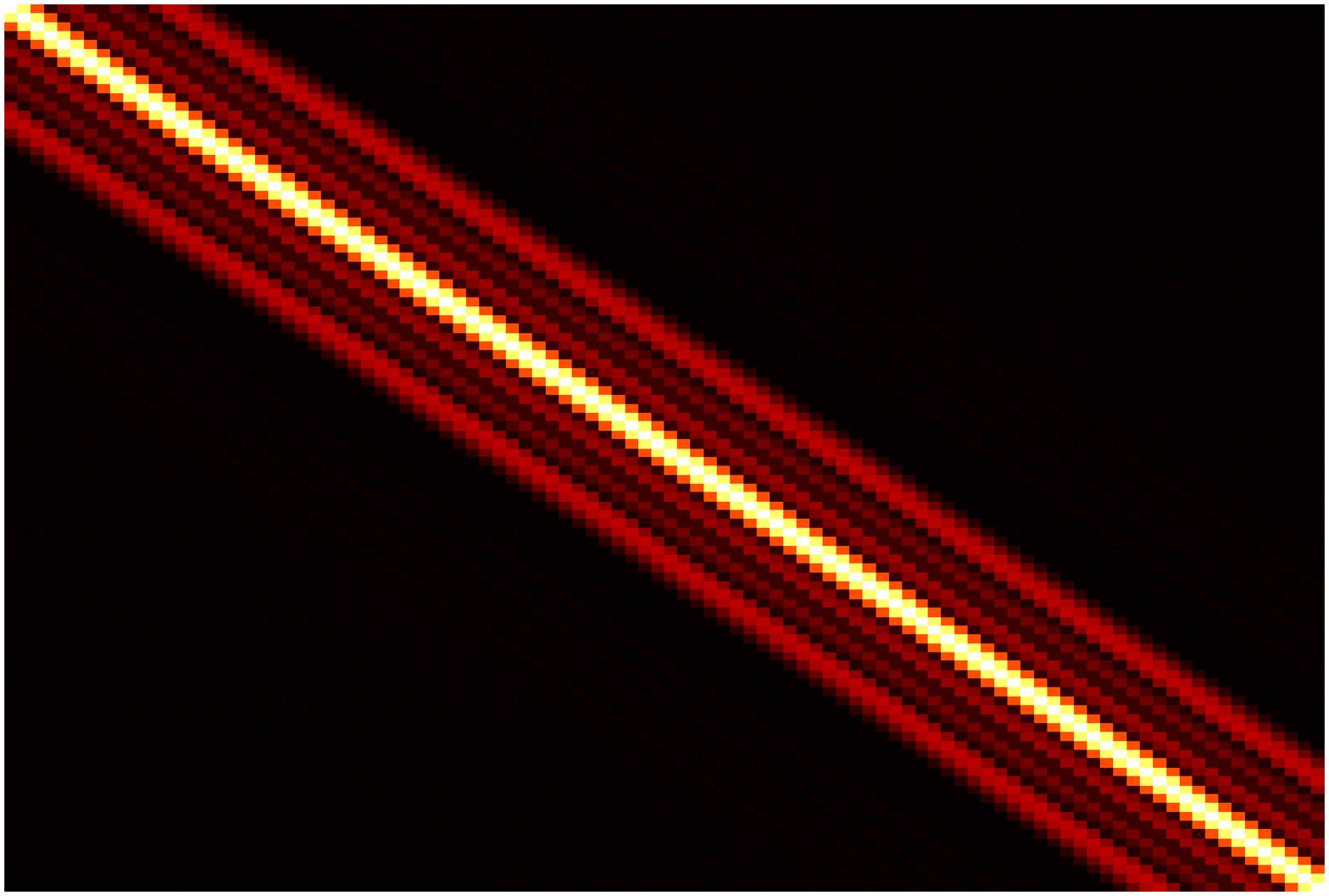}\\

     \caption{
     \linespread{1}
$\|A^{\mathbbm{J}}_{i,j}\| $ on the left and $\|A^{DFT}_{i,j}\| $  on the right  in both case $1\leq i,j \leq 400$}
 \label{Figure Matrix PSWF vs DFT}
    \end{figure}

      \begin{figure}[tbhp]
      \centering
\includegraphics[width=0.4\linewidth]{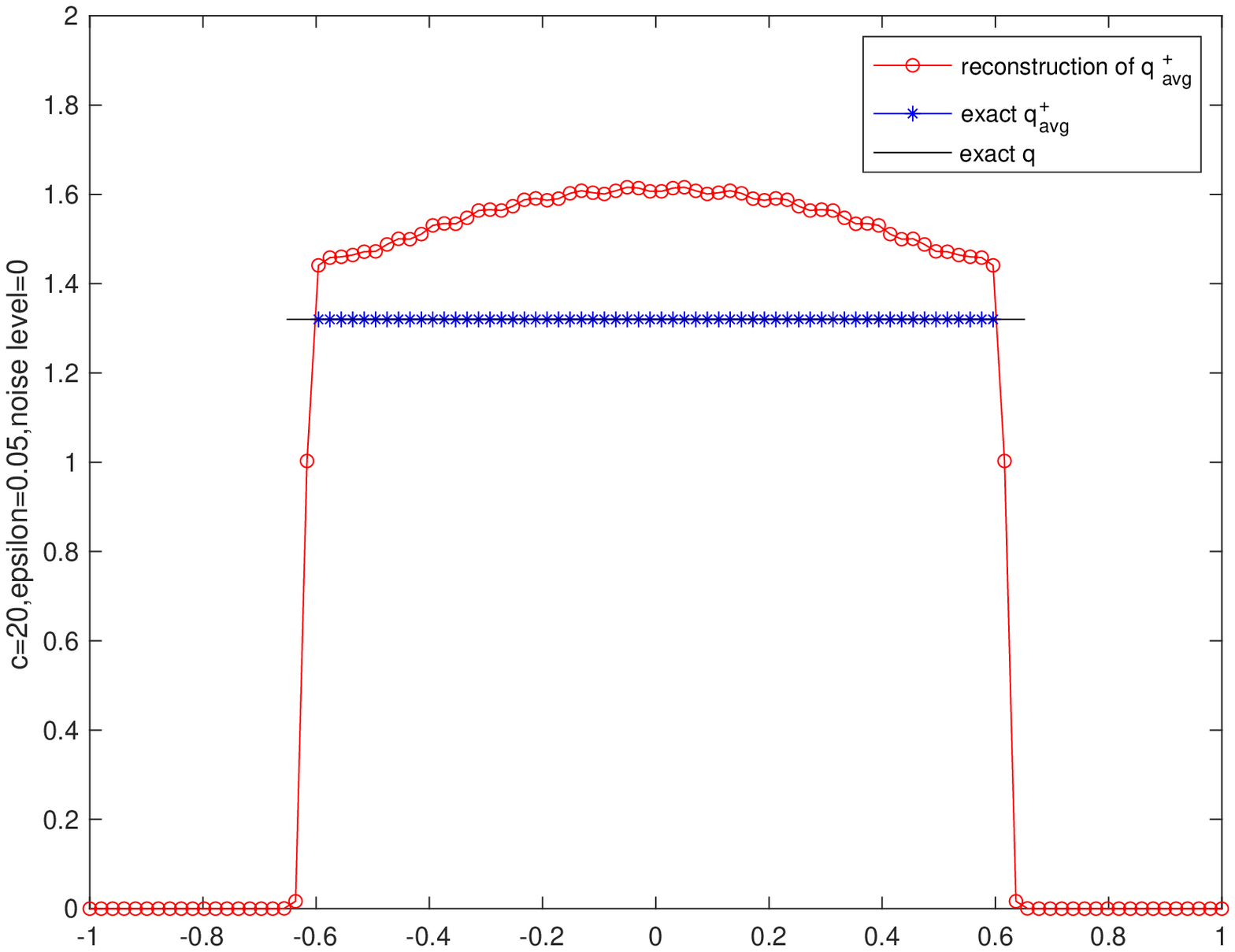}
\includegraphics[width=0.4\linewidth]{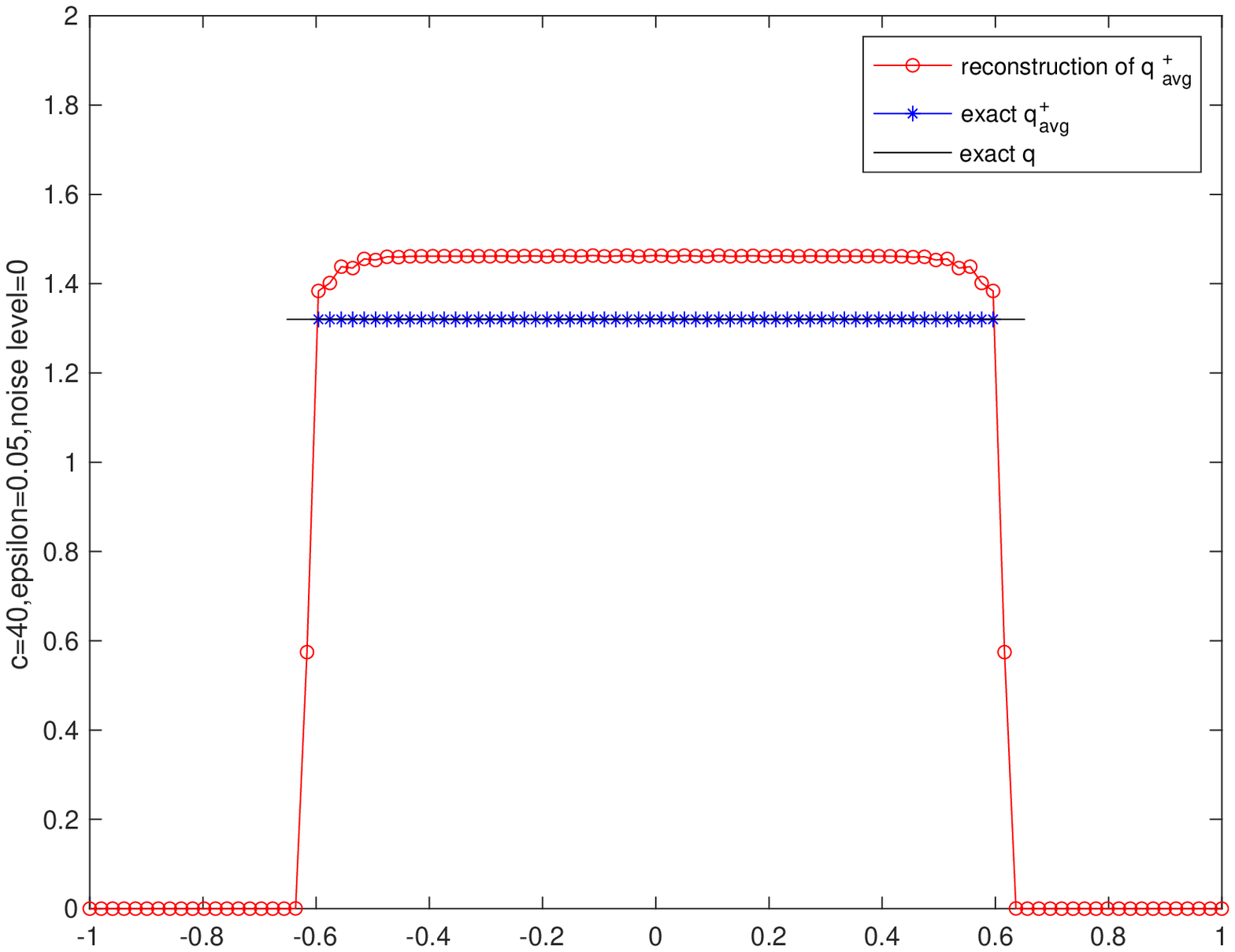}\\
\includegraphics[width=0.4\linewidth]{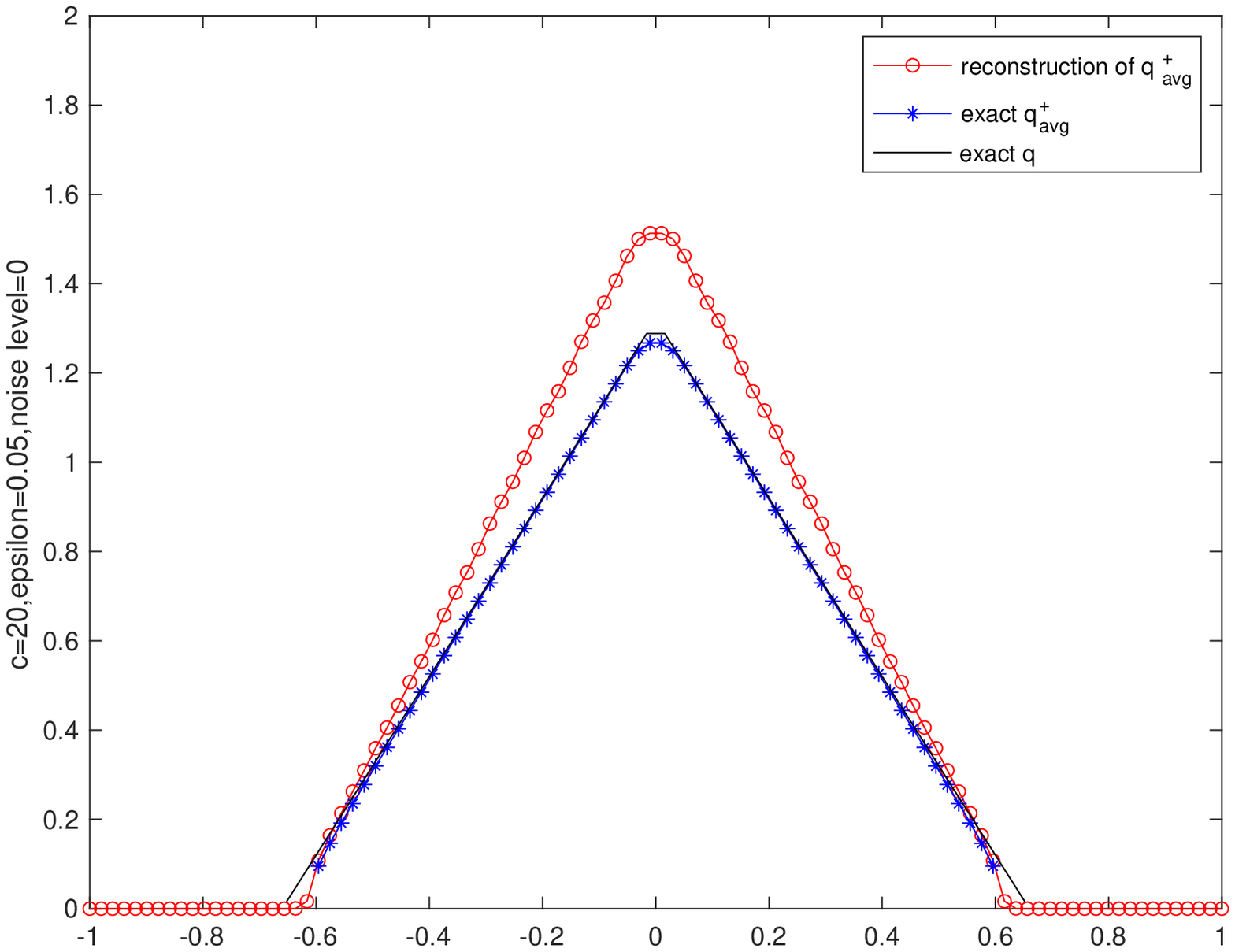}
\includegraphics[width=0.4\linewidth]{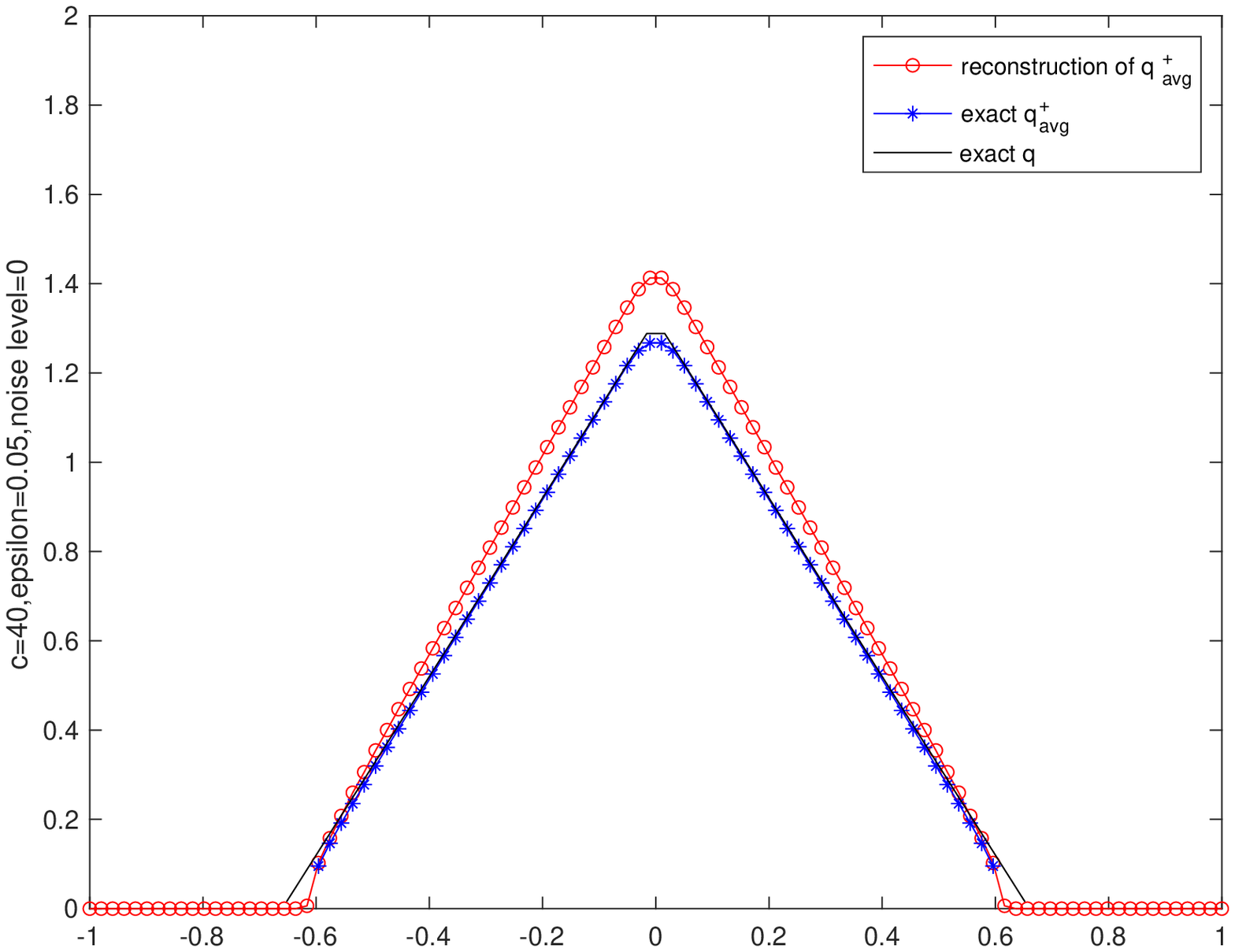}\\
\includegraphics[width=0.4\linewidth]{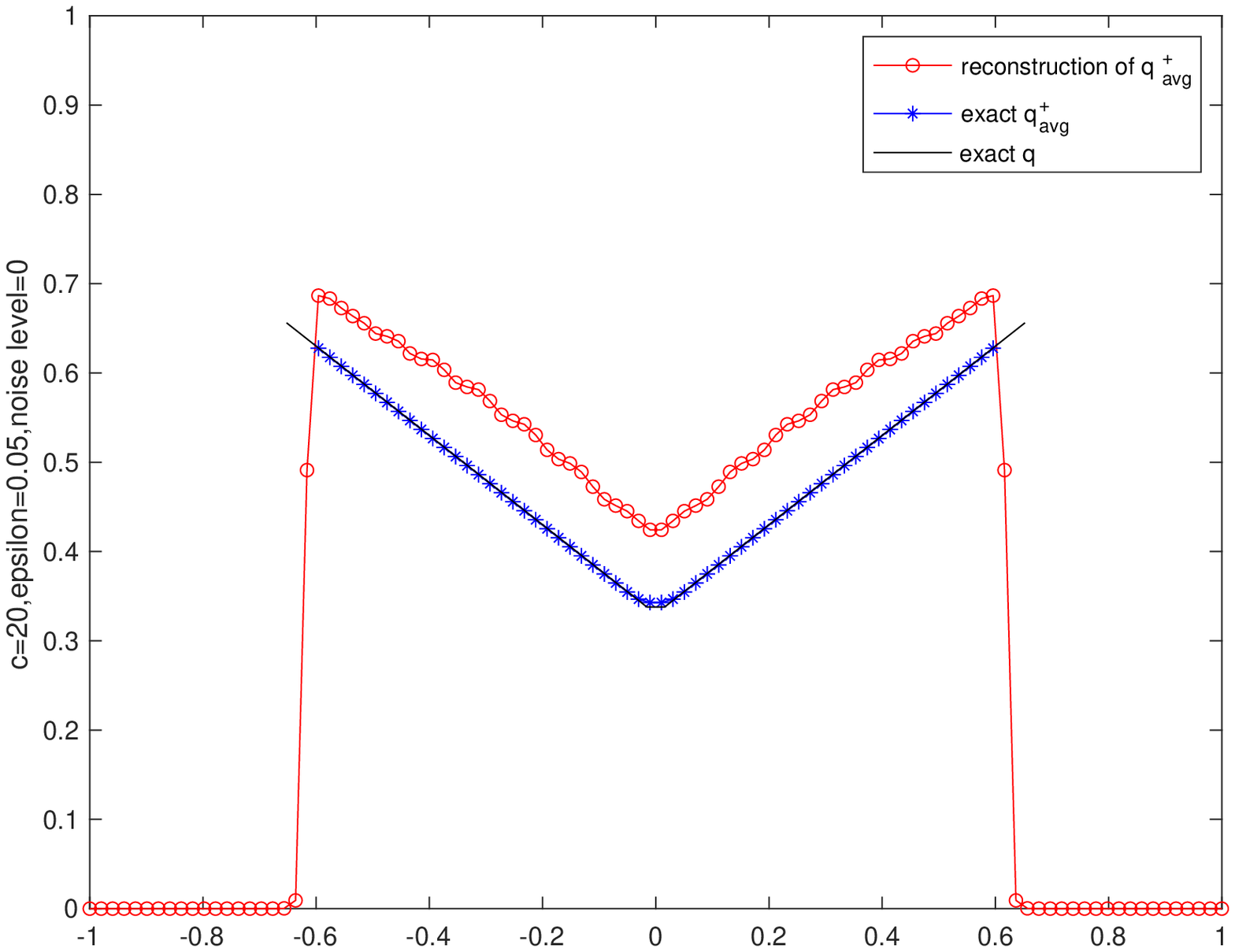}
\includegraphics[width=0.4\linewidth]{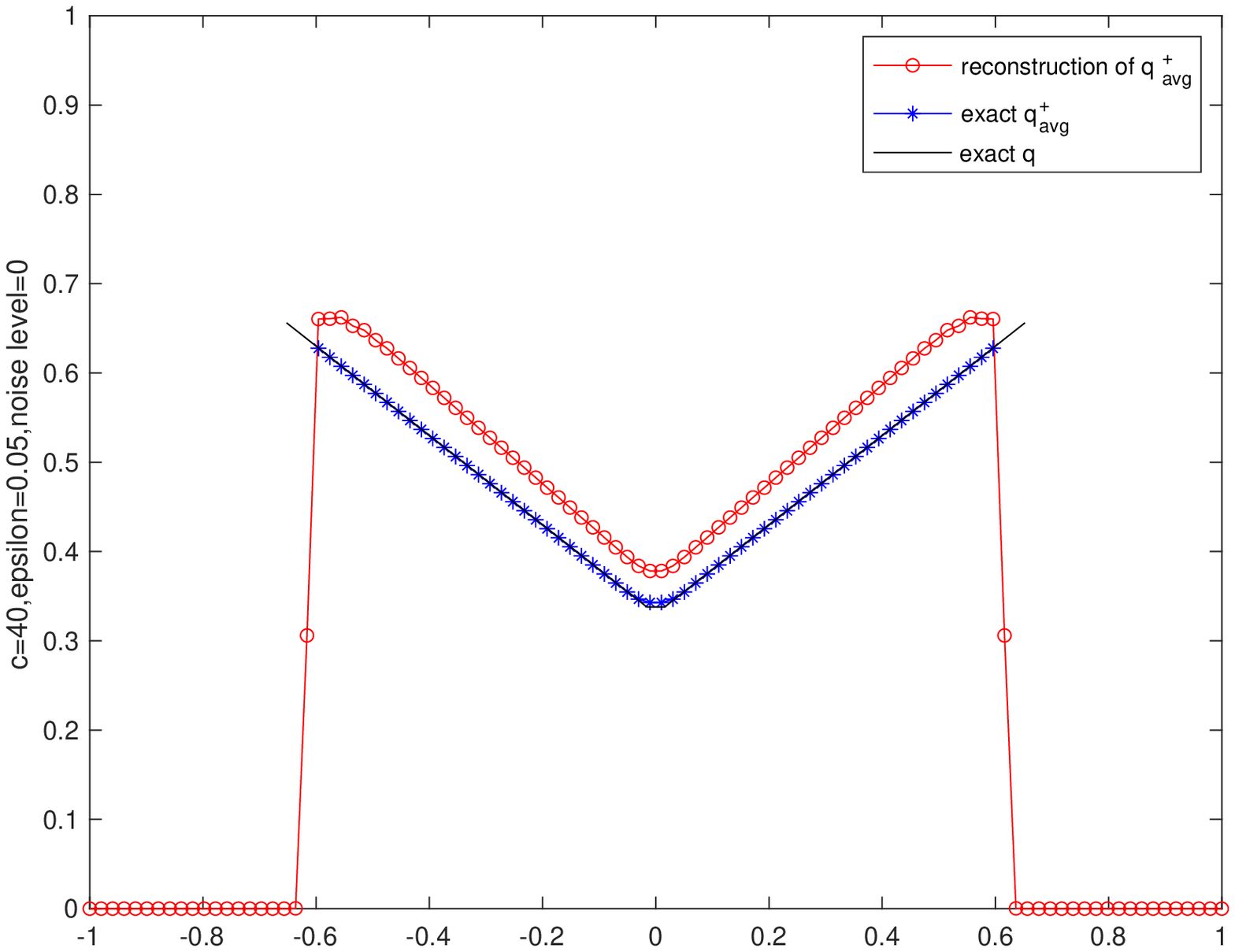}\\
\includegraphics[width=0.4\linewidth]{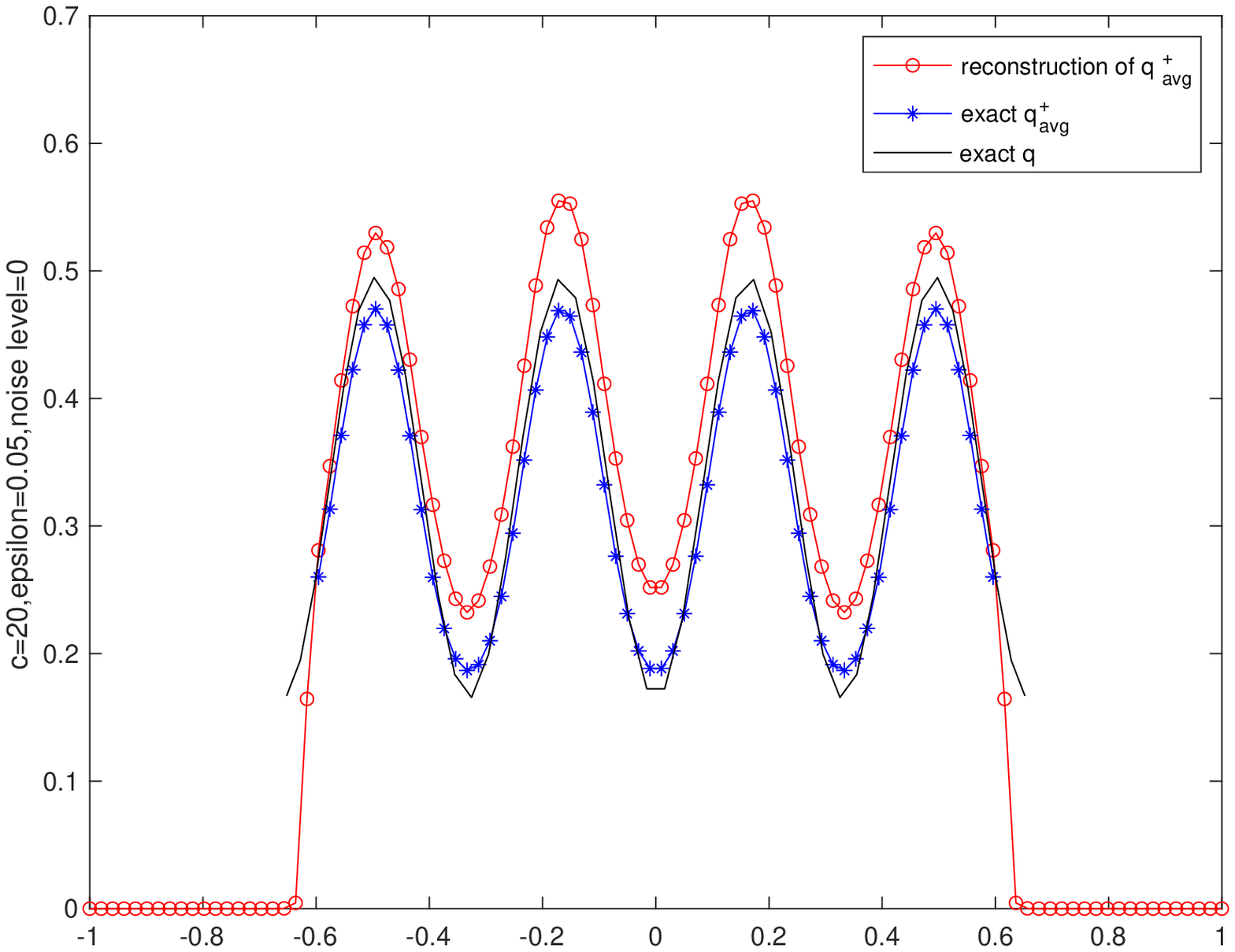}
\includegraphics[width=0.4\linewidth]{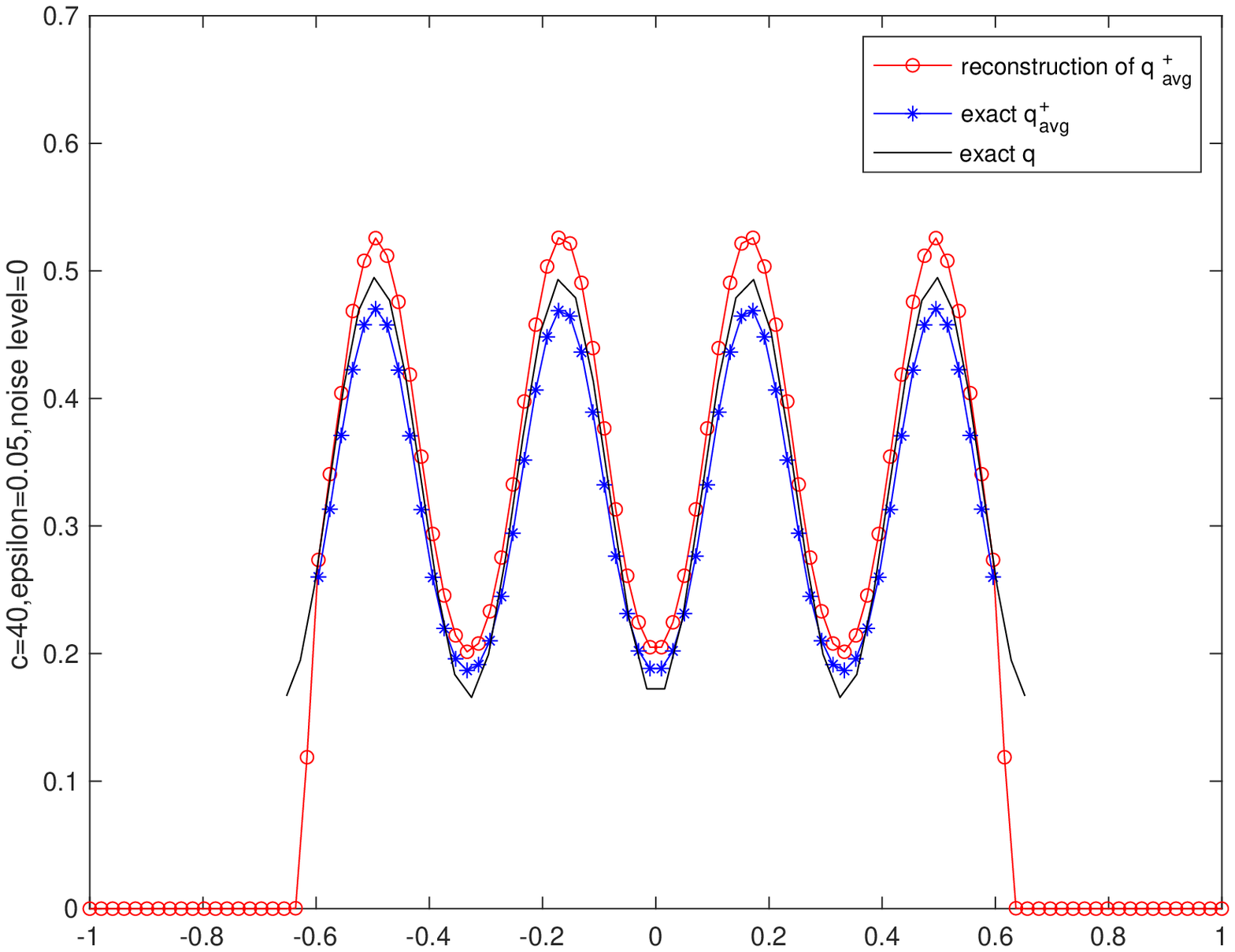}
     \caption{
     \linespread{1}
Plot of $I(z)$, $q_{avg}$, and $q$ for four different types of parameters with noiseless data and $\epsilon=5\times 10^{-2}$. Left column: $\dim(\mathbb{J})=37$. Right column: $\dim(\mathbb{J})=54$. Row $j$ corresponds to type $j$, j=1,2,3,4.
}\label{Section numerical examples subsection parameter figure noiseless}
    \end{figure}

          \begin{figure}[tbhp]
      \centering
\includegraphics[width=0.4\linewidth]{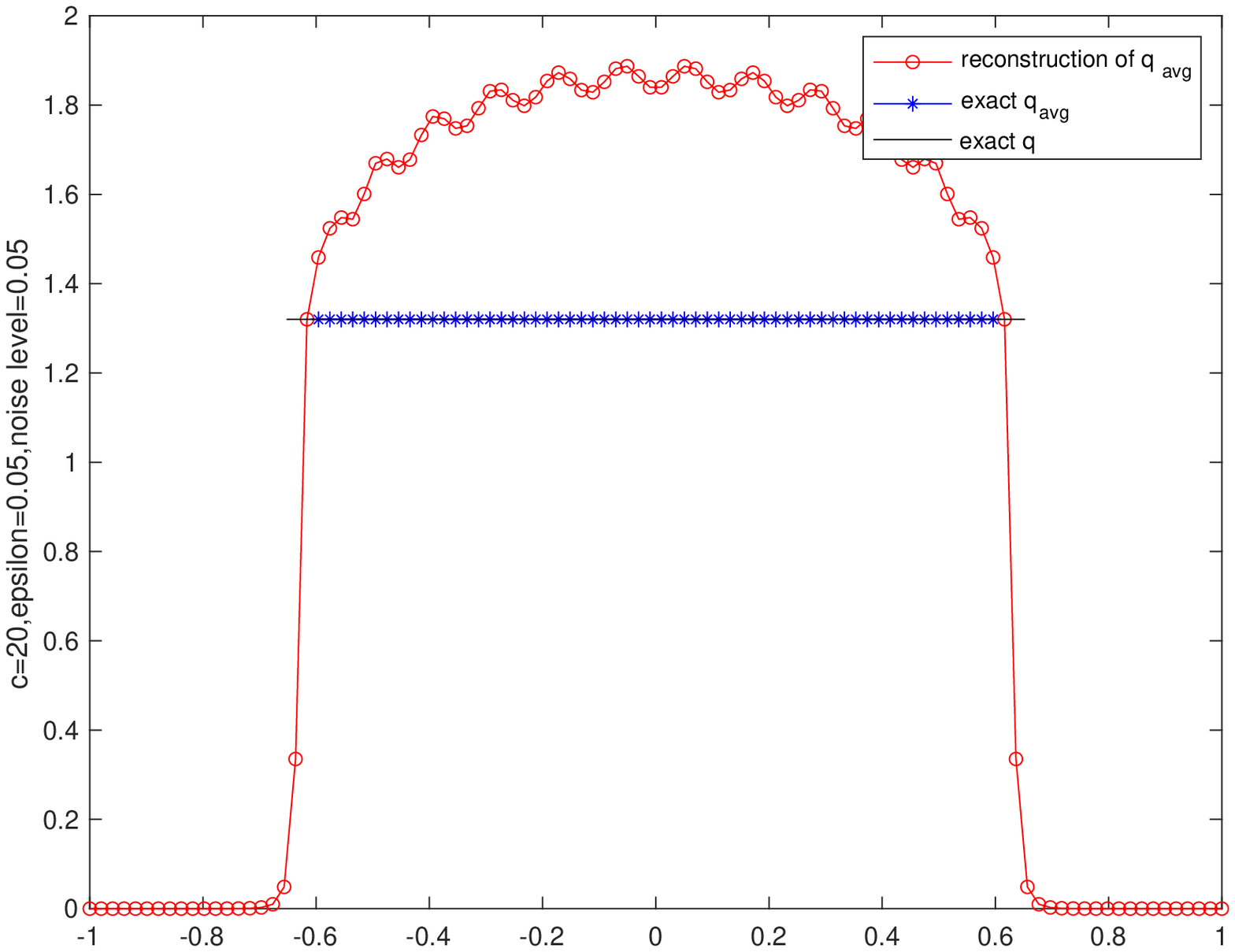}
\includegraphics[width=0.4\linewidth]{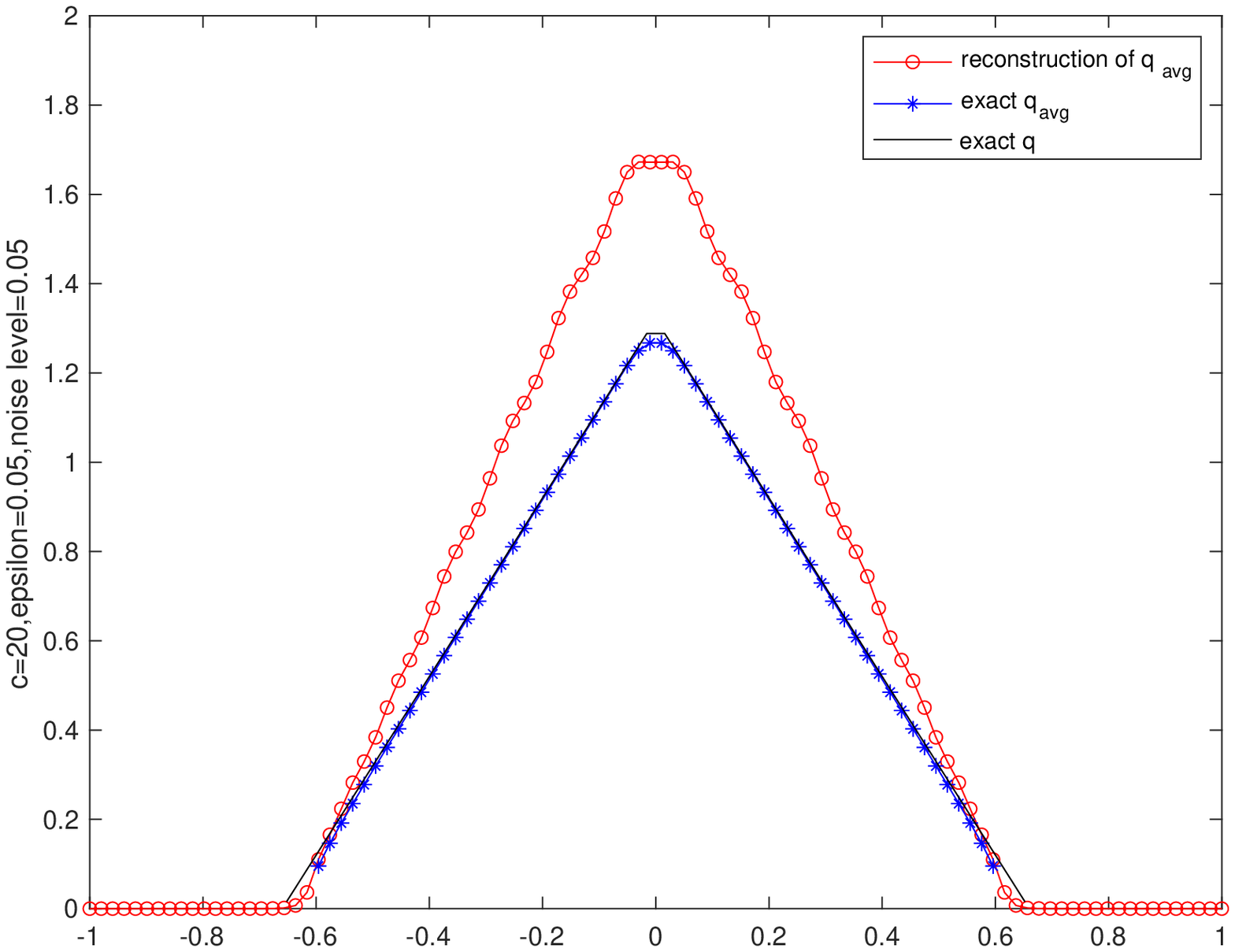}\\
\includegraphics[width=0.4\linewidth]{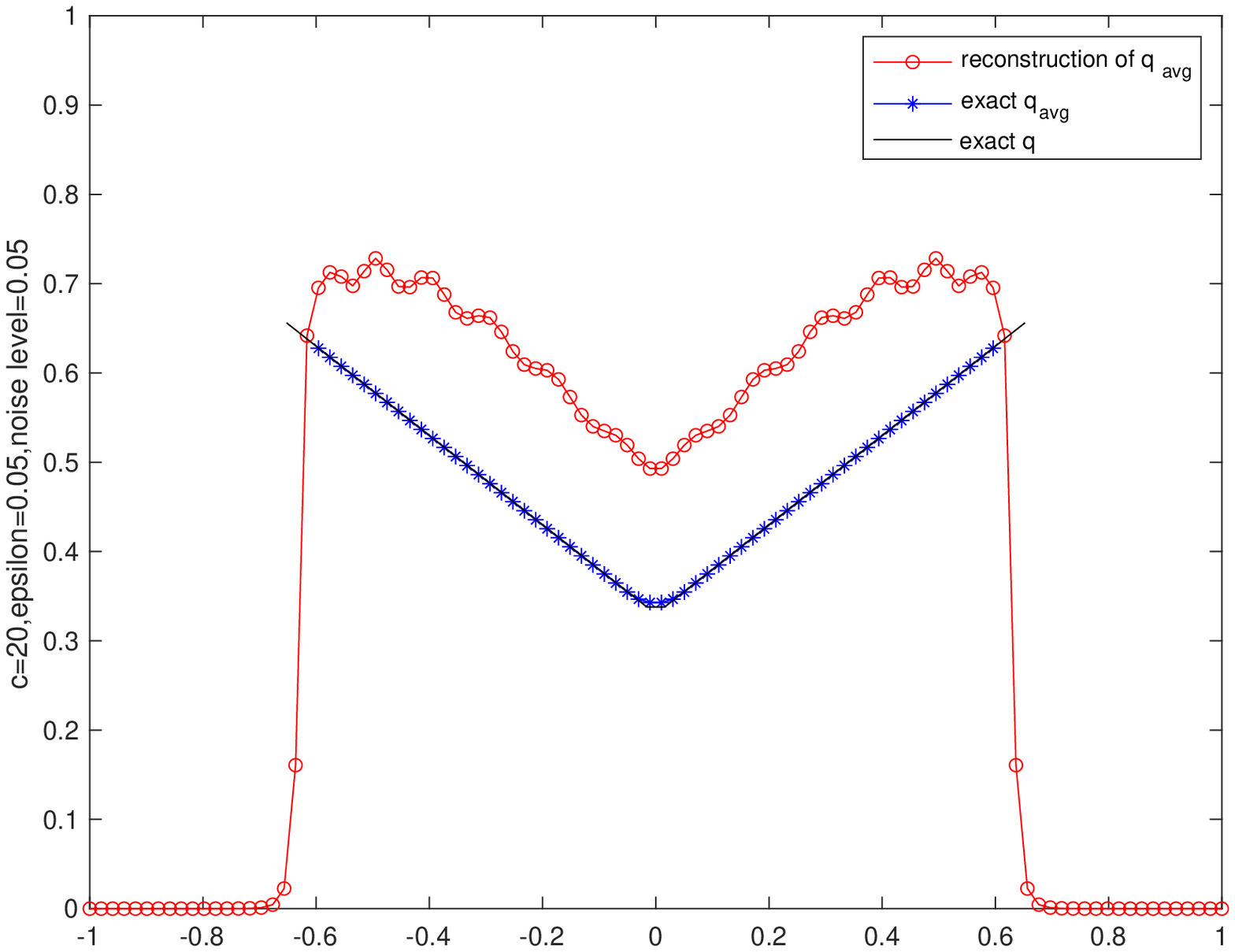}
\includegraphics[width=0.4\linewidth]{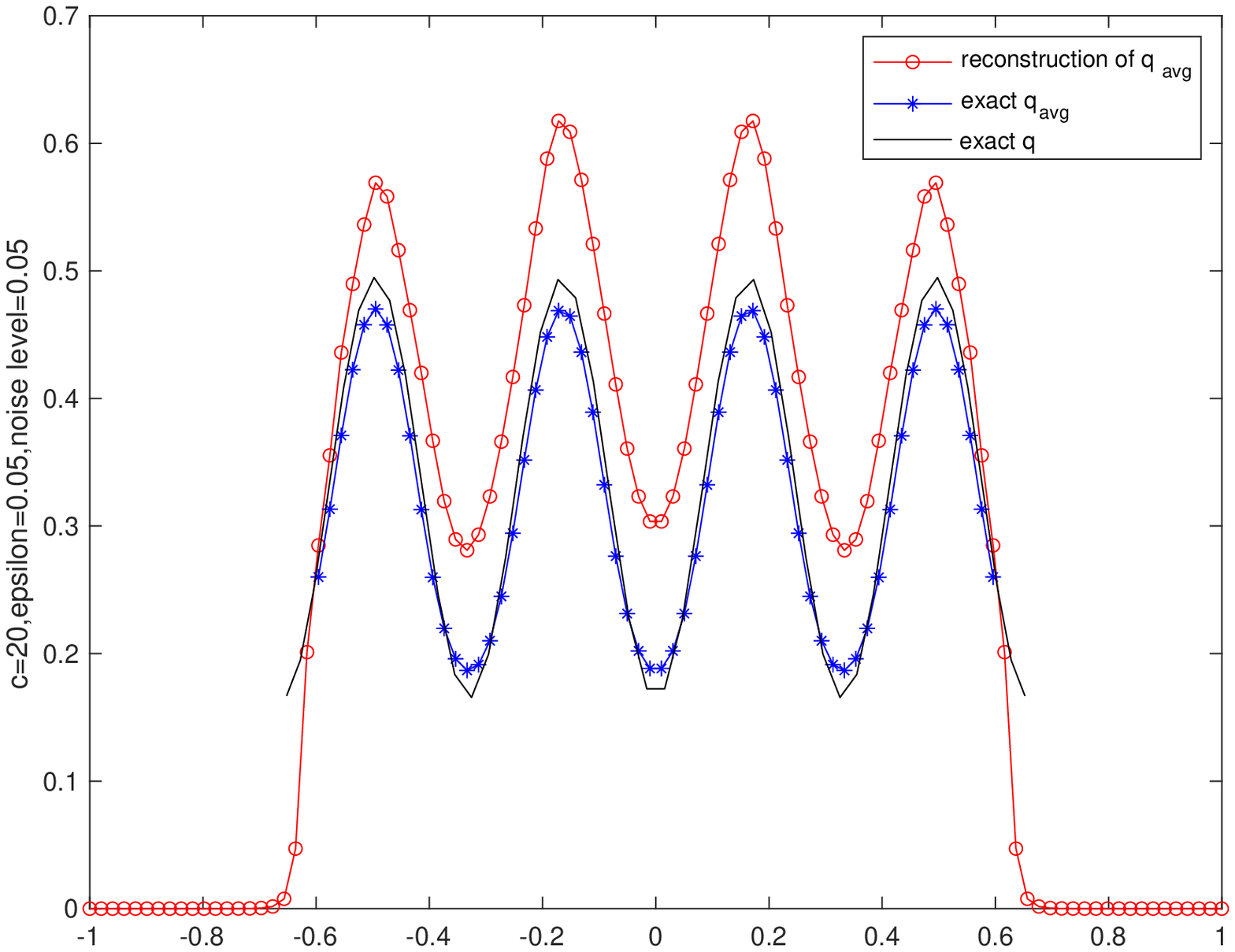}
     \caption{
     \linespread{1}
Plot of $I(z)$, $q_{avg}$, and $q$ for four different types of parameters with noisy level $5\%$. $\epsilon=5\times 10^{-2}$ and $c=20$.
} \label{Section numerical examples subsection parameter figure c20 noisy}
    \end{figure}
    
       \begin{figure}[tbhp]
      \centering
\includegraphics[width=0.4\linewidth]{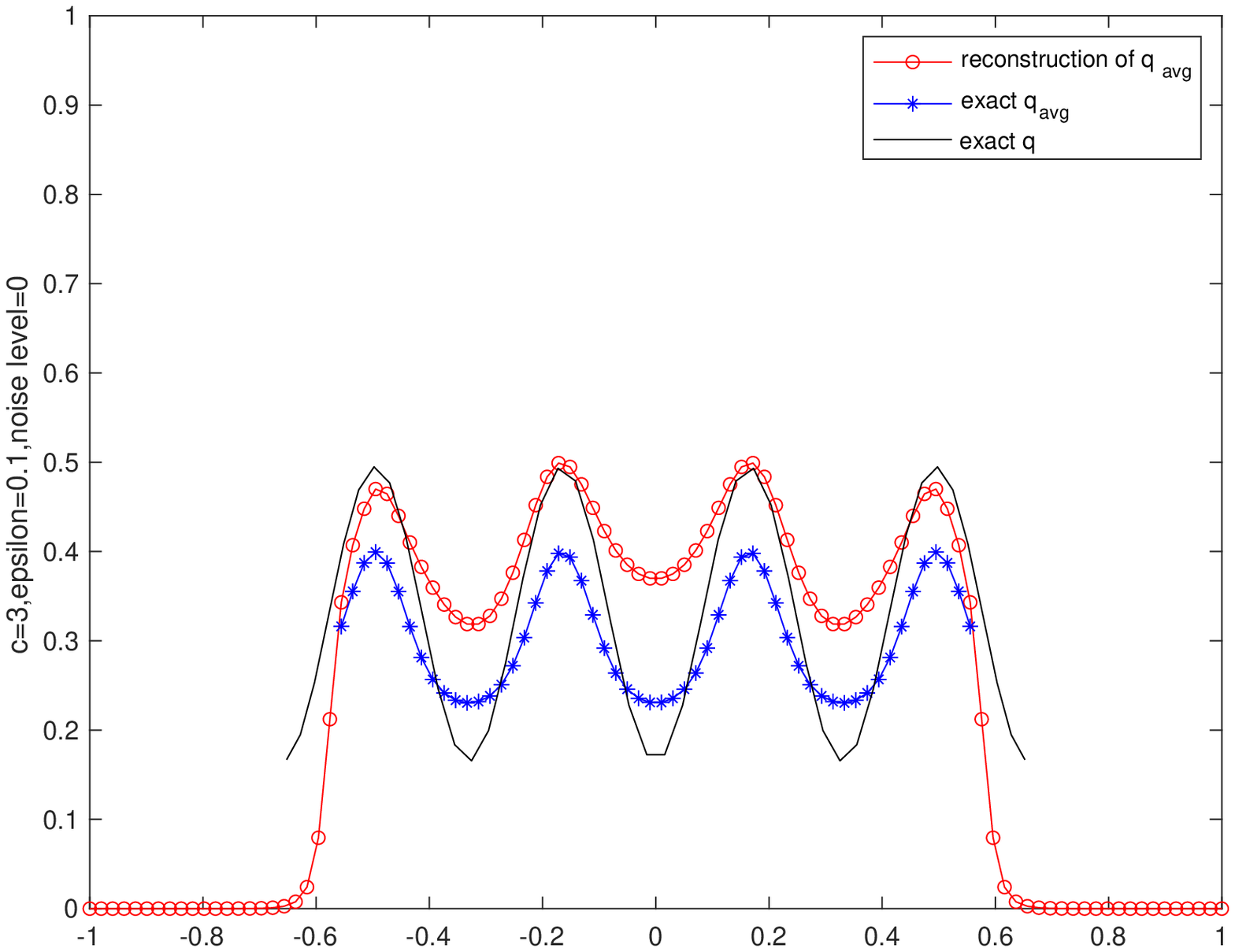}
\includegraphics[width=0.4\linewidth]{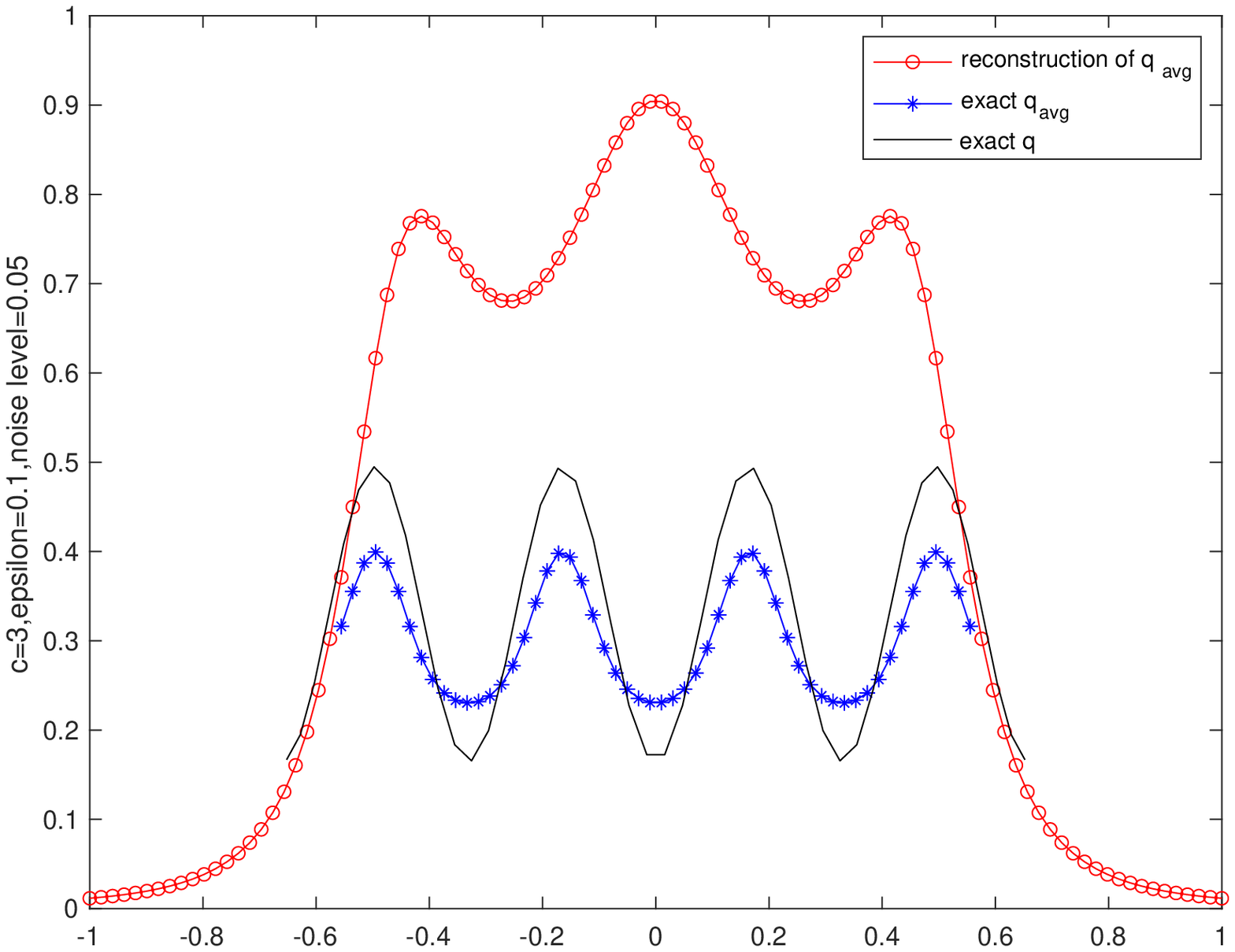}\\
\includegraphics[width=0.4\linewidth]{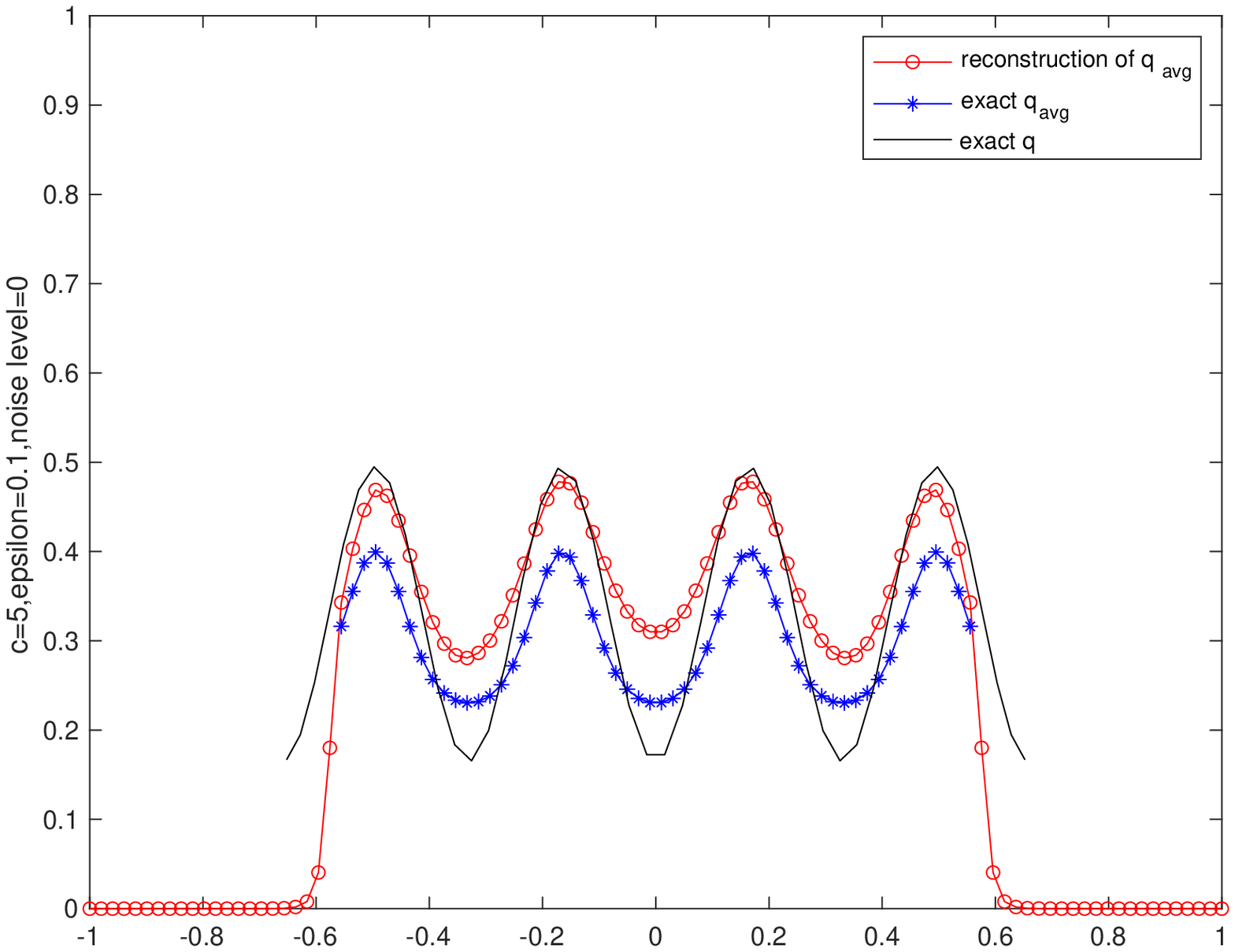}
\includegraphics[width=0.4\linewidth]{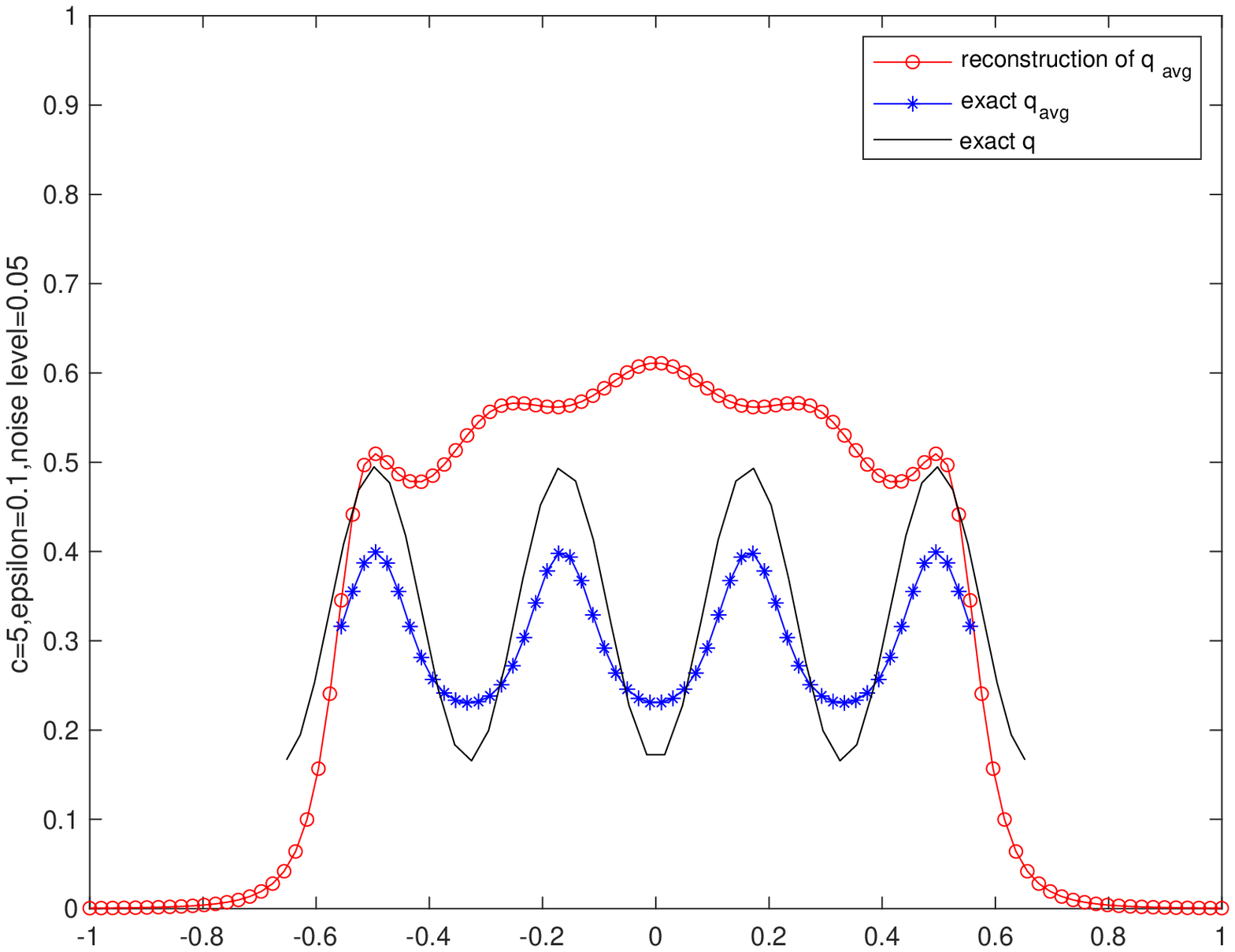}\\
\includegraphics[width=0.4\linewidth]{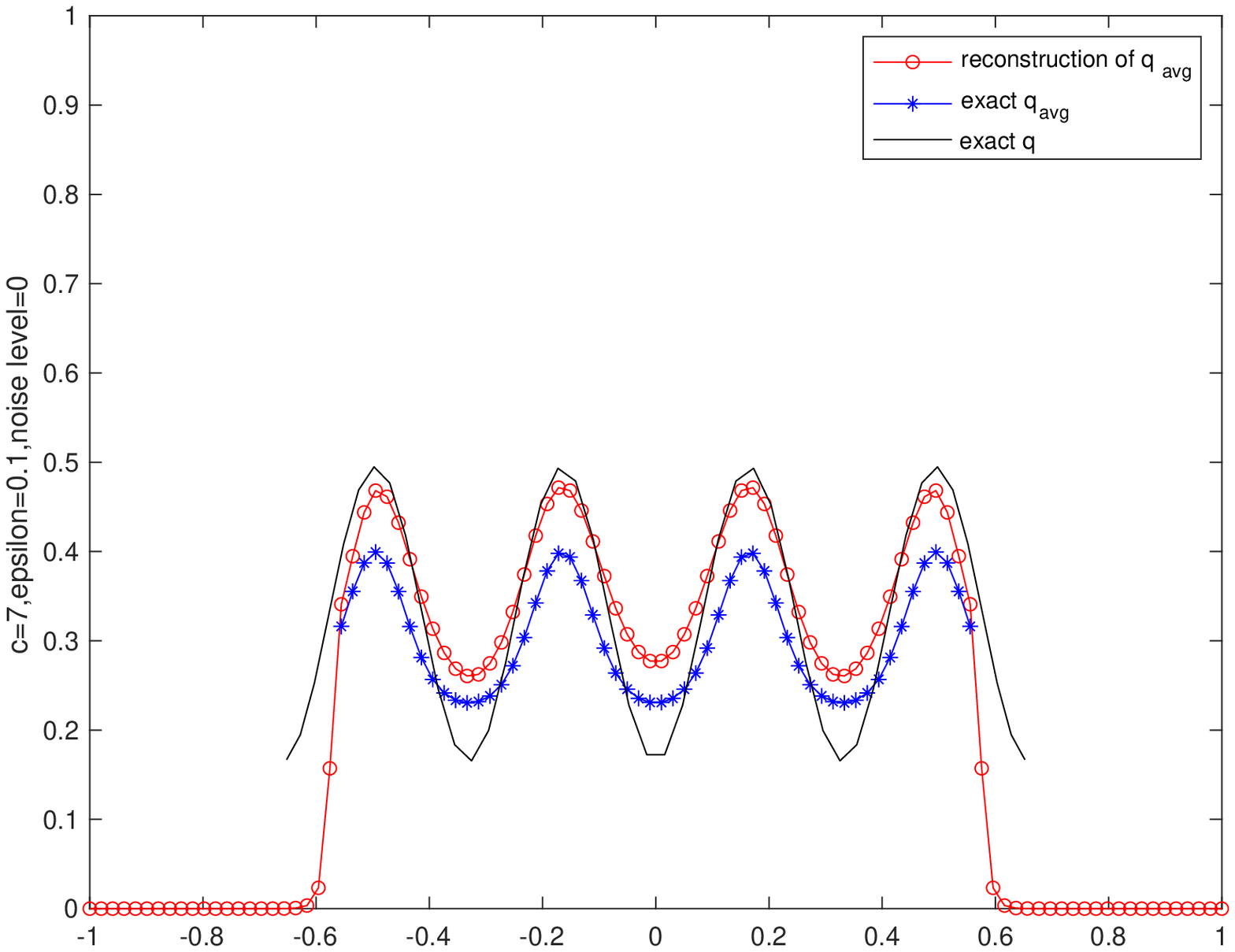}
\includegraphics[width=0.4\linewidth]{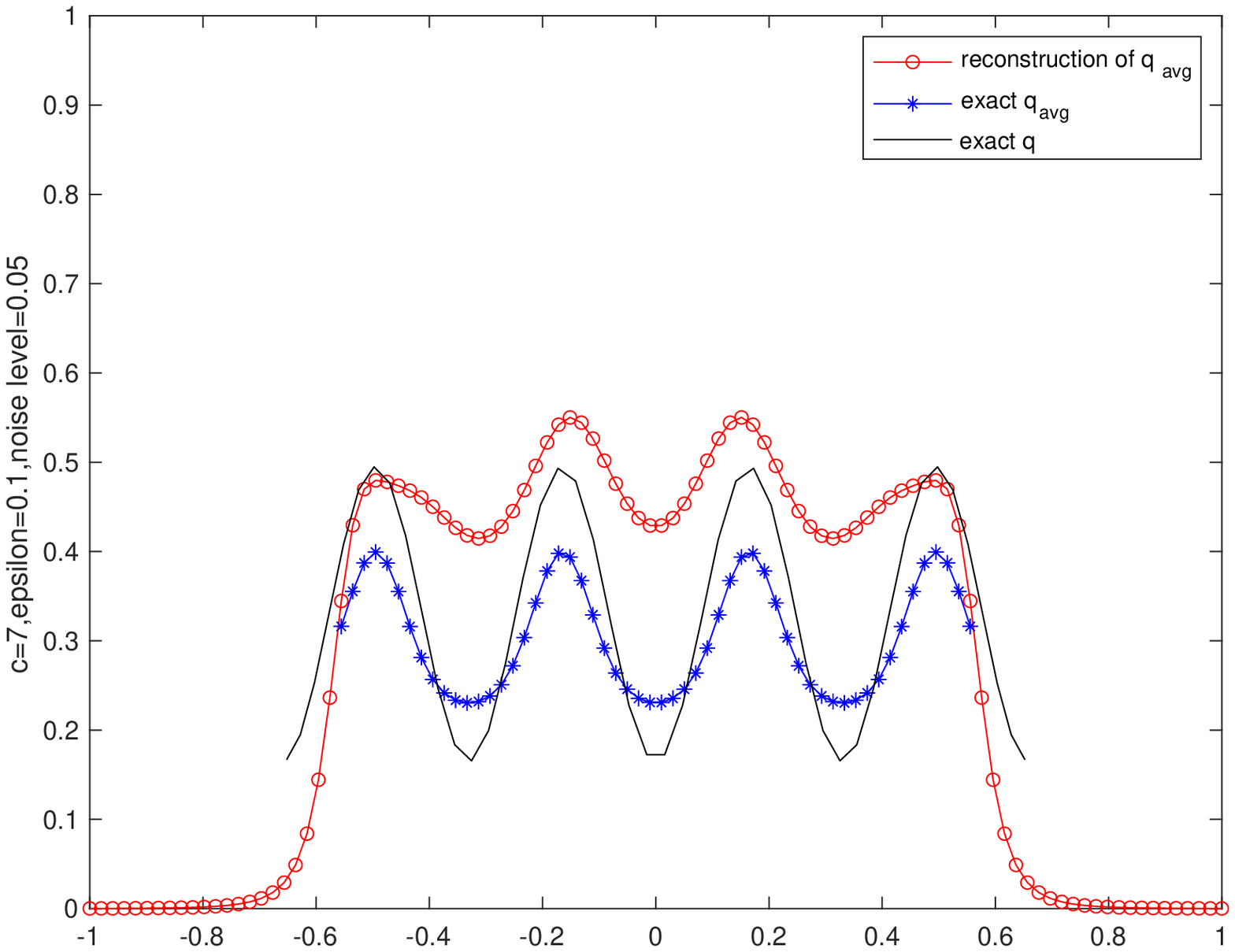}\\
\includegraphics[width=0.4\linewidth]{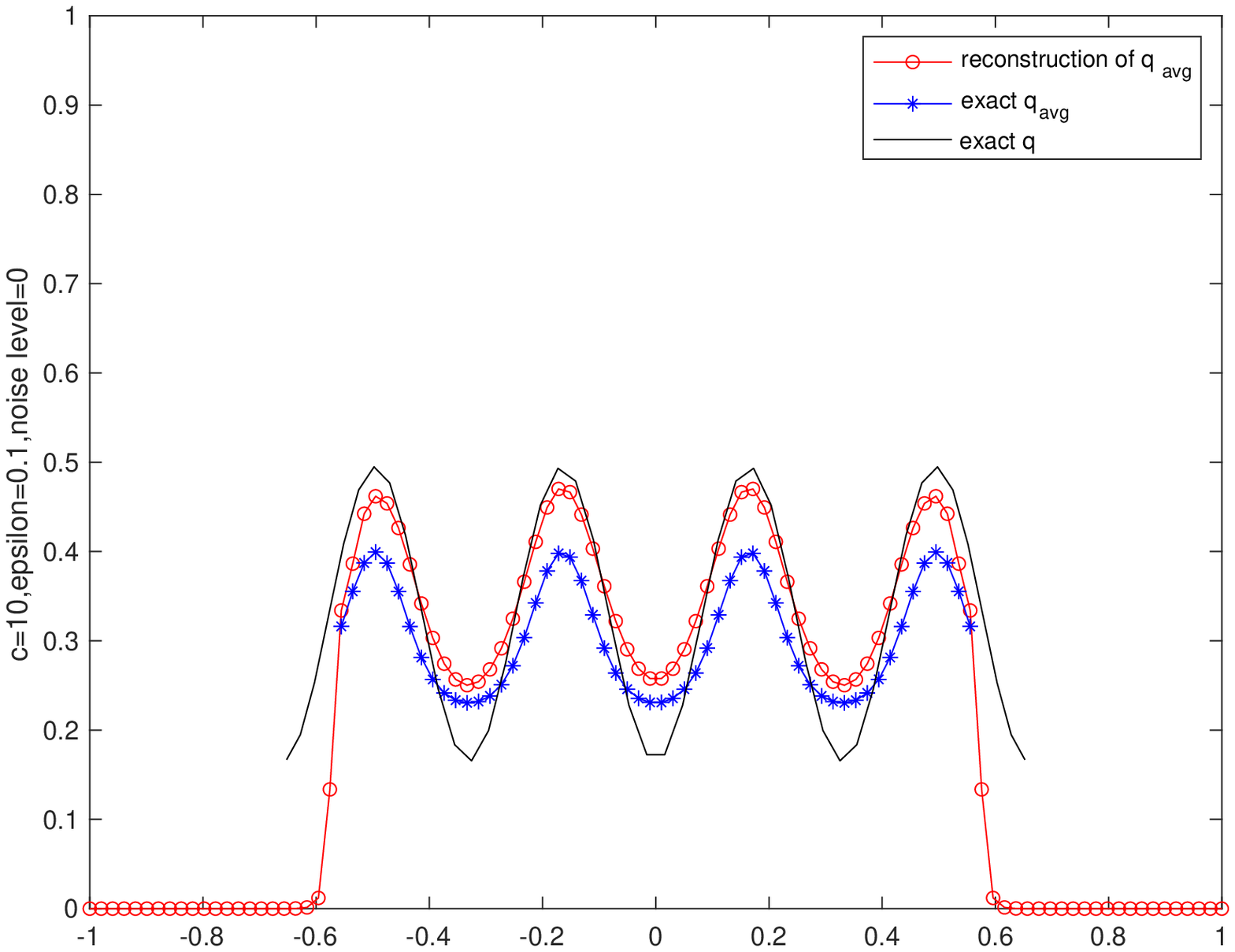}
\includegraphics[width=0.4\linewidth]{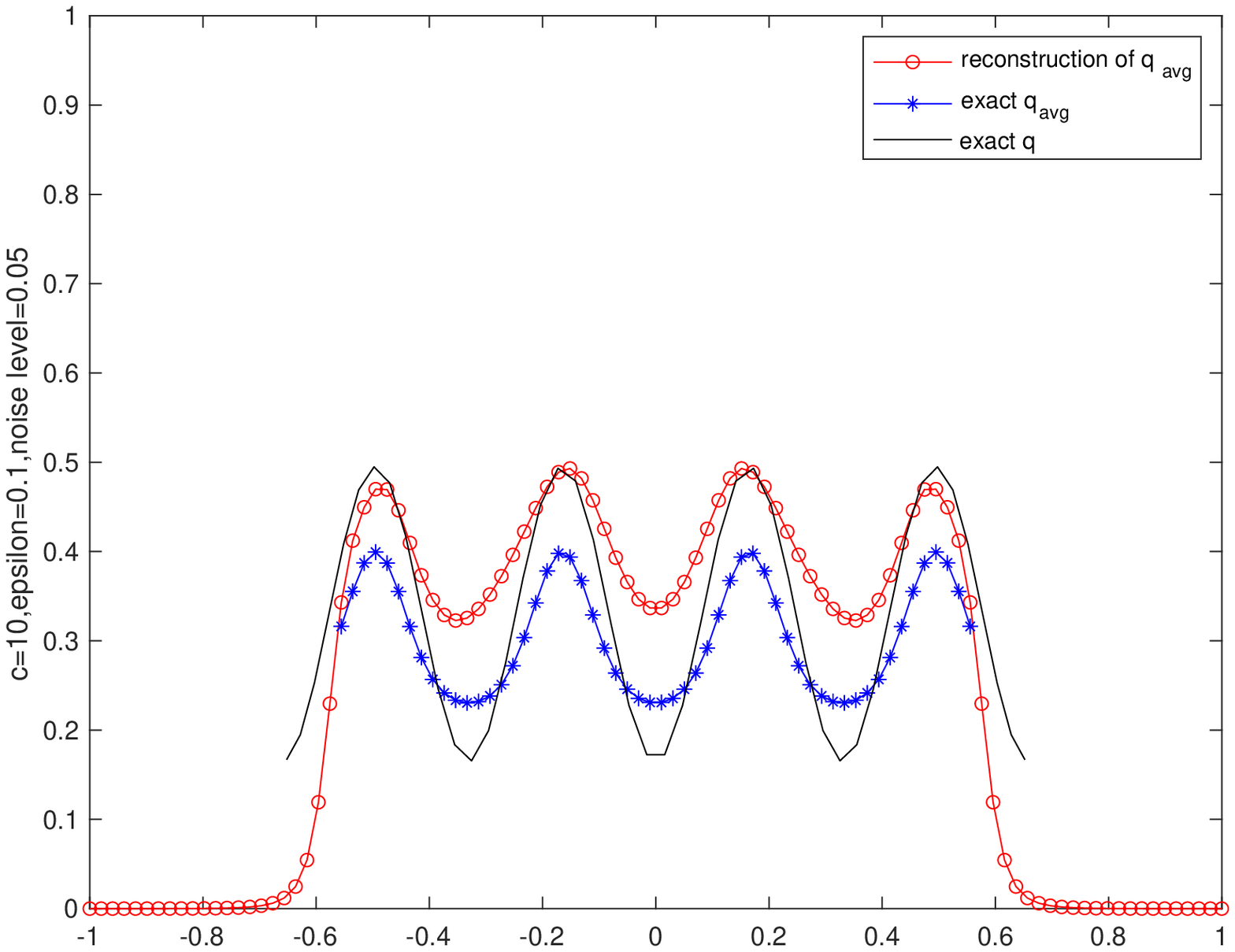}
     \caption{
     \linespread{1}
Plot of $I(z)$, $q_{avg}$, and $q$ for noiseless data (left column) and noisy data with noise level $\delta = 5\%$ (right column). $\epsilon=1\times 10^{-1}$.  From  top to  bottom $c=3,5,7,10$.
} \label{Section numerical examples subsection parameter figure c 3 5 7 10 noiseless noisy}
    \end{figure}
    
In the following we give  details of the implementation of the  prolate-Galerkin formulation of the linear sampling.
 The PSWFs are computed as detailed in Section \ref{Section PSWFs computation} where we use the Matlab code developed in \cite{boyd05code}. The prolate eigenvalues are then computed using the formulas  \eref{Section PSWFs computation lambda_n computation even} -- \eref{Section PSWFs computation lambda_n computation odd} by adding a simple Matlab script to the existing code of \cite{boyd05code}. The exact data are calculated analytically by hand  using the four different $q$ given by \eref{Section numerical example q type 1}--\eref{Section numerical example q type 4}. Given a noisy operator $\mathcal{N}^\delta$, we obtain a noisy data matrix $A^{\mathbbm{J},\delta}$ according to \eref{Section reduced LSM matrix A def} so that
\begin{equation*}  
A^{\mathbbm{J},\delta}_{j\ell} = \langle \mathcal{N}^{\delta} \psi_j(\cdot,;c),  \psi_\ell(\cdot,;c) \rangle, \quad  j,\ell \in \mathbbm{J},
\end{equation*}
where the noisy data are given by adding Gaussian noise to the exact data which introduce $\delta$ noise level such that 
$
\| A^{\mathbbm{J},\delta} - A^{\mathbbm{J}}\|_2  =\delta$.
We integrate the product of the data and the PSWFs using a Legendre-Gauss-Lobatto (LGL) quadrature rule.  The right hand side is given by \eref{Section reduced LSM matrix A def} and \eref{Section FM and a LSM phi_z PSWF expansion} where
 $$
 \phi^{\mathbbm{J}}_{z\ell}=  \lambda_\ell(c) \langle E_z,  \psi_\ell(\cdot,;c) \rangle_{B(z,\epsilon)}, \quad  \ell \in \mathbbm{J},
 $$
 here we approximate the integral over $R(z,\epsilon)$ with $R(z,\epsilon) = B(z,\epsilon)=(z-\epsilon,z+\epsilon)$ using again the Legendre-Gauss-Lobatto (LGL) quadrature rule in this interval $(z-\epsilon,z+\epsilon)$. 
Having a regularized solution $\{g^{{\mathbbm{J}},\delta}_{z,\alpha,j}\}_{j \in {\mathbbm{J}}}$ computed from the noisy linear system (as is similar to the noiseless case \eref{Section reduced LSM linear system})
\begin{equation*} 
\sum_{j\in\mathbbm{J}} A^{\mathbbm{J},\delta}_{j\ell} g^{\mathbbm{J},\delta}_{z,\alpha,j} \sim \phi^{\mathbbm{J}}_{z\ell}, \quad \ell \in \mathbbm{J},
\end{equation*}
we then proceed by
$$
\mathcal{S} g^{{\mathbbm{J}},\delta}_{z,\alpha} = \sum_{j\in\mathbbm{J}} \overline{\lambda_j(c)} g^{{\mathbbm{J}},\delta}_{z,\alpha,j} \psi_j(\cdot;c) \mbox{ so that } \langle \mathcal{S} g^{{\mathbbm{J}},\delta}_{z,\alpha}, 1_{B(z,\epsilon)}  \rangle = \sum_{j\in\mathbbm{J}} \overline{\lambda_j(c)} g^{{\mathbbm{J}},\delta}_{z,\alpha,j} \langle\psi_j(\cdot;c), 1_{B(z,\epsilon)}  \rangle_{_{B(z,\epsilon)}}
$$
 and approximate each integral $\langle\psi_j(\cdot;c), 1_{B(z,\epsilon)}  \rangle_{_{B(z,\epsilon)}}$  over $(z-\epsilon,z+\epsilon)$ using again the Legendre-Gauss-Lobatto (LGL) quadrature rule. We further set the indicator function by
$$
I(z):=\big(\langle \mathcal{S} g^{{\mathbbm{J}},\delta}_{z,\alpha}, 1_{B(z,\epsilon)}  \rangle \big)^{-1},
$$
which is an harmonic mean, as expected the numerical examples confirm the well known fact that harmonic mean are larger than the classical mean.
We chose in this paper the spectral cut off regularization where we chose $\mathbbm{J}$ such that all the corresponding prolate eigenvalues (with indexes in $\mathbbm{J}$) are larger that the noise level $\delta$.

First we illustrate in figure \ref{Figure Matrix PSWF vs DFT} the fact that PSWFs compressed the data operator where we compare the operator expressed in PSWFs basis which is the matrix $A^{\mathbbm{J}}$ with the operator $\mathcal{N}$ discretize on a cartesian grid and computed using Discrete (fast) Fourrier Transform which is the matrix  $A^{DFT}$.

Motivated by the particular example in Section \ref{Section numerics subsection explicit example},
it is likely that we need a large index set ${\mathbb{J}}$ to achieve a sufficiently good approximation to the unknown. This first motivates us to perform  a set of numerical examples with noiseless data where a large index set ${\mathbb{J}}$  can be available  (the  index set ${\mathbb{J}}$  cannot be too large  since this means that we have to compute many   prolate eigenvalues which is computationally challenging due to the super fast exponential decay of the prolate eigenvalues; see also the particular example studied in Section \ref{Section numerics subsection explicit example}).

In figure \ref{Section numerical examples subsection parameter figure noiseless}, we plot   the indicator function $ I(z) = [\langle \mathcal{S} g^{ {\mathbbm{J}},\delta}_{z,\alpha}, 1_{B(z,\epsilon)}  \rangle]^{-1}$ with noiseless data, $\epsilon=0.05$, and two different $c$. Here $c=20$ (left column) allows us to have an index set ${\mathbbm{J}}$ of dimension $37$   and $c=40$ (right column) allows us to have an index set ${\mathbbm{J}}$ of dimension $54$.   We approximate all integrals using a Legendre-Gauss-Lobatto (LGL) quadrature rule with $100$ quadrature nodes. The Matlab command ``ro-'' line represents the indicator function $I(z)$, which is an approximation (according to Theorem \ref{Section parameter LSM theorem}) to 
$$
q_{avg}(z):= \big( \langle 1/q, E_z\rangle_{L^2(B(z,\epsilon))} \big)^{-1}.
$$ 
It is seen that $q_{avg}$ is an approximation of $q$, the Matlab command ``b*-'' line represents $q_{avg}(z)$ (plotted only in $(-r+\epsilon, r-\epsilon)$) and the Matlab command ``k-'' line represents the exact ${q}$ (plotted only in $(-r, r)$). The four types of parameters are given by setting $r=0.66$ (which is roughly $2/3$) in \eref{Section numerical example q type 1}--\eref{Section numerical example q type 4} and we set $m=4$ to have $4$ oscillations in the case of oscillatory $q$.  It is observed that the larger the index set ${\mathbbm{J}}$, the better the convergence. 

To further demonstrate its viability, we report the results in Figure \ref{Section numerical examples subsection parameter figure c20 noisy} by changing the noise level to $5\%$ for the case when $c=20$. The robustness of LSM with respect to noises is observed.

The next set of examples in Figure \ref{Section numerical examples subsection parameter figure c 3 5 7 10 noiseless noisy} is devoted to testing the performance with respect to different $c$ for both noiseless and noisy data. 
We observe that larger $\epsilon$ is expected to give better convergence for parameter identification, thereby  in order to observe a possible convergence  for noisy data, we set $\epsilon=1\times 10^{-1}$ in this set of examples. In the case of noisy data we set noisy level $\delta=5  \times 10^{-2}$. In these examples, we apply the LGL quadrature rule with $100$ quadrature nodes and we report that all the results hold similarly with $40$ quadrature nodes.  In the case of noiseless data, we observe  similar performance with respect to  different $c$; in the case of noisy data, we observe that  the performance becomes better as $c$ becomes larger.

    \begin{figure}[tbhp]
      \centering
\includegraphics[width=0.4\linewidth]{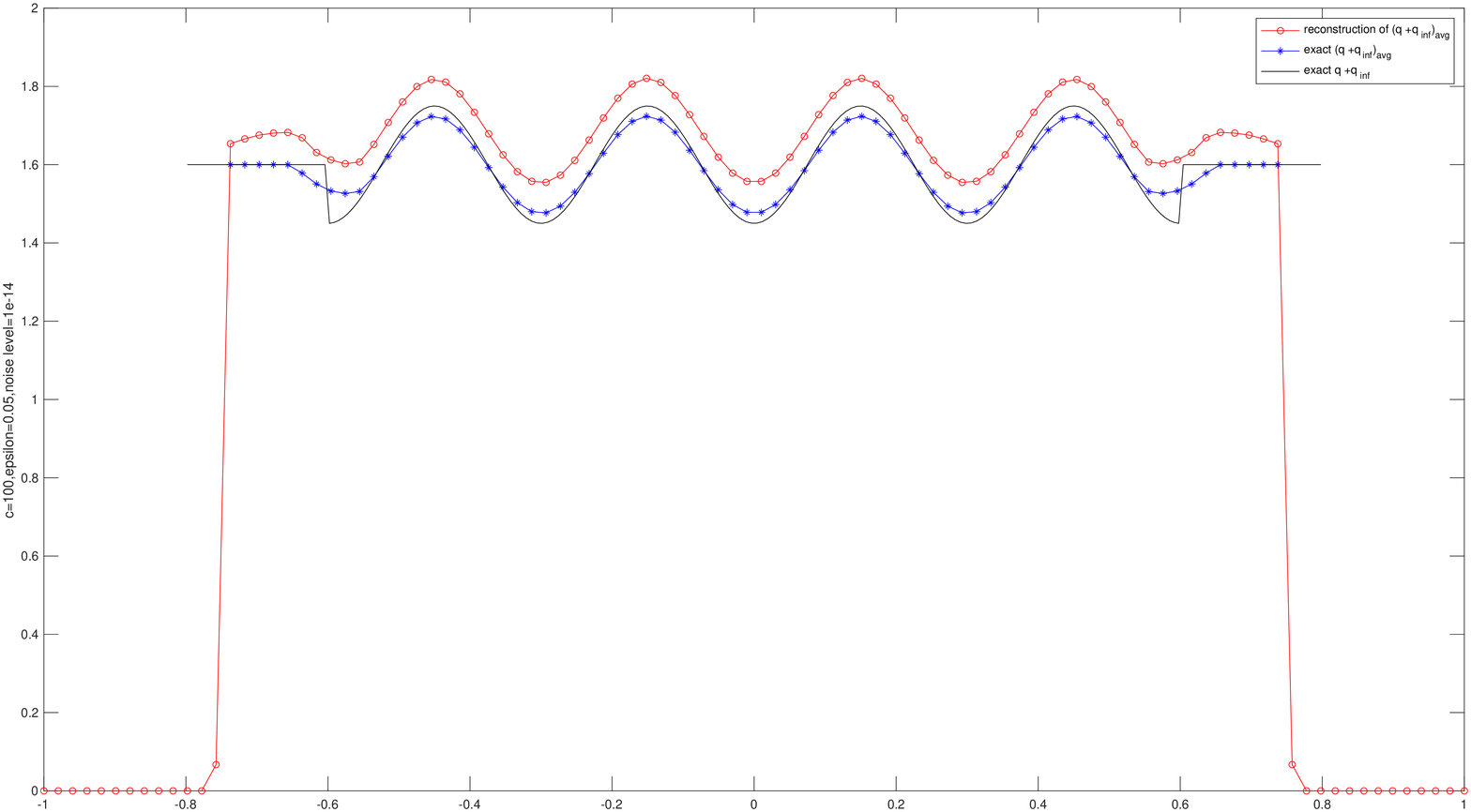}
\includegraphics[width=0.4\linewidth]{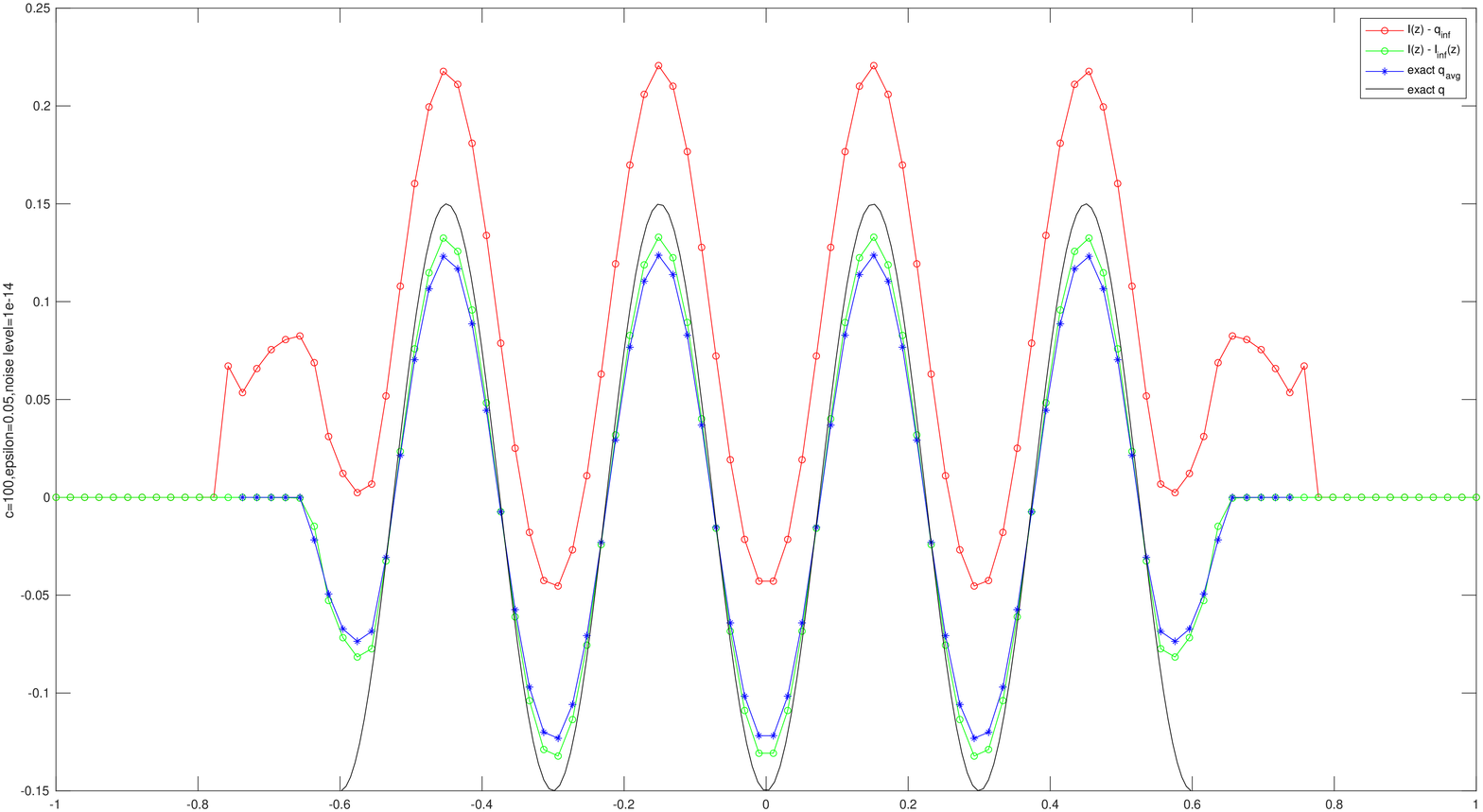}\\
     \caption{
     \linespread{1}
Plot of $I(z)$,  $q+q_{inf}$ on the left and  $I(z)-q_{inf}(z)$, $I(z)-I_{inf}(z)$, $q$ on the right and noiseless data with $\epsilon=5\times 10^{-1}$ and $c=40$.}
 \label{Section numerical examples sign section}
    \end{figure}

To illustrate the case of sign changing parameter we report the results in Figure \ref{Section numerical examples sign section}. First we show the reconstruction of $q+q_{inf} 1_D$ for $D$ of radius 0.8 where $\Omega$ is of radius 0.6 which are similar to the one from the previous examples. Then we show the comparison between $I(z)-q_{inf}(z)$ and $I(z)-I_{inf}(z)$ and as expected by remark \ref{remark num sign} the superiority of the second reconstruction similarly to our inspiration from \cite{DLSM}.

To conclude this numerical section we would like to report the numerical results for disjoint supports.   Figure \ref{Section numerical examples sign section 2 supports} shows the results for the case of a constant $q$ with 2 connected component; the distance between the two connected components is from the set $\{1.16, 0.08, 0.06, 0.04, 0.02, 0.01\}$. This case is covered by the theory but the results are very promising and theoretical analysis of the resolution limit of our method is an interesting subject for future work. 

    \begin{figure}[tbhp]
      \centering
\includegraphics[width=0.4\linewidth]{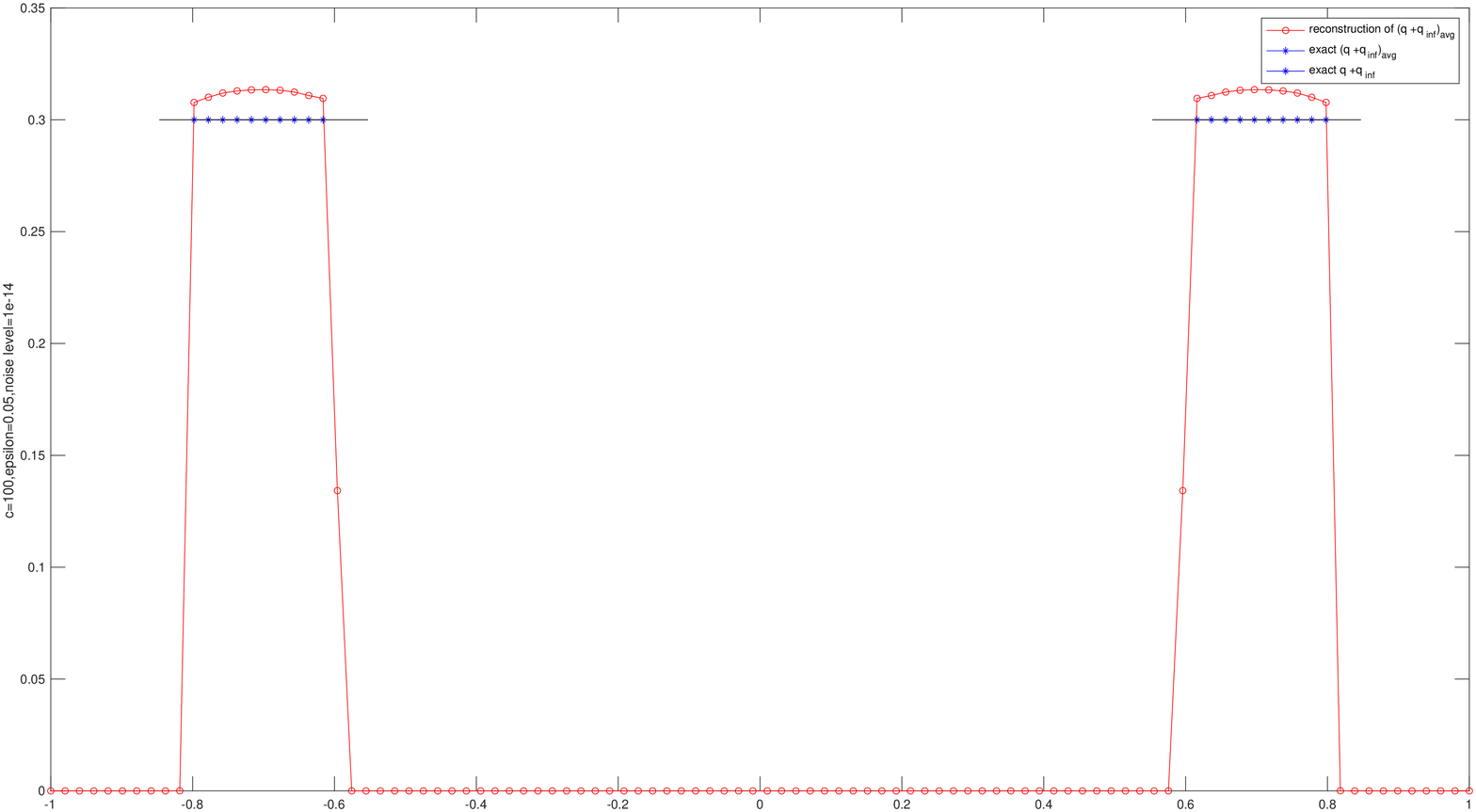}
\includegraphics[width=0.4\linewidth]{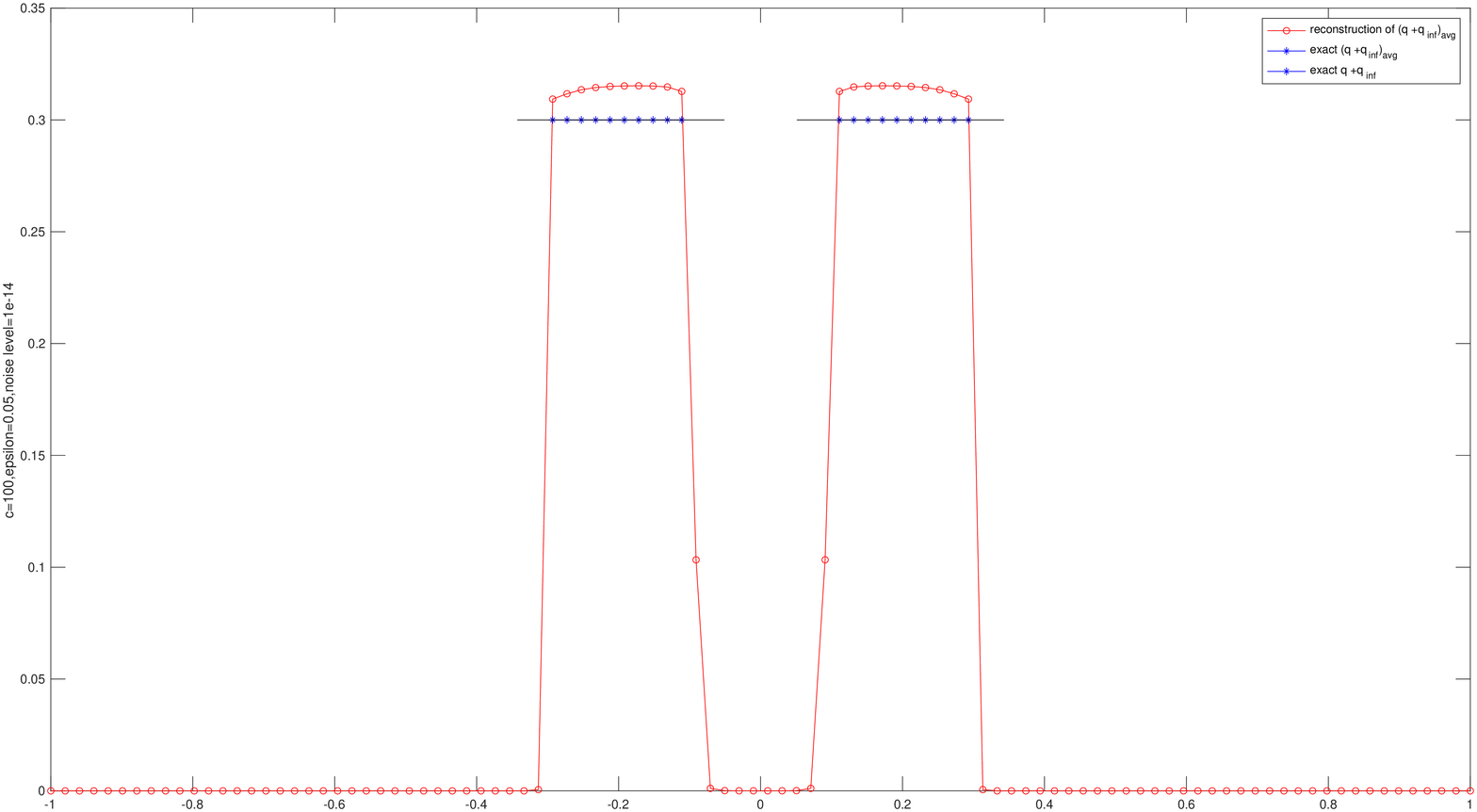}\\
\includegraphics[width=0.4\linewidth]{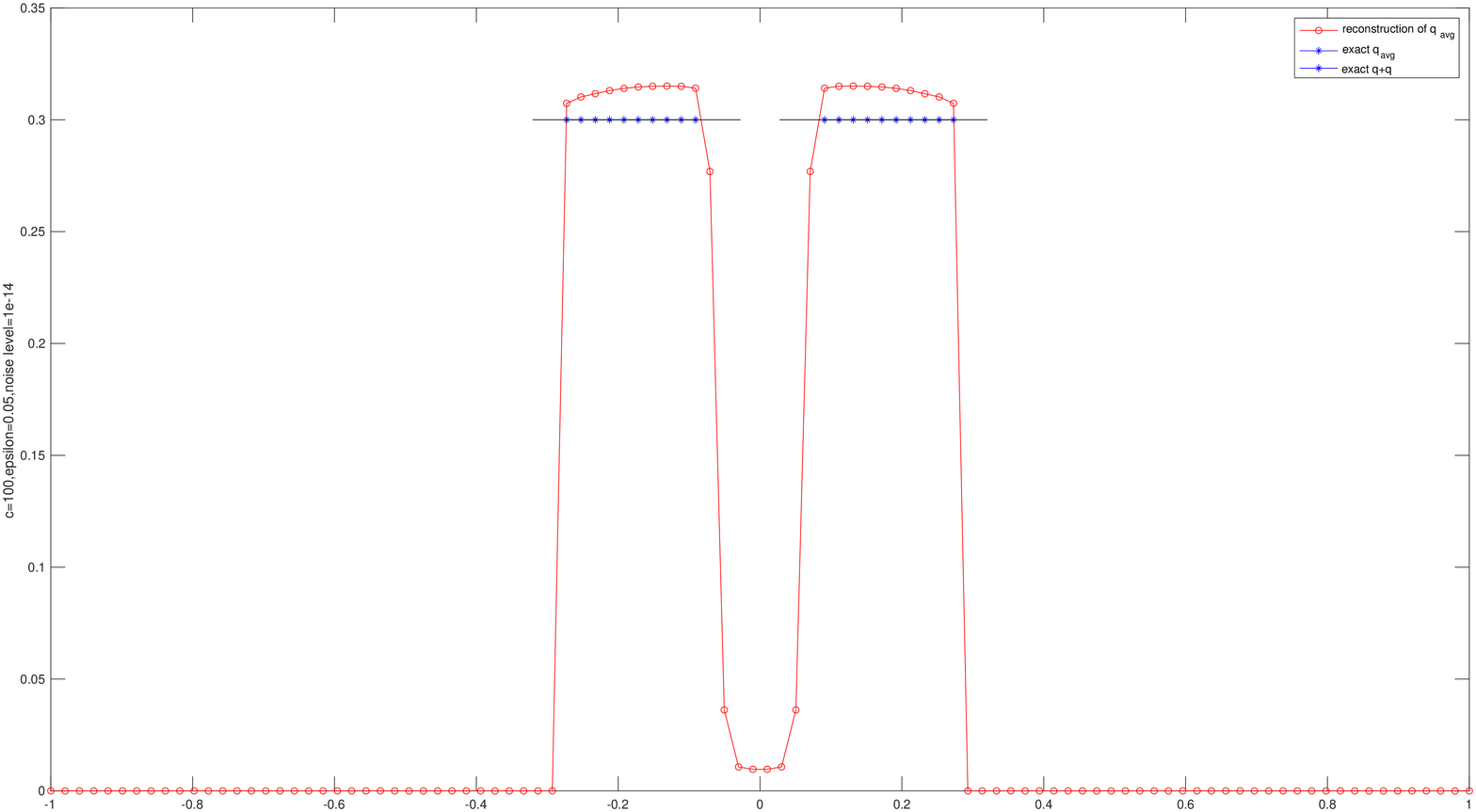}
\includegraphics[width=0.4\linewidth]{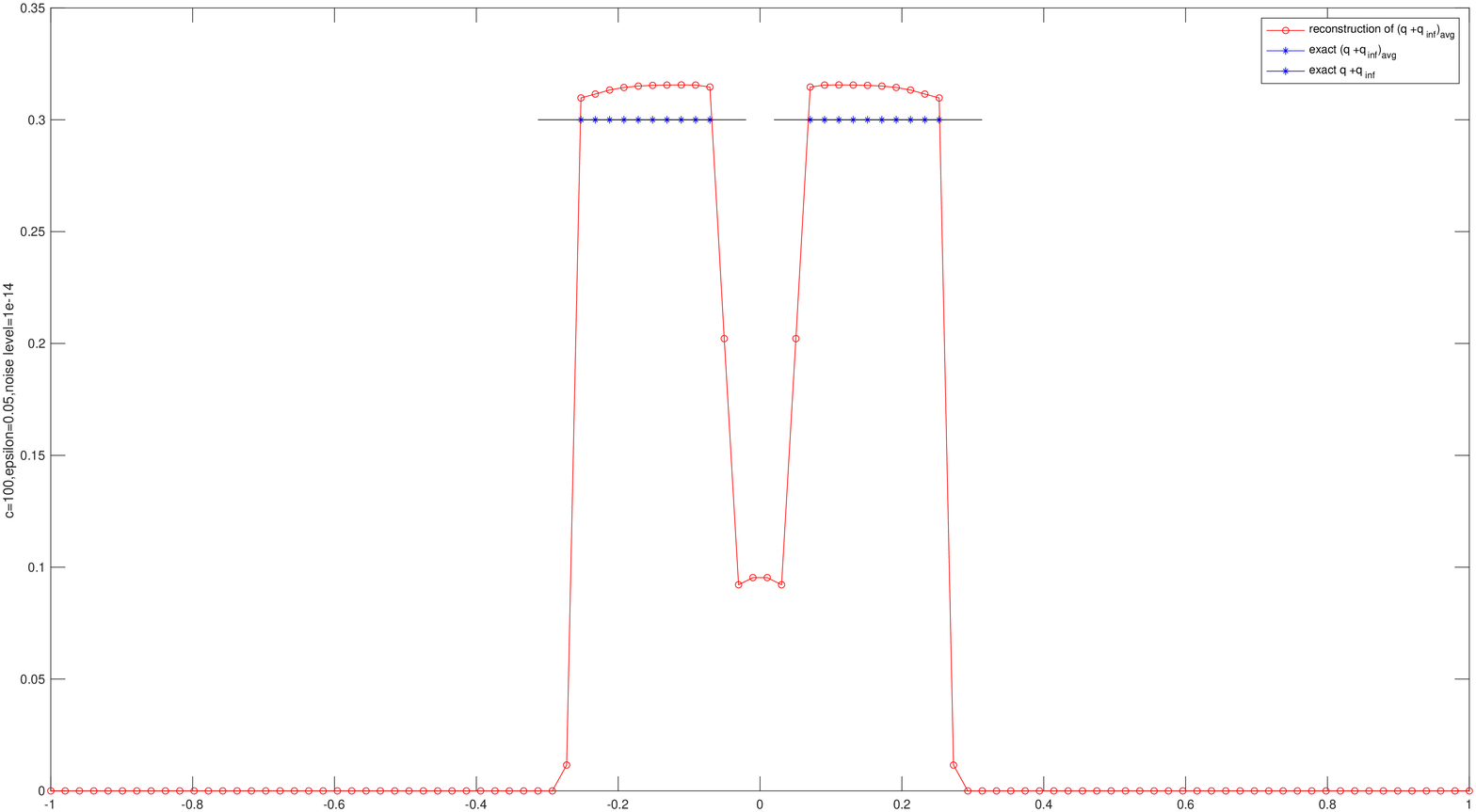}\\
\includegraphics[width=0.4\linewidth]{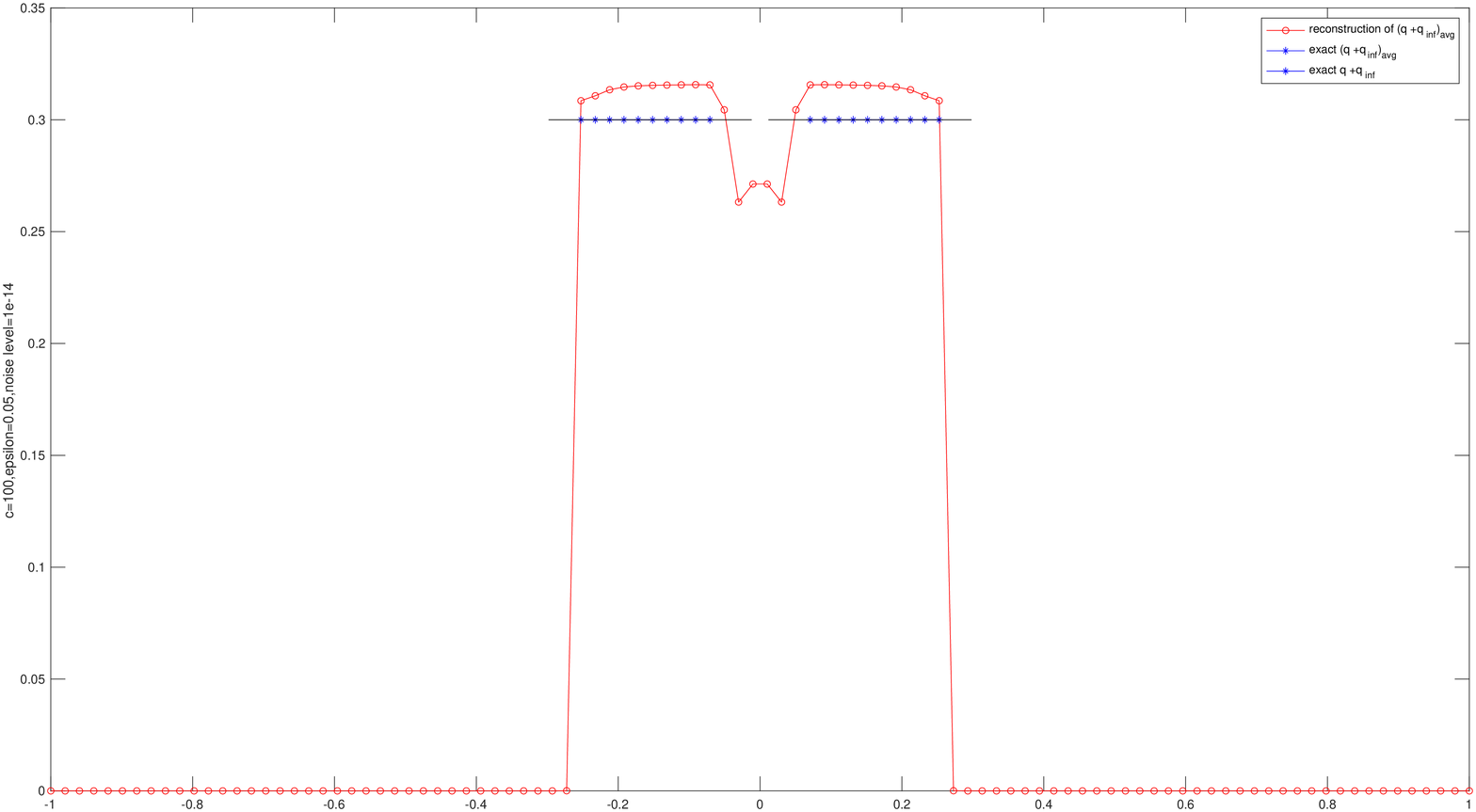}
\includegraphics[width=0.4\linewidth]{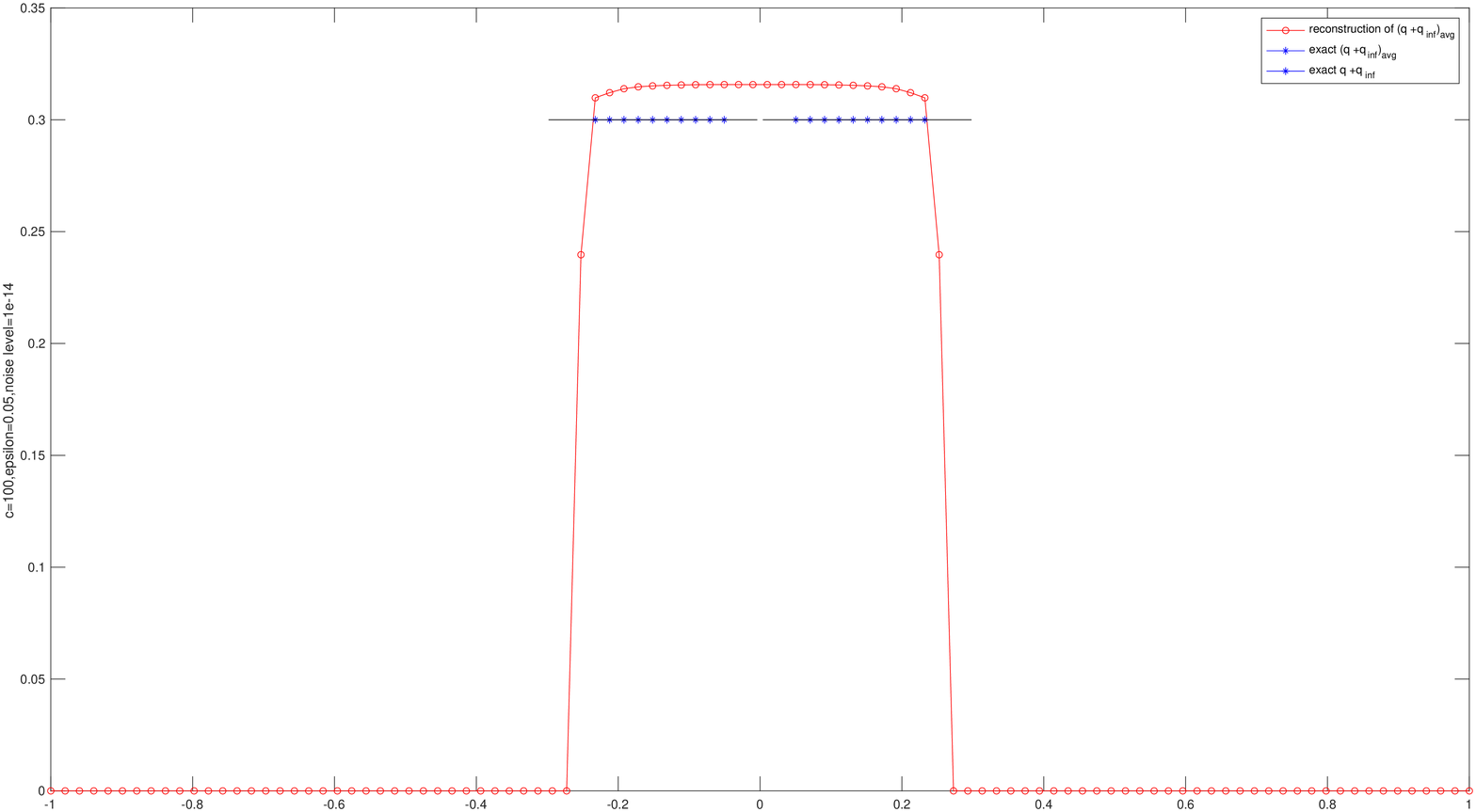}\\
     \caption{
     \linespread{1}
Plot of $I(z)$, $q_{avg}$, and $q$ for noiseless data and components getting closer and closer. $\epsilon=5\times 10^{-2}$ and $c=100$.}
 \label{Section numerical examples sign section 2 supports}
    \end{figure}

\section*{Acknowledgement}
The authors greatly thank Prof. Fioralba Cakoni, Prof. Bojan Guzina and Prof. Houssem Haddar for discussing this subject and encouraging us to work together.

\section*{References}
\bibliographystyle{plain}

\end{document}